\documentclass{amsart}%
\usepackage{amsfonts}
\usepackage{amsmath}
\usepackage{amssymb}
\usepackage{graphicx}
\newtheorem{theorem}{Theorem}[section]
\theoremstyle{plain}

\newtheorem{lemma}[theorem]{Lemma}
\newtheorem{proposition}[theorem]{Proposition}
\newtheorem{definition}[theorem]{Definition}
\newtheorem{conjecture}[theorem]{Conjecture}
\theoremstyle{remark}
\newtheorem{remark}[theorem]{Remark}
\newtheorem{example}[theorem]{Example}

\newcommand{\dbar}{\overline{\partial}}

\newcommand{\T}{{\mathbf{t}}}

\newcommand{\R}{{\mathbb R}}

\newcommand{\pr}{{\prime}}

\numberwithin{equation}{section}
\begin{document}
\title[The Novikov-Veselov Equation]{The Novikov-Veselov Equation: \\Theory and Computation}
\author{R. Croke}
\address{Department of Applied Mathematics, University of Colorado, Boulder, Colorado, 80309-0526}
\email{ryan.croke@colorado.edu}
\author{J. L. Mueller}
\address{Department of Mathematics, Colorado State University, Fort Collins, Colorado,
80523-1874, U. S. A.}
\email{mueller@math.colostate.edu}
\author{M. Music}
\address{Department of Mathematics, University of Kentucky, Lexington, Kentucky
40506-0027, U. S. A.}
\email{Michael.Music@uky.edu}
\author{P. Perry}
\address{Department of Mathematics, University of Kentucky, Lexington, Kentucky
40506-0027, U. S. A.}
\email{Peter.Perry@uky.edu}
\author{S. Siltanen}
\address{Department of Mathematics and Statistics, P.O. Box 68, FI-00014 University of Helsinki}
\email{samuli.siltanen@helsinki.fi}
\author{A. Stahel}
\address{Bern University of Applied Sciences, Engineering and Information Technology,
Mathematics, BFH-TI Biel, Postfach CH-2501 Biel, Switzerland}
\email{Andreas.Stahel@bfh.ch}
\thanks{Michael Music supported in part by NSF\ Grant DMS-1208778.}
\thanks{Peter Perry supported in part by NSF\ Grant DMS-1208778.}
\thanks{Samuli Siltanen supported in part by the Finnish Centre of Excellence in Inverse Problems Research 2012-2017 (Academy of Finland CoE-project 250215).}
\date{December 17,  2013}

\begin{abstract}
We review recent progress in theory and computation for the Novikov-Veselov (NV)
equation with potentials decaying at infinity, focusing mainly on the zero-energy case.  The inverse scattering method for the zero-energy NV equation is presented in the context of Manakov triples, treating initial data of conductivity type rigorously.  Special closed-form solutions are presented, including multisolitons, ring solitons, and breathers.  The computational inverse scattering method is used to study zero-energy exceptional points and the relationship between supercritical, critical, and subcritical potentials.

\end{abstract}
\maketitle
\vspace{-0.5in}
\tableofcontents

\section{Introduction}

The Novikov-Veselov (NV) equation is the completely integrable, nonlinear
dispersive equation%
\begin{align}
q_{t}  &  =4\operatorname{Re}\left(  4\partial^{3}q+\partial\left(  qw\right)
-E\partial w\right) \label{eq:NV}\\
\overline{\partial}w  &  =-3\partial q\nonumber
\end{align}
Here $E$ is a real parameter, the unknown function $q$ is a real-valued
function of two space variables and time, and the operators $\partial$ and
$\overline{\partial}$ are given by%
\begin{align*}
\partial &  =\frac{1}{2}\left(  \partial_{x_{1}}-i\partial_{x_{2}}\right) \\
\overline{\partial}  &  =\frac{1}{2}\left(  \partial_{x_{1}}+i\partial_{x_{2}%
}\right)  .
\end{align*}
At zero energy ($E=0$) it can also be written (after trivial rescalings) as
\begin{equation} \label{eq:NV_JM}
q_t = -\partial_z^3 q -\overline{\partial}^3q +3\partial_z(q\nu)+3\overline{\partial}_z(q\bar{\nu}), \quad \mbox{where} ~~ \overline{\partial}_z \nu=\partial_z q.
\end{equation}
The NV\ equation \eqref{eq:NV_JM} generalizes the celebrated Korteweg-de Vries (KdV) equation
\[
q_{t}=-6qq_{x}-q_{xxx}%
\]
in the sense that, if $q(x_{1},t)$ solves KdV and $\nu_{x_{1}}(x_{1}%
,t)=-3q_{x_{1}}(x_{1},t)$, then $q\left(  x_{1},t\right)  $ solves NV.

The NV\ equation was introduced by\ Novikov and Veselov in
\cite{NV:1986,VN:1984} as one of a hierarchy of completely integrable
equations that generate isospectral flows for the two-dimensional
Schr\"{o}dinger operator at fixed energy $E$. Indeed, the Novikov-Veselov
equation \eqref{eq:NV} admits the Manakov Triple Representation \cite{Manakov:1976}
\begin{equation}
\dot{L}=\left[  A,L\right]  - BL \label{eq:ManakovTriple}%
\end{equation}
where%
\begin{align*}
L  &  =-\Delta+q-E,\\
A  &  =8\left(\partial^{3}+\overline{\partial}^{3}\right)+2\left(w\partial+\overline
{w}\overline{\partial}\right),\\
B  &  =-2\left(  \partial w+\overline{\partial}\overline{w}\right)  .
\end{align*}
Here $B$ is the operator of multiplication by the function $2\left(  \partial
w+\overline{\partial}\overline{w}\right)  $. That is, a pair $\left(
q,w\right)  $ solves the NV\ equation if and only if the operator equation
\eqref{eq:ManakovTriple} holds.

The Manakov triple representation implies that the NV\ equation is, formally
at least, a completely integrable equation. Thus one expects that, for a
suitable notion of \textquotedblleft scattering data for $L$ at fixed energy
$E$,\textquotedblright\ the associated scattering transform will linearize the flow.

For nonzero energy $E$ and potentials $q$ which vanish at infinity, the
scattering transform and inverse scattering method was developed by P.
Grinevich, R. G. Novikov, and S.-P. Novikov (see Kazeykina's thesis
\cite{Kazeykina:2012t} for an excellent survey and see
\cite{Grinevich:2000,GM:1988,GN:1987,GN:1988} for the original papers). Roughly
and informally, there is a scattering transform $\mathcal{T}$ which maps the
potential $q$ to scattering data that obey a linear equation if $q$ obeys the
NV\ equation, and an inverse scattering transform $\mathcal{Q}$ which inverts
$\mathcal{T}$, so that the function%
\begin{equation}
\label{eq:NV.ISM}
q(x,\tau)=\mathcal{Q}\left[  e^{i \tau \left(  \diamond^{3}+\left(  \overline
{\diamond}\right)  ^{3}\right)  }\mathcal{T}(q_{0})\right]  (x)
\end{equation}
solves the NV\ equation with initial data $q_{0}$. The inverse scattering method may
be visualized by the following commutative diagram:
\begin{equation}\label{diagram}
\begin{array}{c}
\begin{picture}(200,92)
\thinlines
% The maps F and R_\alpha
 \put(30,75){\vector(1,0){130}}
 \put(0,72){$\mathbf{t}_0(k)$}
 \put(165,72){$\mathbf{t}_\tau(k)$}
 \put(55,80){\footnotesize{$\exp(i \tau(k^3+\overline{k}^3))(\,\cdot\,)$}}
 \put(7,15){\vector(0,1){47}}
 \put(-6,37){$\mathcal{T}$}
 \put(0,0){$q_0(z)$}
 \put(181,62){\vector(0,-1){47}}
 \put(184,37){$\mathcal{Q}$}
 \put(30,3){\vector(1,0){130}}
 \put(40,7){{\footnotesize Novikov-Veselov evolution}}
 \put(165,0){$q^{\mbox{\tiny \rm NV}}_{\tau}(z)$,}
\end{picture}
\end{array}
\end{equation}
where $\mathcal{T}$ and $\mathcal{Q}$ stand for the direct and inverse nonlinear Fourier transform, respectively, and the function $\mathbf{t}_\tau:\mathbb{C}\rightarrow \mathbb{C}$ is called the {\em scattering transform}.
In the case $E=0$, the
inverse scattering method was studied by Boiti et.~al.~\cite{BLMP:1987}, Tsai
\cite{Tsai:1993}, Nachman \cite{Nachman:1996}, Lassas-Mueller-Siltanen
\cite{LMS:2007}, Lassas-Mueller-Siltanen-Stahel \cite{LMSS:2011,LMSS:2012},
Music \cite{Music:2013}, Music-Perry \cite{MP:2013}, and Perry
\cite{Perry:2013}. 

Recently, Angelopoulos \cite{Angelopoulos:2013} used Bourgain's Fourier restriction method (see \cite{Bourgain:1993a,Bourgain:1993b} and the lectures of Tao \cite{Tao:2006})  together with subtle bilinear estimates on the nonlinearity to prove that the Novikov-Veselov equation at $E=0$ is locally well-posed in the Sobolev spaces $H^s(\mathbb{R}^2)$ for $s>1/2$. Here $H^s(\mathbb{R}^2)$ is the space of square integrable functions $u$ on $\mathbb{R}^2$ whose Fourier transforms $\hat{u}$ obey
$$ \|u\|_{H^s}^2 := \int (1+|\xi|)^{2s} \left\vert \hat{u}(\xi) \right\vert^2 \, d\xi < \infty.$$

Angelopoulos' results place the local well-posedness theory for this equation on a sound footing. The inverse scattering method is expected to elucidate--for a more restrictive but nonetheless rich class of initial data--the \emph{global} behavior of the solutions, including global existence in time, blow-up, and asymptotics. 

The case $E=0$ is intimately
connected with the following trichotomy of behaviors for the two-dimensional
Schr\"{o}dinger operator $L$ at zero energy.

\begin{definition}
\label{def:tri}The operator $L=-\Delta+q$ is said to be:\newline(i)
\emph{subcritical} if the operator $L$ has a positive Green's function and the equation $L\psi=0$ has
a strictly positive distributional solution, \newline(ii) \emph{critical} if $L\psi=0$ has a
bounded strictly positive solution but no positive Green's function, and\newline(iii) \emph{supercritical} otherwise.
\end{definition}

This distinction first arose in the study of Schr\"{o}dinger semigroups, i.e.,
the operators $e^{-tL}$ where $L=-\Delta+q$. Simon
\cite{Simon:1980,Simon:1981} (see also Simon's comprehensive review \cite{Simon:1982} on Schr\"{o}dinger semigroups) studied $L^{p}%
$-mapping properties of $e^{-tL}$ and asymptotics of $\left\Vert
e^{-tL}\right\Vert _{p,p}$, where $\left\Vert ~\cdot~\right\Vert _{p,p}$
denotes the operator norm as maps from $L^{p}$ to itself. Simon shows that
\[
\alpha_{p}(q)=\lim_{t\rightarrow\infty}t^{-1}\ln\left\Vert e^{-tL}\right\Vert
_{p,p}%
\]
is independent of $p\in\left[  1,\infty\right]  $. In the language of Schr\"{o}dinger semigroups,
a potential $q$ is:

(i) subcritical if $\alpha_{\infty}\left(  \left(  1+\varepsilon\right)
q\right)  =0$ for some $\varepsilon>0$,

(ii) critical if $\alpha_{\infty}\left(  q\right)  =0$ but $\alpha_{\infty
}\left(  \left(  1+\varepsilon\right)  q\right)  >0$ for all $\varepsilon>0$, and

(iii) supercritical if $\alpha_{\infty}(q)>0$. \newline 
Clearly, a sufficient condition for $q$ to be supercritical is that
$L$ have a negative eigenvalue.

In \cite{Murata:1984}, Murata showed that, for two-dimensional Schr\"{o}dinger
operators with potentials $q$ with $q(x)$ uniformly H\"{o}lder
continuous and $q(x)=\mathcal{O}\left(|x|^{-4-\epsilon} \right)$ for
some $\epsilon>0$, the trichotomy of behaviors for Schr\"{o}dinger
semigroups is equivalent to Definition \ref{def:tri}. Murata further
studied the existence and properties of positive solutions of the
Schr\"{o}dinger equation in \cite{Murata:1984}, and showed that for
his class of potentials, the trichotomy could be characterized as
follows: a potential is 

(i) subcritical if and only if $L\psi=0$ has a strictly positive solution of the form
$c \log(|x|)+d +\mathcal{O}\left(|x|^{-1}\right)$ with $c \neq 0$,

(ii) critical if $L\psi=0$ if and only if $L\psi=0$ has a strictly positive bounded solution, and

(iii) supercritical if $L\psi=0$ has no strictly positive solutions.

Extending Murata's result,
Gesztesy and Zhao \cite{GZ:1995} used Brownian motion techniques to prove the following optimal result for critical
potentials. Suppose that  $q$ is a real-valued measurable function with 
$$ \lim_{\alpha \downarrow 0} \left\{ \sup_{x \in \mathbb{R}^2} \int_{|x-y|\leq \alpha}
	\ln \left(|x-y|)^{-1} \right)  \left| q(y) \right| \, dy \right\}= 0 $$
and 
$$ \int_{|y| \geq 1} \ln(|y| ) |q(y) | \, dy < \infty.$$
Then $q$ is critical if and only if there exists a positive, bounded
distributional solution $\psi$ of $H\psi=0$. These two conditions mean
essentially that 
$$
q(x)=\mathcal{O}\left(|x|^{-2} \left(\ln(|x|)\right)^{-2-\epsilon}\right)$$
for some $\epsilon>0$. We refer the reader to \cite{GZ:1995} for
further references and history.

As we will see, corresponding to the trichotomy in Definition \ref{def:tri},
the scattering transform of $q$ is either mildly singular, nonsingular, or
highly singular. This is illustrated dramatically in the examples studied by
Music, Perry, and Siltanen \cite{MPS:2013}, described in Section
\ref{sec:zero_excep} below. One would expect the singularities of the
scattering transform to be mirrored in the behavior of solutions to
the NV\ equation. We will discuss the following conjecture, and some
partial results toward its resolution, in the last section of this
paper:

\begin{conjecture}
The Novikov-Veselov equation (\ref{eq:NV}) has a global solution for critical
and subcritical initial data, but its solution may blow up in finite time for
supercritical initial data.
\end{conjecture}

To elucidate this conjecture, it is helpful to recall how the scattering
transform for Schr\"{o}dinger's equation is connected with Calder\'on's inverse
conductivity problem \cite{Calderon:1980} (see \cite{AP:2006} for the solution to Calderon's problem for $L^\infty$ conductivites and for references to the extensive literature on this problem). Critical potentials are
also known in the literature as \textquotedblleft potentials of
conductivity type\textquotedblright\ because of their connection with
the Calder\'on inverse conductivity problem: suppose one wishes to
determine the conductivity $\sigma$ of a bounded plane region 
$\Omega$ by boundary measurements. The potential $u$ of $\Omega$ with voltage
$f$ on the boundary is determined by the equation%
\begin{align*}
\nabla\cdot\left(  \sigma\nabla u\right)   &  =0\\
\left.  u\right\vert _{\partial\Omega}  &  =f
\end{align*}
Calder\'on's problem is to reconstruct $\sigma$ from knowledge of the \emph{Dirichlet-to-Neumann} map, defined as follows. If $\Omega$ has smooth boundary then the above boundary value problem has a unique solution $u$ for given $f \in H^{1/2}(\partial \Omega)$, so that $$\Lambda_\sigma f = \sigma \left. \frac{\partial u}{\partial \nu} \right|_{\partial \Omega}$$
is uniquely determined. The map $\Lambda_\sigma: H^{1/2}(\partial \Omega) \rightarrow H^{-1/2}(\partial \Omega)$ is the \emph{Dirichlet-to-Neumann map}. 

This boundary value problem is equivalent, under the change of variables
$u=\sigma^{-1/2}\psi$, to the Schr\"{o}dinger problem%
\[
\Delta \psi-q \psi=0
\]
where $q=\sigma^{-1/2}\Delta\left(  \sigma^{1/2}\right)  $. A potential of
this form for strictly positive $\sigma\in L^{\infty}$ (and some additional
regularity) is called a \emph{potential of conductivity type}. More precisely, the class of potentials originally studied by Nachman \cite{Nachman:1996} is as follows. We denote by $L^p_\rho(\mathbb{R}^2)$ the space of measurable functions with norm $\| f \|_{L^p_\rho} =\left \| \langle x \rangle^\rho f\right \|_p$.

\begin{definition}
\label{def:cond.Nachman}
Let $p\in( 1,2)$ and  $\rho>1$. A real-valued measurable function $q \in L^p_\rho(\mathbb{R}^2)$ is a \emph{potential of conductivity type}  if there is a function $\sigma \in L^\infty(\mathbb{R}^2)$ with $\sigma(x)\geq c_0>0$ so that $q= \sigma^{-1/2}\Delta ( \sigma^{1/2}) $ in the sense of distribution derivatives.
\end{definition}
A real-valued potential in $L^p_\rho(\mathbb{R}^2)$ is of conductivity type if and only if it is critical: the bounded,
positive solution to $\Delta\psi-q\psi=0$ is exactly $\psi=\sigma^{1/2}$. 

As shown by Murata \cite{Murata:1984,Murata:1986}, critical potentials are
very unstable: if $w\in\mathcal{C}_{0}^{\infty}\left(  \mathbb{C}\right)  $ is
a nonnegative bump function and $q_{0}$ is a critical potential, the potential
$q_{\lambda}=q_{0}+\lambda w$ is supercritical for any $\lambda>0$. This means
that the set of critical potentials is nowhere dense in any reasonable function
space! Music, Siltanen, and Perry \cite{MPS:2013} studied the behavior of the
scattering transform for families of this type when $q_{0}$ and $w$ are both
smooth, compactly supported, and radial. The corresponding scattering
transforms are mildly singular for subcritical potentials, regular for
critical potentials, and have circles of singularities for supercritical potentials.

The NV\ equation may be solved by inverse scattering for subcritical and
critical potentials, but it is not yet clear how to construct a solution by
inverse scattering for supercritical potentials.

In this article, we will focus primarily on the Novikov-Veselov equation at zero energy. For the Novikov-Veselov equation at nonzero energy, we refer the reader to the thesis of Kazeykina \cite{Kazeykina:2012t} and to the papers \cite{Kazeykina:2011,Kazeykina:2012,KN:2011a,KN:2011b,KN:2011d} for other recent work on qualitative properties of solutions and for further references to the literature.
We will report on recent progress on both the theoretical and the numerical
analysis of this equation, and pose a number of open problems. In section \ref{sec:background}
we review the history of the inverse scattering method, the dispersion relation, symmetries and scaling properties, and conservation laws for the NV equation. In section \ref{sec:inverse},
we give an exposition of the inverse scattering method for the NV equation at zero energy
from the point of view of the Manakov triple
representation, treating with full mathematical rigor the case of ``smooth potentials of 
conductivity type'' (see Definition \ref{def:cond}).  We discuss the numerical implementation of the maps $\mathcal{T}$ and $\mathcal{Q}$ in sections \ref{sec:scatcomp} and \ref{sec:InvScat}, respectively. In section \ref{sec:Special}, we discuss special closed-form solutions of the NV equation including ring solitons and breathers.  In section \ref{sec:zero_excep}, the computational inverse scattering method is used to study zero-energy exceptional points and the relationship between supercritical, critical, and subcritical potentials.  Finally, in section \ref{sec:open}, we
discuss open problems. In an appendix, we collect some useful tools for the mathematical analysis of the direct and inverse scattering maps.

\emph{Notation}. In what follows, we use the variable $t$ to denote time except when discussing the solution of NV via the inverse scattering method. In this case, $\tau$ denotes time in order to distinguish $t$ from $\mathbf{t}$, the scattering transform. We denote the spatial variable by $z=x+iy$. Functions $q(z)$, $q(z,t)$, etc., are  functions of $x$ and $y$ but are generally \emph{not} analytic in $z$. 

\section{Background for the Zero-Energy NV\ Equation}
\label{sec:background}

First, we summarize the historical progress that led to the completion of the diagram (\ref{diagram}) for the NV equation at zero energy.
In 1987, Boiti, Leon, Manna and Pempinelli  \cite{BLMP:1987} studied the evolution under the assumption that $q_0$ is such that the solution $q^{\mbox{\tiny \rm NV}}_{\tau}$ to (\ref{eq:NV}) exists and does not have exceptional points and established that  the scattering data evolves as ${\mathcal{T}}(q^{\mbox{\tiny \rm NV}}_{\tau})=e^{i\tau(k^3+\overline{k}^3)}{\mathcal{T}}(q_0)$. 
In 1994, Tsai \cite{Tsai:1993} considered a certain class of small and rapidly decaying initial data (which excludes conductivity-type potentials) and assumed that $q_0$ has no exceptional points and that $q_{\tau}$ is well-defined. Under such assumptions, he then showed that  $q_{\tau}$ is a solution of the Novikov-Veselov equation (\ref{eq:NV}). 
In 1996, Nachman \cite{Nachman:1996} established that initial data of conductivity type does not have exceptional points and the scattering data ${\mathcal{T}}(q_0)$ is well-defined.
Nachman's work paved the way for rigorous results: all studies about diagram (\ref{diagram}) published before \cite{Nachman:1996} were formal as they had to assume the absence of exceptional points without specifying acceptable initial data.
In 2007, \cite{LMS:2007} established for smooth, compactly supported conductivity-type initial data with $\sigma\equiv 1$ outside supp$(q_0)$ that there is a well-defined continuous function $q_{\tau}:{\mathbb R}^2\rightarrow \mathbb{C}$ from the inverse scattering method satisfying the estimate $|q_\tau(z)|\leq C(1+|z|)^{-2}$ for all $\tau> 0$.  In \cite{LMSS:2011} it was shown that an initially radially-symmetric
conductivity-type potential evolved under the ISM does not have
exceptional points and is itself of conductivity-type.  Note that in \cite{MPS:2013} the set of conductivity type potentials is shown to be unstable under $C^{\infty}_0$ perturbations.  In
\cite{LMSS:2012}  evolutions computed from a numerical
implementation of the inverse scattering method of rotationally symmetric, compactly
supported conductivity type initial data are compared to
evolutions of the NV equation computed from a semi-implicit finite-difference
discretization of NV and are found to agree with high precision.  This
supported the integrability conjecture that was then established in
\cite{Perry:2013} for a larger class of initial data using the inverse scattering map for the Davey-Stewartson equation and Bogdanov's Miura transform.

In Section \ref{sec:Special} of this paper, we present several closed-form solutions of the NV equation.  We briefly review earlier constructions of solutions for the NV equation without presenting an exhaustive list. Grinevich, Manakov and R.~G.~Novikov constructed solition solutions using
nonlocal Riemann problem techniques in \cite{Grinevich:1986,Grinevich:2000,GM:1988,GN:1985,GN:1986} for nonzero energy and with small initial data. Also,
solitons are constructed by Grinevich using rational potentials in \cite{Grinevich:1986} (see also the survey \cite{Grinevich:2000}, by Tagami using the Hirota method in \cite{tagami:1989}, by Athorne and Nimmo using Moutard
transformation in \cite{ANimmo:1991}, by Hu and Willox using a nonlinear superposition formula in \cite{HW:1996}, by Xia, Li and Zhang using hyperbola function method and
Wu-elimination method in \cite{Xia:2001}, by Ruan and Chen using separation of variables in \cite{RuanChen2001,RuanChen,RuanChen2003,ZFC:2005}, and by J.-L.~Zhang, Wang, Wang and Fang using the
homogeoneous balance principle and B\"acklund transformation in \cite{HW:1996}. Lump solutions are constructed by Dubrovsky and Formusatik using the $\overline{\partial}$-dressing method in
\cite{dubrovsky}. Dromion solutions are constructed by Ohta and \"Unal using Pfaffians in \cite{ohta,unal}. The Darboux transformation is used by Hu, Lou, Liu, Rogers, Konopelchenko,
Stallybrass and Schief to construct solutions in \cite{HuLouLiu,rogers}. Taimanov and Tsar\"{e}v \cite{TT:2007,TT:2008a,TT:2008b,TT:2010a,TT:2010b} use the Moutard transformation to construct examples of Schr\"{o}dinger operators $L$ with $L^2$ eigenvalues at zero energy, and solutions of NV which blow up in finite time.  In \cite{zhang,Zheng:2005} Zheng, J.-F.~Zhang and Sheng explore chaotic and fractal properties of solutions to the NV equation.

\subsection{Dispersion, group velocity and phase velocity}

Solitons form when there is a balance between nonlinearity and dispersion.  The dispersion relation is the relation that gives the frequency as a function of the wave vector $(k_1,k_2)$.  To find the dispersion relation for the NV equation, consider the linear part of the equation
\begin{equation}\label{eq:linearNV} 
q_t = -\frac{1}{4}q_{xxx} + \frac{3}{4}q_{xyy} 
\end{equation} 
The plane wave functions $q(x,y,t) = e^{k_1x + k_2y - \omega t}$ are solutions to \eqref{eq:linearNV} provided
\begin{equation}\label{NVdispersion} \omega = -\frac{1}{4}k_1^3 + \frac{3}{4}k_1k_2^2.
\end{equation}
Equation \eqref{NVdispersion} defines the dispersion relation for the
NV equation. 

% \begin{figure}[!ht]
% \centering
% \begin{tabular}{cc}
% \epsfig{file=images/DispersionRelationNV.pdf,width=0.5\linewidth,height=0.3\textheight,clip=} &
% \epsfig{file=images/DispersionRelationNVHeatMap.pdf,width=0.5\linewidth,height=0.3\textheight,clip=} 
% \end{tabular}
% \caption[Two views of the dispersion relation $\omega(k)$ for the NV equation]{Two views of the dispersion relation $\omega(k)$ for the NV equation}
% \label{fig:DispersionNV}
% \end{figure}

\begin{figure}[!ht]
\centering
\begin{tabular}{cc}
\includegraphics[width=0.5\linewidth,height=0.3\textheight]{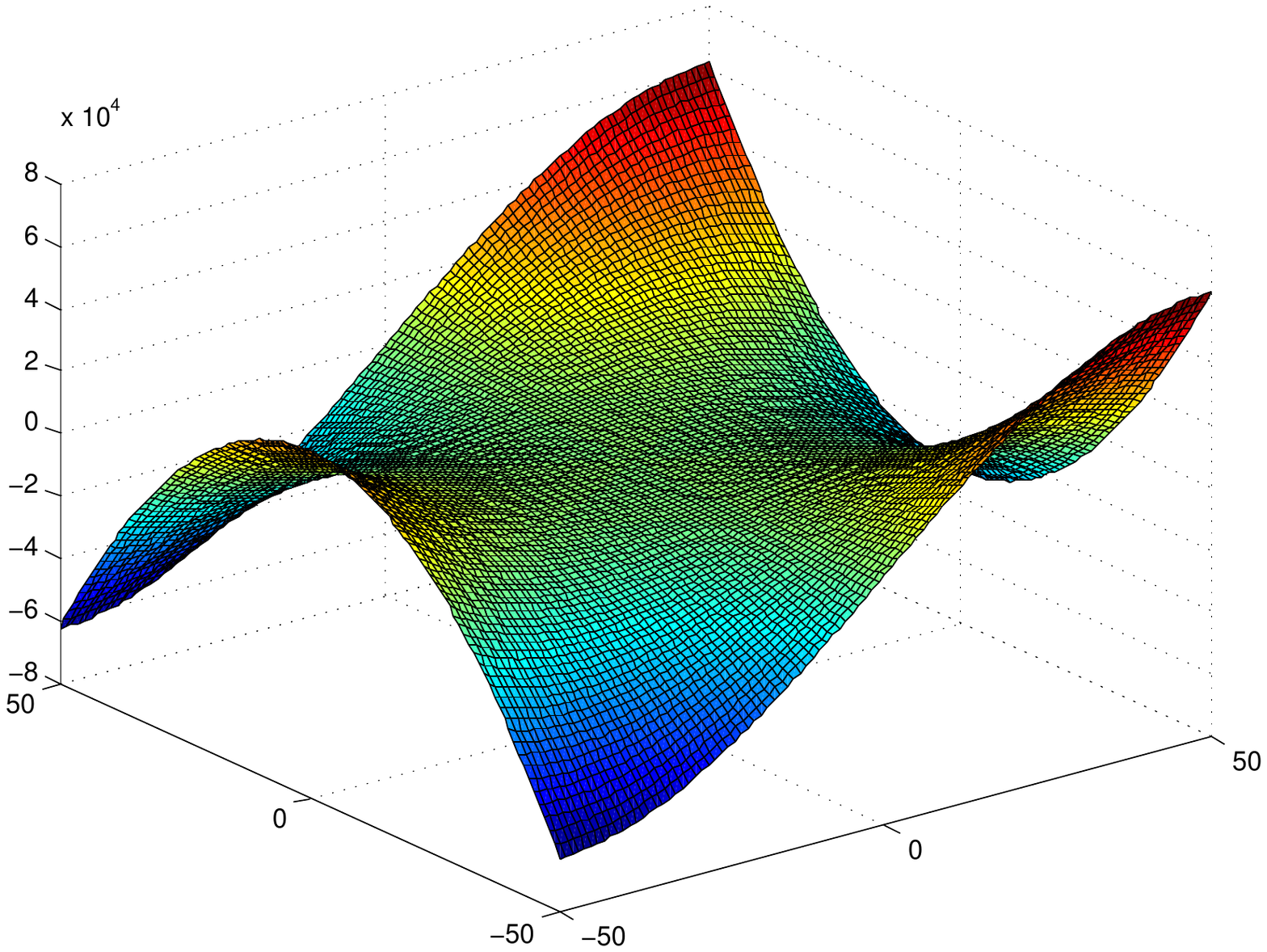} &
%DispersionRelationNV.pdf} &
\includegraphics[width=0.5\linewidth,height=0.3\textheight]{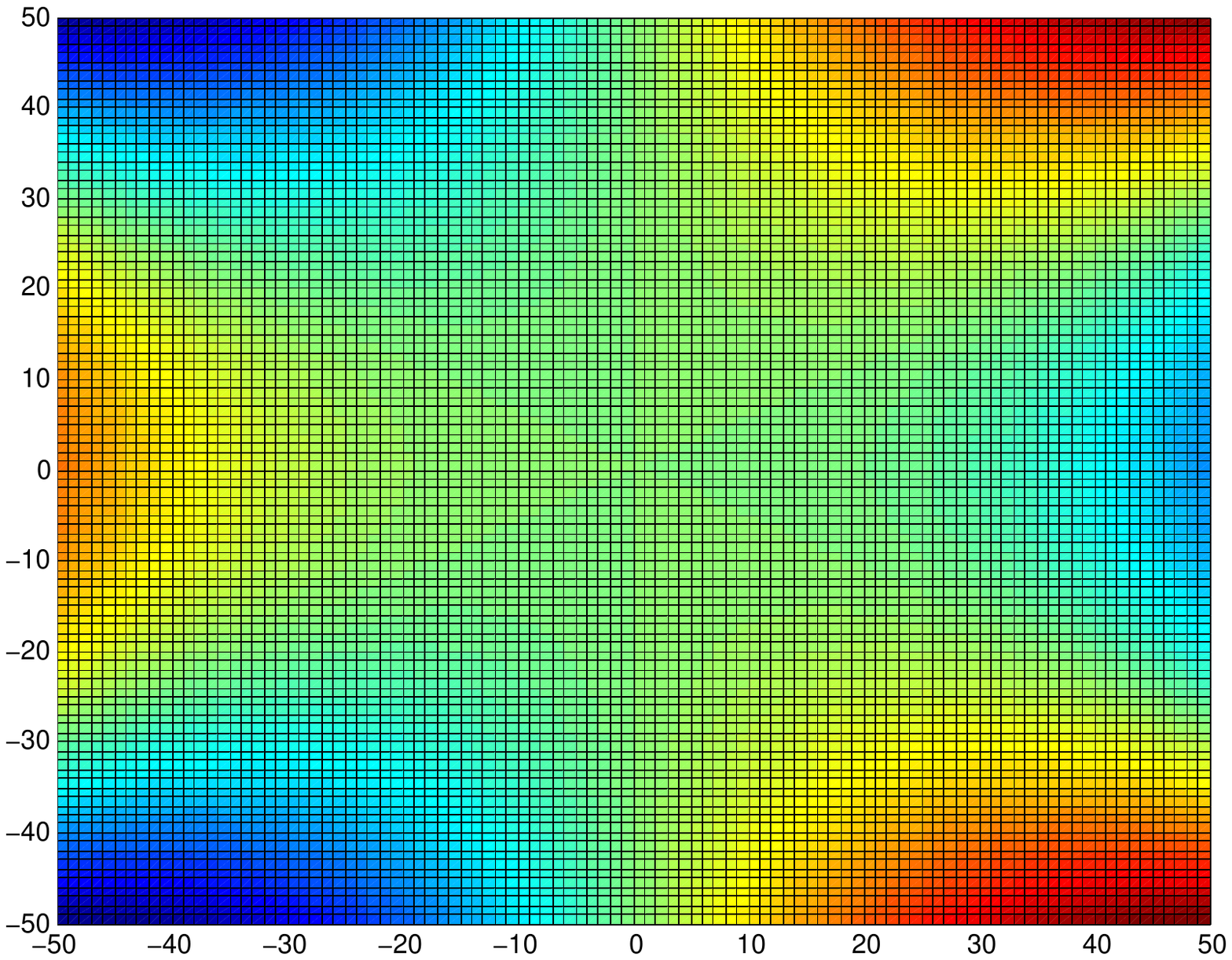}
%DispersionRelationNVHeatMap.pdf}
\end{tabular}
%% get rid of whitespace
\vskip -0.70in
\caption[Surface plot (left) and heat map (right) of the dispersion relation $\omega(k)=k_1^2/4+3k_1 k_2^2/4$ for the NV equation]{Surface plot (left) and heat map (right) of the NV dispersion relation $\omega(k)=k_1^2/4+3k_1 k_2^2/4$}
\label{fig:DispersionNV}
\end{figure}

The phase velocity, $\textbf{c}_p$, which gives the velocity of the wavefronts, is defined by $\textbf{c}_p = \frac{\omega(\textbf{k})}{|\textbf{k}|^2}(k_1,k_2)$ and for the NV equation is \begin{equation}\label{phsvelLinVN} \textbf{c}_p = \frac{k_1^3 - 3k_1k_2^2}{4(k_1^2 + k_2^2)^{3/2}}(k_1,k_2).\end{equation}
 The group velocity, which gives the velocity of the wave packet, is \[ \textbf{c}_g  \equiv \nabla\omega = \frac{3}{4}\left(-k_1^2 + k_2^2,2k_1k_2\right).\]
The group velocity is bounded below by 0, and the sign of the phase
velocity depends on sgn($k_1^3 - 3k_1k_2^2)$.

\subsection{Symmetries and Scaling}
To understand how scaling of the dependent and independent variables
change the NV equation, let us first consider the auxilliary
$\overline{\partial}$ equation in equation \eqref{eq:NV_JM} in the form 
\[\overline{\partial}\nu = \partial q, \quad \nu = v + i w. \]
Note that under the transformation $r(x,y,t) \equiv \gamma\nu(\alpha t,\beta x, \beta y)$ and  $s(x,y,t) \equiv \gamma q(\alpha t,\beta x, \beta y),$   
the $\overline{\partial}$ equation  remains unchanged, i.e
$\overline{\partial}\nu = \partial q$ if and only if
$\overline{\partial} r = \partial s$. 
Then $r_x(x,y,t)  = \beta\gamma \nu_x(\alpha t,\beta x, \beta y)$ and $r_y(x,y,t)  = \beta\gamma \nu_y(\alpha t,\beta x, \beta y)$.

Now, we examine the main equation as presented in equation \eqref{eq:NV_JM}. Note that
$s_t(x,y,t)  = \alpha\gamma q_t(\alpha t,\beta x, \beta y), $
$s_{xxx}(x,y,t)  = \beta^3\gamma q_{xxx}(\alpha t,\beta x, \beta y), $
$s_{xyy}(x,y,t)  = \beta^3\gamma q_{xyy}(\alpha t,\beta x, \beta y),$ and
$(qv)_x + (qw)_y = \gamma^2\beta((s\Re(r))_x + (s\Im(r)_y)$.
Assuming $q$ is a solution to the NV equation \eqref{eq:NV}, we find 
\begin{align*} 
4q_t & = -u_{xxx} + 3q_{xyy} + 3(qv)_x + 3(qw)_y  \\ \Longrightarrow 
\frac{4}{\alpha\gamma}s_t & = -\frac{1}{\beta^3\gamma}s_{xxx} + \frac{3}{\beta^3\gamma}s_{xyy} + \frac{3}{\beta\gamma^2}(s\Re(r))_x + \frac{3}{\beta\gamma^2}(s\Im(r))_y.  
\end{align*} Multiplying by $\alpha\gamma$ leads to\begin{equation} 4s_t  = -\frac{\alpha}{\beta^3}s_{xxx} + \frac{3\alpha}{\beta^3}s_{xyy} + \frac{3\alpha}{\beta\gamma}((sr_1)_x + (sr_2)_y).\label{eq:ReScaleNV} \end{equation} 
 
The table below shows the possible sign conventions possible for each term of the right hand side of equation \eqref{eq:ReScaleNV}.  

\begin{center}
    \begin{tabular}{ | l | l | l | l |}
    \hline
    $\alpha$ & $\beta$ & $\gamma$ & Signs in \eqref{eq:ReScaleNV} \\ \hline
    + & + & + & -\hfill +\hfill + \\ \hline
    - & - & + & -\hfill +\hfill + \\ \hline
    - & + & + & +\hfill -\hfill - \\ \hline
    + & - & + & +\hfill -\hfill - \\ \hline
    + & + & - & -\hfill +\hfill - \\ \hline
    - & - & - & -\hfill +\hfill - \\ \hline
    - & + & - & +\hfill -\hfill + \\ \hline
    + & - & - & +\hfill -\hfill + \\ \hline  
    \end{tabular}
    %\caption{Computations for the signs in the right hand side of \ref{eq:ReScaleNV}}
\end{center}
Thus, there is a fixed ratio of -3 of the coefficients of the linear spatial terms, and
any other coefficient is possible by proper rescaling of independent and dependent variables.

We also consider under what rotations the Novikov-Vesolov is invariant.  Writing the NV equation as
\begin{align}
\label{eq:NVComplex2}  q_t & = -\partial^{3}q - \overline{\partial^{3}}q +
3\partial(q\nu) + 3\overline{\partial}(q\overline{\nu}),\\
 \label{eq:Dbar2} \overline{\partial}\nu & = \partial q,
\end{align}
it is easy to see by conjugating \eqref{eq:NVComplex2}, that if $q$ is real at time $t_0$, then $q$ stays real.
%\begin{align*} \bar{u}_t & = \overline{-\partial^{3}u - \overline{\partial^{3}}u +3\partial(u\nu) + 3\overline{\partial}(u\overline{\nu})},\\
%& = -\overline{\partial}^{3}u - \partial^{3}u + 3\overline{\partial}(u\overline{\nu}) + 3\partial(u\nu),\\
 %& = u_t,
 %\end{align*}
%  Thus, \[ \frac{d}{dt}\mbox{Im}u = \frac{1}{2i}(u_t - \bar{u}_t) = 0,\] and so $u$ remains real.  

To consider rotations let 
$x = x'\cos(\theta) - y'\sin(\theta)$ and $y = x'\sin(\theta) + y'\cos(\theta).$
%Then 
%$$\frac{\partial}{\partial x}  = \frac{\partial}{\partial x'}\cos(\theta) - \frac{\partial}{\partial y'}\sin(\theta)$$ 
%and 
%$$\frac{\partial}{\partial y}  = \frac{\partial}{\partial x'}\sin(\theta) + %\frac{\partial}{\partial y'}\cos(\theta), $$
so that
%\begin{eqnarray*}
%\partial_z  = \frac{\partial}{\partial x} - i\frac{\partial}{\partial y} = (\cos\theta - i%\sin\theta)\frac{\partial}{\partial x'} - 
%		(\sin\theta + i\cos\theta)\frac{\partial}{\partial y'} = e^{-i\theta}
%	  \left( \frac{\partial}{\partial x'} - i\frac{\partial}{\partial y'} \right),
%\end{eqnarray*}
%i.e., 
\begin{align*}
\partial_z &= e^{i\theta}{\partial}_{z'}.\\
\overline{\partial}_z &=e^{-i\theta}\overline{\partial}_{z'}
\end{align*} 

Under such rotations, equation \eqref{eq:NVComplex2} becomes
\begin{equation} q_t = e^{i\theta}\partial_{z'}(\nu' q) + e^{i\theta}\overline{\partial}_{z'}(\overline{\nu'}q) - e^{-3i\theta}\overline{\partial}^3_{z'}q - e^{-3i\theta}\partial^3_{z'}q
\end{equation}
where $\nu' = e^{-i\theta}\nu$.  The auxiliary equation becomes \[ e^{i\theta}\overline{\partial}_{z'}\nu = e^{i\theta}\partial_{z'}q\] or \[ \overline{\partial}_{z'}\nu' = e^{-3i\theta}\partial_{z'}q,\] and so we have invariant solutions under rotations of $2\pi/3$ and $4\pi/3$.  This shows that if a solution to the NV equation has this symmetry, it must be preserved under the evolution.  It does not mean that all solutions will display this type of symmetry.

\subsection{Conservation Laws for the NV equation}

In order to present the conservation laws for \eqref{eq:NV}, we need to recall some ideas from the inverse scattering method. A rigorous derivation of the conservation laws for smooth potentials of
conductivity type is given below in section \ref{sec:inverse}.  

Suppose that $q \in L^p(\mathbb{R}^2)$ for some $p \in (1,2)$. The scattering data, or \textit{scattering transform} $\mathbf{t}:\mathbb{C}\rightarrow\mathbb{C}$ of $q$ is defined via 
Faddeev's \cite{Faddeev:1965} complex geometric optics (CGO) solutions,  which we now recall. Let $z = x + iy$ and $k = k_1 + ik_2$. For $k \in \mathbb{C}$ with $k \neq 0$, the function $\psi(z,k)$ is the exponentially growing solution of the Schr{\"o}dinger equation 
\begin{equation} \label{eq:CGO-1} 
(-\Delta + q)\psi(\cdot,k) = 0 
\end{equation}
with asymptotic behavior $\psi(z,k)\sim e^{ikz}$ in the following sense: for $\tilde{p}$ defined by 
$ \frac{1}{\tilde{p}} = \frac{1}{p}-\frac{1}{2}$,  we have
\begin{equation} \label{asy_cond_psi}
  e^{-ikz}\psi(z,k)-1 \in L^{\tilde{p}}(\mathbb{R}^2) \cap L^\infty({\mathbb R}^2). 
\end{equation}  
It is more convenient to work with the {\em normalized complex geometric optics solutions} (NCGO) $\mu(z,k)$ defined by
\begin{equation} \label{NCGO}
\mu(z,k) = \psi(z,k)e^{-ikz}.
\end{equation}

A straightforward computation shows that $\mu$ obeys the equation
\begin{align} \label{eq:mu}
\overline{\partial} \left(\partial+ik\right)\mu = \frac{1}{4}q \mu, \qquad \mu(\cdot,k)-1 &\in L^{\tilde{p}} \cap L^\infty.
\end{align} 

One can reduce the problem \eqref{eq:mu} to an integral equation of Fredholm type (see the discussion in section \ref{sec:inverse}). Faddeev's Green's function $g_k$ is the fundamental solution for the equation 
$$-4\overline{\partial}(\partial+ik)u=f$$
(the factor of $-4$ is chosen so that, if $k=0$, the equation reduces to $-\Delta u = f$ whose fundamental solution is the logarithmic potential; see Appendix \ref{sec:Faddeev} for details). One has
\begin{equation}
\label{eq:mu.int}
\mu = 1 - g_k*(q \mu)
\end{equation}
(where $*$ denotes convolution) and it can be shown that the operator $\varphi \mapsto g_k*(q\varphi)$ is compact on $L^{\tilde{p}}$.  Thus, for given $k$, the solution $\mu(z,k)$ exists if and only if it is unique. 

It is known that such solutions exist for any $q \in L^p$ provided that $|k|$ is sufficiently large.  In general, however, the equation \eqref{eq:mu} need not have a unique solution for every $k$.
 
Points for which uniqueness fails, i.e., points for which the homogeneous problem
\begin{align*}
\overline{\partial} \left(\partial+ik\right)\mu &= \frac{1}{4}q \mu, \\
\mu(\cdot,k) &\in L^{\tilde{p}} \cap L^\infty.
\end{align*}
has a nontrivial solution are called \emph{exceptional points}. It is known that the exceptional points form a bounded closed set in $\mathbb{C}$. Nachman \cite{Nachman:1996} proved that the exceptional set is empty for potentials of conductivity type; more recently, Music \cite{Music:2013} has shown that the same is true for subcritical potentials. We discuss this further in section \ref{sec:zero_excep}.

Let%
\begin{equation}
\label{eq:ek}
e_{k}(z)=e^{i\left(  kz+\overline{k}\overline{z}\right)  }.%
\end{equation}
Then, if $q$ decays rapidly at infinity (say $q \in \mathcal{S}(\mathbb{R}^n)$),
the function $\mu(z,k)$ obeys the large-$z$ asymptotic formula%
\begin{equation}
\mu(z,k)\sim1+\frac{1}{4\pi ikz}s(k)-\frac{e_{-k}(z)}{4\pi i\overline
{k}\overline{z}}\mathbf{t}\left(  k\right)  +\mathcal{O}\left(  \left\vert
z\right\vert ^{-2}\right)  \label{eq:mu.asy}%
\end{equation}
where%
\begin{align}
\mathbf{t}(k)  &  =\int e_{k}(z)q(z)\mu(z,k)~dz,\label{eq:t}\\
s(k)  &  =\int q(z)\mu(z,k)~dz. \label{eq:s}%
\end{align}
Note that $s(k)$ and $\mathbf{t}(k)$ are \emph{always} defined for large $|k|$, whether the potential $q$ is subcritical, critical, or supercritical.

The asymptotic formula (\ref{eq:mu.asy}) is a consequence of the following
simple lemma.

\begin{lemma}
\label{lemma:mu.exp.z}Suppose $k\in\mathbb{C}$ and $k\neq0$,that $p>2$ and
suppose that $u\in C^{2}\left(  \mathbb{C}\right)  \cap L^{p}\left(
\mathbb{C}\right)  $ satisfies%
\[
-4\overline{\partial}\left(  \partial+ik\right)  u=f
\]
for $f\in\mathcal{S}\left(  \mathbb{C}\right)  $. Then%
\begin{equation}
u(z)\underset{\left\vert z\right\vert \rightarrow\infty}{\sim}\sum_{\ell\geq
0}\frac{a_{\ell}}{z^{\ell+1}}+e_{-k}(z)\frac{b_{\ell}}{\overline{z}^{\ell+1}}
\label{eq:FGF.exp}%
\end{equation}
where%
\begin{align*}
a_{0}  &  =\frac{1}{4\pi ik}\int f(z)~dz,\\
b_{0}  &  =\frac{1}{4\pi i\overline{k}}\int e_{k}(z)f(z)~dz.
\end{align*}

\end{lemma}

\begin{proof}
The conditions on $u$ imply that
\[
u=g_{k}\ast f
\]
so using the asymptotic expansion (\ref{eq:gk.exp}) for $g_k$ we obtain%
\begin{align}
u(z)  &  =-\frac{1}{4\pi}\sum_{j=0}^{N}%
{\displaystyle\int}
\left[  \frac{j!}{\left(  ik(z-z')\right)  ^{j+1}}+e^{-i\left(  kz+\overline
{k}\overline{z}\right) } 
\frac{j!}{\left(-i\overline{k}\left(\overline{z-z'}\right)\right)
^{j+1}}\right]  f(z^{\prime})~dz^{\prime}\label{eq:u.exp.pre}\\
&  +\mathcal{O}\left(  \left\vert z\right\vert ^{-N-2}\right) \nonumber
\end{align}
It is not difficult to see that for $f\in\mathcal{S}\left(  \mathbb{C}\right)
$ and any positive integer $N$, the expansion
\[
\int\frac{1}{\left(  z-z^{\prime}\right)  ^{j}}f(z^{\prime})~dz^{\prime}%
=\sum_{\ell=0}^{N}c_{\ell,j}z^{-j-\ell}+\mathcal{O}\left(  \left\vert
z\right\vert ^{-N-1-j}\right)
\]
holds, with an analogous expansion for the terms involving $\overline
{z}-\overline{z^{\prime}}$. The existence of the expansion (\ref{eq:FGF.exp}) is
immediate. The leading terms come from the $j=0$ term of (\ref{eq:u.exp.pre}).
\end{proof}

If $q(z,\tau)$ solves the NV equation at zero energy, it can be shown that $\mathbf{t}(\cdot,\tau)$, the scattering transform of $q(\cdot,\tau)$, is then given by 
\begin{equation}
\label{eq:t.motion}
\mathbf{t}(k,\tau) = \mathbf{m}(k,\tau)\mathbf{t}(k,0),
\end{equation}
where 
$$\mathbf{m}(k,\tau) = \mbox{exp}(i\tau(k^3 + \overline{k}^3)).$$ 
On the other hand, 
\begin{equation}
\label{eq:s.motion}
s(k,\tau) = s(k,0).
\end{equation}
Here $\mathbf{t}(k,\tau)$ and $s(k,\tau)$ are computed from the solution $\mu(z,k,\tau)$ of
\begin{equation} 
\label{Dbar_for_mu}
\overline{\partial}_z\left(\partial_z +ik \right)\mu(z,k,\tau)= \frac{\textstyle{1}}{\textstyle{4}}
q(z,\tau) \mu(z,k,\tau).
\end{equation}
A rigorous proof of \eqref{eq:t.motion} and \eqref{eq:s.motion} for smooth potentials of conductivity type is given in section \ref{sec:inverse} below.

If $q$ is smooth, rapidly decreasing, and either critical or subcritical, the Schr{\"o}\-dinger potential $q$ can be recovered using the $\overline{\partial}$-method of Beals and Coifman \cite{BC:1989} (see \cite{LMS:2007} for the critical case, and \cite{Music:2013} for the subcritical case). Both of these papers use techniques developed by Nachman \cite{Nachman:1996} in the context of the inverse conductivity problem. 

We can now derive a set of conservation laws for the NV equation by using the large-$k$ asymptotic expansion of $s(k)$. Since $s(k,\tau)$ is conserved we set $\tau=0$ and suppress $\tau$-dependence henceforward. We will give a formal derivation of the conserved quantities by using a large-$k$ asymptotic expansion of $\mu(z,k)$ and inserting this expansion into the formula \eqref{eq:s}. Since $s(k)$ is conserved, the coefficients of that large-$k$ expansion are also conserved quantities. For the moment, we assume that $\mu(z,k)$ admits a large-$k$ expansion of the form
\begin{equation}
\mu(z,k) \sim  1 + \sum_{j=1}^{\infty}\frac{a_j(z)}{k^j}.
\label{eq:muSeries}
\end{equation}   
We will derive such an expansion for smooth potentials of conductivity type in the next section (see Lemma \ref{lemma:mu.exp.k}). It is expected to hold in general. 

Substituting the series \eqref{eq:muSeries} into \eqref{Dbar_for_mu}, we may solve the resulting system for the coefficients $a_j$
\begin{equation}
-\sum_{j=1}^{\infty}\frac{\Delta a_j(z)}{k^j} - \sum_{j=1}^{\infty}\frac{4i\overline{\partial} a_j(z)}{k^{j-1}} + q\sum_{j=1}^{\infty}\frac{a_j(z)}{k^j} = -q.
\label{eq:SeriesConv}
\end{equation}
We find \[ a_1 = \frac{1}{4i}\overline{\partial}^{-1}q.\]  A recursion formula can then be derived, \begin{equation}
a_{j+1} = \frac{1}{4i}\overline{\partial}^{-1}(-4\overline{\partial}\partial a_j + qa_j) = i\partial a_j + \frac{1}{4i}\overline{\partial}^{-1}(qa_j).
\label{eq:RecConver}
\end{equation}
From this, it is clear that
$$ s(k) \sim \sum_{j=0}^\infty \frac{s_j}{k^{j}}, \quad s_j = \int_{\mathbb{R}^2} q(z) a_j(z) \, dz. $$
Thus, the first three conserved quantities are
\begin{align*}
s_0 & = \int_{\mathbb{R}^2}q(z)dz, \\
s_1 & = \int_{\mathbb{R}^2}\frac{1}{4i}q(z)(\overline{\partial}^{-1}q)(z)dz,\\
s_2 & = \int_{\mathbb{R}^2}\left(\frac{1}{4}q(z)v(z) - \frac{1}{16}q\overline{\partial}^{-1}(q\overline{\partial}^{-1}q)(z)\right) dz,
\end{align*} 
where with $z = x + iy$ and $\zeta= \zeta_1 + i\zeta_2$,
\[ (\overline{\partial}^{-1}q)(z) = \frac{1}{\pi}\int_{\mathbb{R}^2}\frac{q(\zeta)d\zeta}{z - \zeta}.\]  

\section{Inverse Scattering via Manakov Triples}
\label{sec:inverse}

In this section we develop the inverse scattering method to solve the Cauchy
problem for the Novikov-Veselov (NV) equation at zero energy
\begin{align}
q_{\tau}&=2\operatorname{Re}\left(  \partial^{3}q-\frac{3}{4}\partial(uq)\right)
\label{eq:NV.Cauchy}\\
\overline{\partial}u &  =\partial q\nonumber\\
\left.  q\right\vert _{\tau=0} &  =q_{0}\nonumber
\end{align}
for smooth Cauchy data $q_{0}$ of \emph{conductivity type} (see Definition
\ref{def:cond}). Note that our convention for the NV equation differs slightly from \eqref{eq:NV}; the form used here is more convenient for the zero-energy inverse scattering formalism. This section should be regarded as expository and\ the
material here is undoubtedly \textquotedblleft well known to the
experts\textquotedblright\ (see the Manakov \cite{Manakov:1976}, Nizhnik \cite{Nizhnik:1980}, and see e.g. Boiti, Leon, Manna, and Pempenelli
\cite{BLMP:1987} for the NV equation), although we give an essentially self-contained and
mathematically rigorous presentation.  

Earlier, Perry \cite{Perry:2013} exploited Bogdanov's \cite{Bogdanov:1987} observation that the NV equation is connected to the modified Novikov-Veselov (mNV) equation by a map analogous to the Miura map that connects the mKdV and KdV equations \cite{Miura:1968}. The mNV equation belongs to the Davey-Stewartson hierarchy, and the Miura map from mNV to NV has range precisely in the potentials of conductivity type. Building on the analysis of the Davey-Stewartson scattering maps in \cite{Perry:2012}, Perry solves the NV equation for initial data of conductivity type with additional smoothness and decay assumptions. 

The approach taken here is more robust because it works directly with the scattering transform for NV and avoids the Miura-type map. An extension of these ideas to subcritical potentials will appear in \cite{Music:2013} and  \cite{MP:2013}.

To analyze the scattering transform for NV, we draw on previous work of 
Lassas, Mueller, and Siltanen \cite{LMS:2007}, Lassas, Mueller, Siltanen, 
and Stahel \cite{LMSS:2011}, and
Perry \cite{Perry:2012}, particularly for mapping properties of the scattering
transform and its inverse on the space of smooth functions of conductivity
type as defined below. \ The main ingredient in our analysis (as
contrasted to \cite{LMS:2007,LMSS:2012,Perry:2013}) is the systematic use of the Manakov triple representation for the NV equation. 

To describe the Manakov triple representation, suppose that $q$ is a smooth function of $z$ and $\tau$. Suppose that there is a smooth function $u(z,\tau)$ with the property that $\overline{\partial}u=\partial q$ (existence of such a function for suitable classes of $q$ can be deduced from properties of the Beurling transform; see the Appendix). Define\footnote{This Manakov triple differs from that of the introduction by numerical factors since we use, for convenience, the version \eqref{eq:NV.Cauchy} of the NV equation.}
\begin{align}
L  &  =\partial\overline{\partial}-q/4,\label{eq:L}\\
A  &  =\partial^{3}+\overline{\partial}^{3}-\frac{3}{4}\left(  u\partial
+\overline{u}\overline{\partial}\right)  ,\label{eq:A}\\
B  &  =\frac{3}{4} \left(  \partial u+\overline{\partial}\overline{u}\right)  , \label{eq:B}%
\end{align}
where $B$ is a multiplication operator. The $(L,A,B)$ representation means the following:

\begin{proposition}
Let $q\in\mathcal{C}^{1}\left(  \left[  0,T\right]  ;\mathcal{C}^{\infty
}\left(  \mathbb{C}\right)  \right)  $ and suppose that there is a smooth function $u$ with $\overline{\partial}u=\partial q$. Then $q$ is a classical solution to the Cauchy problem \eqref{eq:NV.Cauchy}  if and only if the operator identity
\begin{equation}
\dot{L}=\left[  A,L\right]  -BL \label{eq:Manakov}%
\end{equation}
holds.
\end{proposition}

The proof is a straightforward calculation. The significance of the Manakov
triple is that it defines a scattering problem at zero energy for the operator
$L$, and a law of evolution of scattering data through the operator $A$. We
will fully describe the inverse scattering method for smooth initial data of
conductivity type, defined as follows:

\begin{definition}
\label{def:cond}A function $q_{0}\in\mathcal{C}^{\infty}(\mathbb{C})$ is a
smooth function of conductivity type if
\[
q_{0}=2\partial u_{0}+\left\vert u_{0}\right\vert ^{2}%
\]
for some $u_0\in\mathcal{S}\left(  \mathbb{C}\right)  $ with $\partial
u_0=\overline{\partial}\overline{u_0}$.
\end{definition}

The regularity requirements can be considerably relaxed but we make them here to ease the exposition. To compare this definition with Nachman's definition (see Definition \ref{def:cond.Nachman} in the introduction), one should think of $u_0$ as $ 2 \overline{\partial}\log \sigma$. 

We develop in turn the direct scattering transform,
the inverse scattering transform, and the solution formula for NV. We also comment on numerical methods for implementing the direct and inverse scattering transforms.

\subsection{The Direct Scattering Map}

To compute the\emph{ }scattering transform of $q_{0}$, one first constructs
the complex geometric optics (CGO) solutions to \eqref{eq:CGO-1}. Analytically, it is
more convenient to study the normalized complex geometric optics (NCGO)
solutions $\mu$ defined by (\ref{NCGO}). As shown by Nachman \cite{Nachman:1996}, there exists a unique solution of \eqref{eq:mu} for every nonzero $k$, so that $\mathbf{t}(k)$ is defined for every nonzero $k$. Nachman also shows that $\mathbf{t}(k)$ is $\mathcal{O}(|k|^\varepsilon)$ as $k \rightarrow 0$ for conductivity-type potentials.

\begin{definition}
The map $\mathcal{T}:q\mapsto\mathbf{t}$ defined by the problem (\ref{eq:mu})
and the representation formula (\ref{eq:t}) is called the \emph{direct
scattering map}.
\end{definition}

We will use without proof the following result of Lassas, Mueller, Siltanen,
Stahel \cite{LMSS:2011} (see also \cite{Perry:2013} for a different proof)

\begin{lemma}
Suppose that $q$ is a smooth function of conductivity type, and let
$\mathbf{t}$ be the scattering transform of $q$. Then $\mathbf{t}\left(
k\right)  /\overline{k}\in\mathcal{S}\left(  \mathbb{C}\right)  $.
\end{lemma}

\subsection{Computation of Scattering Transforms}
\label{sec:scatcomp}

We describe two approaches for the computation of $\mathbf{t}=\mathcal{T}q$ for a given compactly supported and continuous $q$. The {\em LS method} is most accurate for $k$ away from zero, and the {\em DN method}  is more effective for $k$ near zero. Matlab codes for both approaches are available at the webpage of the book \cite{Mueller:2012}.

Without loss of generality we can assume that $\mbox{supp}(q)\subset\Omega$ where $\Omega\subset\R^2$ is the open unit disc. The LS method is based on the definition 
\begin{equation}\label{approach1formula}
\mathbf{t}(k)  =\int_{\R^2} e^{i\overline{k}\overline{z}}q(z)\psi(z,k)~dz,
\end{equation}
and the DN method uses integration by parts to transform (\ref{approach1formula}) into
\begin{equation}\label{approach2formula}
\mathbf{t}(k)  =\int_{\partial\Omega}e^{i\overline{k}\overline{z}}(\Lambda_q-\Lambda_0)\psi(\,\cdot\,,k)~dS(z),
\end{equation}
where $\Lambda_q$ is the {\em Dirichlet-to-Neumann map} defined below. A rigorous derivation of formulas (\ref{approach1formula}) and (\ref{approach2formula}) was given by Nachman in \cite{Nachman:1996}.

The LS method requires numerical evaluation of the  {\em complex geometrical optics solutions} $\psi(z,k)$. Numerically, it is better to solve the Lippmann-Schwinger equation for $\mu$
\begin{equation}\label{mLS}
  \mu(z,k) = 1-\int_{\Omega}g_k(z-w)q(w)\mu(w,k)~dw.
\end{equation}
A rigorous solvability analysis for equation (\ref{mLS}) can be found in \cite{MPS:2013,Nachman:1993,Nachman:1996}. 
Here $g_k$ is the fundamental solution satisfying $ (-\Delta-4ik\dbar)g_k(z) = \delta(z)$. The origin of $g_k$ is Faddeev's 1965 article \cite{Faddeev:1965}. Computationally, $g_1(z)$ can be evaluated using the Matlab expression
``{\tt  exp(-1i*z).*real(expint(-1i*z))/(2*pi);}'' The symmetry relation $g_k(z)=g_1(kz)$ extends this to all values  $k\not=0$. Note that  $g_k$ has a $\log k$ singularity  when $k\rightarrow 0$, causing numerical difficulties for $k$ near zero.

Equation (\ref{mLS}) is defined in the whole plane $z\in\R^2$, so some kind of truncation is needed for practical computation. The first numerical computation of complex geometrical optics solutions was reported in \cite{Siltanen:1999} in the context of (\ref{mLS}). That computation was used as a part of the first numerical implementation \cite{SMI:2000} of the $\overline{\partial}$ method for electrical impedance tomography. A more effective approach for computing $\mu$ is based on the periodization technique introduced by Gennadi Vainikko in \cite{Vainikko:2000}; see also \cite[Section 10.5]{Saranen:2002}. The adaptation of Vainikko's method to equation (\ref{mLS}) was first introduced in \cite{Mueller:2003}. For more details see \cite[Section 14.3]{Mueller:2012}.

Now let us turn to the DN method.  This method has practical use in the $\overline{\partial}$ reconstruction method for electrical impedance tomography.  (See \cite{Mueller:2012} and the relevant references therein.)  We first define the Dirichlet-to-Neumann map $\Lambda_{q}$ for the Dirichlet problem
\begin{eqnarray}
\left(  -\Delta+q\right)  u  &  =&0\mbox{ in }\Omega\label{eq:DP}\\
\left.  u\right\vert _{\partial\Omega}  &  =&f.\nonumber
\end{eqnarray}
If zero is not a Dirichlet eigenvalue of $-\Delta+q$ in $\Omega$,
the problem (\ref{eq:DP}) has a unique solution $u$
 for given $f\in
H^{1/2}(S^{1})$, and we set%
\begin{equation}\label{def:Lambda_q}
\Lambda_{q}f=\left.  \frac{\partial u}{\partial\nu}\right\vert _{\partial\Omega}%
\end{equation}
where $\partial/\partial\nu$ denotes differentiation with respect to the
outward normal on $\partial\Omega$. 

Formula (\ref{approach2formula}) requires the trace  $\psi(\,\cdot\,,k)\vert _{\partial\Omega}$. According to \cite{Nachman:1996}, the traces can be solved from the boundary integral equation 
\begin{equation}\label{eka_int_eq}
  \psi(z,k)|_{\partial\Omega} = e^{ikz} - \int_{{\partial\Omega}}G_k(z-w)\,(\Lambda_q-\Lambda_0)\psi(w,k)|_{\partial\Omega}~dS(w),
\end{equation}
if $k$ is not an exceptional point of $q$. Here $G_k(z):=e^{ikz}g_k(z)$ is Faddeev's Green function for the Laplace operator. For details of the numerical solution of (\ref{eka_int_eq}) see \cite[Section 15.3]{Mueller:2012}.

\subsection{The Inverse Scattering Map}\label{sec:InvScat}

It turns out the the NCGO solutions $\mu(z,k)$ also solve a $\overline
{\partial}$-problem in the $k$ variable determined by $\mathbf{t}(k)$. Letting%
\begin{equation}
\mathbf{t}^{\sharp}(k)=\frac{\mathbf{t}(k)}{4\pi\overline{k}},
\label{eq:tsharp}%
\end{equation}
we have, for any $p>2$:

\begin{lemma}
Suppose that $q$ is a smooth potential of conductivity type, and let
$\mu(z,k)$ be the corresponding NCGO\ solutions. Then:%
\begin{align}
\overline{\partial}_{k}\mu(z,k)  &  =e_{-k}\mathbf{t}^{\sharp}(k)\overline
{\mu(z,k)}\label{eq:mu.dbar.k.pre}\\
\mu(z,~\cdot~)-1  &  \in L^{p}\left(  \mathbb{C}\right)  .\nonumber
\end{align}
Moreover, $q(z)$ is recovered from $\mathbf{t}$ and $\mu(z,k)$ by the formula%
\begin{equation}
q(z)=\frac{4i}{\pi} \overline{\partial}_{z}\left(  \int e_{-k}(z)\mathbf{t}^{\sharp}%
(k)\overline{\mu(z,k)}~dk\right)  \label{eq:q.recon.pre}%
\end{equation}

\end{lemma}

\begin{remark}
This equation has at most one solution for each $z$, provided only that
$\mathbf{t}^{\sharp}\in L^{2}$ by a standard uniqueness theorem for the
$\overline{\partial}$-problem (see Brown-Uhlmann \cite{BU:1997}, Corollary 3.11).
\end{remark}

The fact that $\mu(z,k)$ obeys a $\overline{\partial}$-equation in the
$k$-variable also implies a large-$k$ asymptotic expansion for $\mu(z,k)$.

\begin{lemma}
\label{lemma:mu.exp.k}Suppose that $q$ is a smooth potential of conductivity
type, and let $\mu(z,k)$ be the corresponding NCGO solution. Then%
\begin{equation}
\mu(z,k)\sim1+\sum_{j\geq0}\frac{c_{j}(z)}{k^{j+1}}
\label{eq:mu.exp.k}%
\end{equation}
where%
\begin{align}
c_{0}  &  =-\frac{i}{4}\overline{\partial}^{-1}q\label{eq:c0}\\
c_{1}  &  =-\frac{1}{4}\partial\overline{\partial}^{-1}q+\frac{1}{16}%
\overline{\partial}^{-1}\left(  q\overline{\partial}^{-1}q\right)  .
\label{eq:c1}%
\end{align}
and the remaining $c_{j}$ are determined by the recurrence%
\begin{equation}
i\overline{\partial}c_{j+1}=\left(  \frac{q}{4}-\partial\overline{\partial
}\right)  c_{j} \label{eq:c.recur}%
\end{equation}

\end{lemma}

\begin{proof}
The coefficients in the asymptotic expansion may be computed recursively from
the equation $\overline{\partial}\left(  \partial+ik\right)  \mu=(q/4)\mu$
once the existence of the asymptotic expansion is established. To do so, note
that
\[
\mu(z,k)=1+\frac{1}{\pi}\int_{\mathbb{C}}\frac{1}{k-\kappa}e_{-k}%
\mathbf{t}^{\sharp}(\kappa)\overline{\mu(z,\kappa)}~dm(\kappa).
\]
Writing
\[
\left(  k-\kappa\right)  ^{-1}=k^{-1}\sum_{j=0}^{N}\left(  \frac{\kappa}%
{k}\right)  ^{j}+\left(\frac{\kappa}{k}\right)^{N+1}\frac{1}{k-\kappa}%
\]
we obtain an expansion of the desired form with
\[
c_{j}(z)=\frac{1}{\pi}\int\kappa^{j}e_{-k}\mathbf{t}^{\sharp}(\kappa
)\overline{\mu(z,\kappa)}~dm(\kappa)
\]
and remainder%
\[
\frac{1}{\pi}k^{-(N+1)}\int\frac{1}{k-\kappa}\kappa^{N+1}e_{-k}\mathbf{t}^{\sharp
}(\kappa)\overline{\mu(z,\kappa)}~dm(\kappa)
\]
Finally to obtain explicit formulae for the $c_{j}$, we substitute the
expansion (\ref{eq:mu.exp.k}) into (\ref{eq:mu}) to obtain%
\[
i\overline{\partial}c_{0}=\frac{1}{4}q
\]
and the recurrence (\ref{eq:c.recur}) from which (\ref{eq:c1}) follows.
\end{proof}

Motivated by these results, suppose given $\mathbf{t}^{\sharp}\in
\mathcal{S}\left(  \mathbb{C}\right)  $ and let $\mu(z,k)$ be the unique
solution to (\ref{eq:mu.dbar.k.pre}). \emph{Define} $q$ by the reconstruction
formula (\ref{eq:q.recon.pre}). Then the solution $\mu$ of
(\ref{eq:mu.dbar.k.pre}) obeys the partial differential equation%
\begin{align*}
\overline{\partial}\left(  \partial+ik\right)  \mu   =\frac{q}{4}\mu,\qquad \lim_{\left\vert z\right\vert \rightarrow\infty}\mu(z,~\cdot~)-1   =0.
\end{align*}

\begin{definition}
The map $\mathcal{Q}:\mathbf{t}\rightarrow q$ defined by
(\ref{eq:mu.dbar.k.pre}) and (\ref{eq:q.recon.pre}) is called the
\emph{inverse scattering map}.
\end{definition}

We conclude this subsection by obtaining a full asymptotic expansion for
$\mu(z,k)$ which encodes relations between $s$ and $\mathbf{t}$.

\begin{lemma}
\label{lemma:mu.exp.z.bis}Suppose that $q$ is a smooth potential of
conductivity type, and let $\mathbf{t}^{\sharp}$ be given by (\ref{eq:tsharp}%
). Then, the expansion%
\begin{equation}
\mu(z,k)\underset{\left\vert z\right\vert \rightarrow\infty}{\sim}1+\sum
_{\ell\geq0}\left(  \frac{a_{\ell}}{z^{\ell+1}}+e_{-k}\frac{b_{\ell}%
}{\overline{z}^{\ell+1}}\right)  \label{eq:mu.exp2}%
\end{equation}
holds, where:%
\begin{align*}
a_{0}  &  =-i\overline{\partial}_{k}^{-1}\left(  \left\vert \mathbf{t}%
^{\sharp}\right\vert ^{2}\right)  ,\\
b_{0}  &  =i\mathbf{t}^{\sharp}%
\end{align*}
and the subsequent $a_{\ell},b_{\ell}~\ $are determined by the recurrence
relations%
\begin{align*}
\overline{\partial}_{k}a_{\ell}  &  =\mathbf{t}^{\sharp}\overline{b_{\ell}},\\
b_{\ell+1}  &  =i\overline{a_{\ell}}\mathbf{t}^\sharp-i\overline{\partial}_{k}b_{\ell}.
\end{align*}

\end{lemma}

\begin{proof}
The existence of an expansion of the form (\ref{eq:mu.exp2}) was already
established in Lemma \ref{lemma:mu.exp.z}. To compute the coefficients, we
substitute the asymptotic series into (\ref{eq:mu.dbar.k.pre}).
\end{proof}

\begin{remark}
Comparing Lemmas \ref{lemma:mu.exp.z} and \ref{lemma:mu.exp.z.bis}, we see that%
\[
s(k)=4 \pi k\overline{\partial}_{k}^{-1}\left(  \left\vert \mathbf{t}%
^{\sharp}\right\vert ^{2}\right)
\]

\end{remark}

\subsection{Computation of Inverse Scattering Transforms}
\label{sec:invscatcomp}

The first step in the computation of the inverse scattering map is to solve the $\overline{\partial}$ equation
\begin{equation}\label{dbar_eq_numer}
  \frac{\partial}{\partial\overline{k}} \,\mu_\tau(z,k) = \frac{\T_\tau(k)}{4\pi\overline{k}}e_{-k}(z)\overline{\mu_\tau(z,k)} 
\end{equation}
with a fixed parameter $z\in\mathbb{R}^2$ and requiring large $|k|$ asymptotics $\mu_{\tau}(z,\cdot)-1\in L^\infty\cap L^{r}(\mathbb{C})$ for some $2<r<\infty$.  Since $q_{0}(z)$ is compactly supported and of conductivity type, by \cite{LMS:2007} the scattering transform $\T_\tau(k)$ is in the Schwartz class, and the solution $\mu_{\tau}$ to equation \eqref{dbar_eq_numer} can be computed by numerically solving the integral equation 
\begin{equation}\label{int_form}
\mu_{\tau}(z,k) = 1+\frac{1}{4\pi^2}\int_{\R^2}\frac{\T_{\tau}(k')}{\bar{k'}(k-k')}e_{-k'}(z)\overline{\mu_{\tau}(z,k')}dk'.
\end{equation}
%%% PAP changed wording
To compute the solution of the $\overline{\partial}$ equation (\ref{dbar_eq_numer}),  the scattering transform $\mathbf{t}_\tau(k)$ is truncated to a large disc of radius $R$, generally chosen by inspection of the
scattering transform.  The truncated integral equation is solved numerically by the
method described in \cite{Knudsen:2004} for each point 
$z$ at which the evolved potential is to be computed.  The method in
\cite{Knudsen:2004} is based on Vainikko's periodic method
\cite{Vainikko:2000}; see also \cite[Section 10.5]{Saranen:2002}.
Note  that since the $\dbar$ equation \eqref{dbar_eq_numer} is
real-linear and not complex-linear due to the complex conjugate on the
right-hand side of \eqref{dbar_eq_numer}, one must write the real and
imaginary parts of the unknown function $\mu$ separately in the vector
of function values at the grid points.   It is proven in
\cite{KLMS:2009} that the error  decreases as $R$ tends to infinity.
The first computational solutions of equation \eqref{dbar_eq_numer}
can be found in \cite{SMI:2000}, and the first computations based on
\cite{Vainikko:2000} are found in \cite{Knudsen:2004}; for more
details see \cite[Section 15.4]{Mueller:2012}.

The inverse scattering transform is defined by
\begin{equation}\label{def:QQp}
    (\mathcal{Q}\mathbf{t}_\tau)(z) := \frac{i}{\pi^2}\overline{\partial}_z\int_{\mathbb{C}}\frac{\mathbf{t}_\tau(k)}{\overline{k}}\,e^{-ikz}\,\overline{\psi_\tau(z,k)} dk, 
\end{equation}
where $\psi_\tau(z,k):=e^{ikz}\mu_\tau(z,k)$.
The inverse transform (\ref{def:QQp}) first appeared in \cite[formula (4.10)]{BLMP:1987}.  See \cite{LMS:2007} and the references therein for an analysis of the solvability of (\ref{dbar_eq_numer}) and the domain of definition for (\ref{def:QQp}).  Under the assumption that real-valued, smooth initial data of conductivity type remain of conductivity type under evolution by the ISM, the conductivity $\gamma_{\tau}$ associated with the potential $q_{\tau}$ is given by
$$ \gamma_{\tau}^{1/2}(z) = \mu_{\tau}(z,0).$$
Then $q_{\tau}$ is computed by numerical
differentiation of $\gamma_\tau$ by the formula
$$q_{\tau}(z) = \gamma_{\tau}^{-1/2}\Delta\gamma_{\tau}^{1/2}.$$

The reader is referred to \cite{LMSS:2012} for numerical examples of the computation of the time evolution of conductivity-type potentials by the ISM.

\subsection{Time Evolution of NCGO\ Solutions}

In order to prove the solution formula (\ref{eq:NV.ISM}), we first study the
time evolution of NCGO\ solutions using the Manakov triple representation.
First we note the following important uniqueness theorem which is actually a special case of results of Nachman.

\begin{theorem}
\label{thm:unique1} Let $k\in\mathbb{C}$, $k\neq0$. Suppose that $q$ is a
smooth potential of conductivity type and that $\psi$ is a solution of
$L\psi=0$ with $\lim_{\left\vert z\right\vert \rightarrow\infty}\left(
e^{-ikz}\psi(z)\right)  =0$. Then $\psi(z)=0$.
\end{theorem}

Now suppose that $q(z,\tau)$ solves the NV equation and that $t\mapsto q(z,\tau)$ is
a $C^{1}$ map from $\left[  0,T\right]  $ into $\mathcal{S}\left(
\mathbb{C}\right)  $. Suppose that, for each $\tau$, $\varphi(\tau)$ solves
$L(\tau)\varphi(\tau)=0$. Differentiating the equation $L(\tau)\varphi(\tau)=0$ and using
the Manakov triple representation, we find%
\[
\left[  A(\tau),L(\tau)\right]  \varphi(\tau)+L(\tau)\dot{\varphi}(\tau)=0
\]
or%
\[
L(\tau)\left[  \dot{\varphi}(\tau)-A(\tau)\varphi(\tau)\right]  =0
\]
From this simple computation and Theorem \ref{thm:unique1}, we can derive an
equation of motion for the NCGO\ solutions, and recover (\ref{eq:t.motion}%
) and (\ref{eq:s.motion}) from a careful calculation of asymptotics. Later, we will
show by explicit construction that, if $q_{0}$ is a smooth function of
conductivity type, then there is a solution $q(z,\tau)$ of the NV equation so
that $q(z,\tau)$ is smooth and of conductivity type for each $\tau$.

\begin{lemma}
Suppose that $q(z,\tau)$ is a solution of the NV\ equation where, for each $\tau$,
$q(z,\tau)$ is a smooth function of conductivity type. Let $u=\overline{\partial
}^{-1}\partial q$. Then
\begin{equation}
\dot{\mu}=ik^{3}\mu+\left(  \partial+ik\right)  ^{3}\mu+\overline{\partial
}^{3}\mu-\frac{3}{4}u\left(  \partial+ik\right)  \mu-\frac{3}{4}\overline
{u}\overline{\partial}\mu\label{eq:mudot}%
\end{equation}

\end{lemma}

\begin{proof}
Before giving the proof we make several remarks. Since $\psi=e^{ikz}\mu$, the
evolution equation (\ref{eq:mudot}) is equivalent to%
\begin{equation}
\dot{\psi}=ik^{3}\psi+\partial^{3}\psi+\overline{\partial}^{3}\psi-\frac{3}%
{4}\left(  u\partial+\overline{u}\overline{\partial}\right)  \psi.
\label{eq:psidot}%
\end{equation}
Next, let $\varphi(z,k,\tau)=e^{iS}\mu(z,k,\tau)$ with $S(z,k,\tau)=kz-k^{3}\tau$. From the argument above we have%
\[
L(\tau)\left[  \dot{\varphi}(\tau)-A(\tau)\varphi(\tau)\right]  =0.
\]
To conclude that $\dot{\varphi}(\tau)=A(\tau)\varphi(\tau)$, we must show that
\[
\lim_{\left\vert z\right\vert \rightarrow\infty}\left(  e^{-ikz}\left[
\dot{\varphi}(\tau)-A(\tau)\varphi(\tau)\right]  \right)  =0.
\]
Write
\[
f\sim_{k}g
\]
if%
\[
\lim_{\left\vert z\right\vert \rightarrow\infty}\left[  e^{-ikz}\left(
f-g\right)  \right]  =0
\]
Noting that $\mu-1$ and its derivatives in $z$ and $\overline{z}$ vanish as
$\left\vert z\right\vert \rightarrow\infty$, a simple calculation shows that%
\begin{align*}
\dot{\varphi}-A\varphi &  \sim_{k}e^{iS}\left(  -ik^{3}\mu-\left(
\partial+ik\right)  ^{3}\mu\right) \\
&  \sim_{k}0
\end{align*}
Hence $\dot{\varphi}=A\varphi$ from which (\ref{eq:psidot}) follows.
\end{proof}

Hence: 

\begin{lemma}
Suppose that $q(z,\tau)$ is a solution of the NV\ equation where, for each $\tau$,
$q(z,\tau)$ is a smooth function of conductivity type. Let $\mu(z,k,\tau)$ be the
corresponding NCGO solution with%
\begin{equation}
\mu(z,k,\tau)\sim 1+\frac{1}{4\pi ikz}s(k,\tau)-\frac{e_{-k}(z)}{4\pi i\overline
{k}\overline{z}}\mathbf{t}\left(  k,\tau \right)  +\mathcal{O}\left(  \left\vert
z\right\vert ^{-2}\right)  . \label{eq:mu.asy.bis}%
\end{equation}
Then%
\begin{align}
\dot{s}(k,\tau)  &  =0,\label{eq:sdot.bis}\\
\mathbf{\dot{t}}\left(  k, \tau \right)   &  =i\left(  k^{3}+\overline{k}%
^{3}\right)  \mathbf{t}\left(  k, \tau  \right). \label{eq:tdot.bis}%
\end{align}

\end{lemma}

\begin{proof}
Substituting the asymptotic relation (\ref{eq:mu.asy.bis}) into
(\ref{eq:mudot}), we may compute, modulo terms of order $z^{-2}$,%
\[
\frac{1}{\pi ikz}\dot{s}-\frac{e_{-k}(z)}{\pi i\overline{k}%
\overline{z}}\mathbf{\dot{t}}=-\frac{e_{-k}\left(  z\right)  }{\pi i\overline
{k}\overline{z}}\left(  ik^{3}+i\overline{k}^{3}\right)  \mathbf{t}.%
\]
The computation uses the following facts. If
\[
\mu(z,k,\tau)=1+\frac{a_0}{z}+e_{-k}\frac{b_0}{\overline{z}}+\mathcal{O}\left(
\left\vert z\right\vert ^{-2}\right)
\]
then (\textquotedblleft$\sim$\textquotedblright\ means \textquotedblleft is
asymptotic as $\left\vert z\right\vert \rightarrow\infty$ to\textquotedblright%
modulo $\mathcal{O}\left(|z|^{-2}\right)$)
\begin{align*}
\partial^{3}\mu &  \sim e_{-k}\left(  -ik\right)  ^{3}\frac{b_0}{\overline{z}%
},\\
\overline{\partial}^{3}\mu &  \sim e_{-k}\left(  -i\overline{k}\right)
^{3}\frac{b_0}{\overline{z}},\\
3ik\partial^{2}\mu-3k^{2}\partial\mu &  \sim0
\end{align*}
together with the fact that $u$ defined by $\overline{\partial}u=\partial q$ satisfies $u=\mathcal{O}\left(  \left\vert z\right\vert
^{-2}\right)  $. The identities (\ref{eq:sdot.bis}) and (\ref{eq:tdot.bis})
are immediate.
\end{proof}

\subsection{Solution by Inverse Scattering}

Motivated by the computations of the preceding subsection, we now consider the
problem%
\begin{align}
\overline{\partial}_{k}\mu &  =e^{i\tau S}\mathbf{t}^{\sharp}\mu
\label{eq:mu.dbar.k}\\
\mu(z,~\cdot~,\tau)-1  &  \in L^{p}\left(  \mathbb{C}\right) \nonumber
\end{align}
for a function $\mu(z,k,\tau)$ and the putative reconstruction%
\begin{equation}
q(z,\tau)=\frac{4i}{\pi}\overline{\partial}_{z}\left(  \int_{\mathbb{C}}%
e^{i\tau S}\mathbf{t}^{\sharp}(k)\overline{\mu(z,k,\tau)}~dk\right)  .
\label{eq:q.recon}%
\end{equation}
Here%
\begin{equation}
S(z,k,\tau)=-\frac{1}{\tau}\left(  kz+\overline{k}\overline{z}\right)  +\left(
k^{3}+\overline{k}^{3}\right)   \label{eq:S}%
\end{equation}
and $\mathbf{t}^{\sharp}$ is obtained from the Cauchy data $q_{0}$. We will
show that $q(z,\tau)$ solves the NV equation by deriving an equation of motion
for $\mu(z,k,\tau)$ and using this equation to compute $q_{\tau}$ if $q$ is given by
(\ref{eq:q.recon}) and $\mu$ is the unique solution of (\ref{eq:mu.dbar.k})

First, we establish an equation of motion for the solution $\mu$ of
(\ref{eq:mu.dbar.k}). Although this equation is the same equation as
(\ref{eq:mudot}) for the solution of the direct problem, our starting point
here is (\ref{eq:mu.dbar.k}).

\begin{lemma}
\label{lemma:mu.dot}Suppose that $\mathbf{t}^{\sharp}\in\mathcal{S}\left(
\mathbb{C}\right)  $ and $\mu$ solves (\ref{eq:mu.dbar.k}). For each $\tau$, define $q(z,\tau)$ by \eqref{eq:q.recon} and define $u(z,\tau)$ by $u=\overline{\partial}^{-1} \partial q$. Then%
\[
\dot{\mu}=ik^{3}\mu+\left(  \partial+ik\right)  ^{3}\mu+\overline{\partial
}^{3}\mu-\frac{3}{4}u\left(  \partial+ik\right)  \mu-\frac{3}{4}\overline
{u}\overline{\partial}\mu
\]
where $\partial$ and $\overline{\partial}$ denote differentiation with respect
to the $z$ and $\overline{z}$ variables.
\end{lemma}

\begin{proof}
Let%
\[
w=\dot{\mu}-\left(  ik^{3}\mu+\left(  \partial+ik\right)  ^{3}\mu
+\overline{\partial}^{3}\mu-\frac{3}{4}u\left(  \partial+ik\right)  \mu
-\frac{3}{4}\overline{u}\overline{\partial}\mu\right)  .
\]
We will show that $w=0$ in two steps. First, we show that%
\[
\overline{\partial}_{k}w=e^{i\tau S}\mathbf{t}^\sharp \overline{w}.
\]
This is an easy consequence of the formulas%
\begin{align*}
\overline{\partial}_{k}\left(  \partial+ik\right)  \mu &  =e^{i\tau S
}\mathbf{t}^\sharp \overline{\overline{\partial}\mu},\\
\overline{\partial}_{k}\left(  \overline{\partial}\mu\right)   &
=e^{i\tau S}\mathbf{t}^\sharp\overline{\left(  \partial+ik\right)  \mu}%
\end{align*}
and holds for any smooth function $u$.

Next, we show that for any fixed values of the parameters $\tau$ and $z$,
\[
\lim_{\left\vert k\right\vert \rightarrow\infty}w(z,k,\tau)=0.
\]
Here we must choose $u=\overline{\partial}^{-1}\partial q$ in order for the
assertion to be correct. Owing to Lemma \ref{lemma:mu.exp.k} and the formula
\[
w=\dot{\mu}-\left(  \partial^{3}\mu+\overline{\partial}^{3}\mu+3ik\partial
^{2}\mu-3k^{2}\partial\mu-\frac{3}{4}u\left(  \partial+ik\right)  \mu-\frac
{3}{4}\overline{u}\overline{\partial}\mu\right)  ,
\]
we have
\begin{align*}
w  &  =-3ik\partial^{2}\mu+3k^{2}\partial\mu+\frac{3}{4}u\left(
\partial+ik\right)  \mu+\mathcal{O}\left(  k^{-1}\right) \\
&  =A_{-1}k+A_{0}+\mathcal{O}\left(  k^{-1}\right)
\end{align*}
where%
\[
A_{-1}=3\left(  \partial a_{0}+\frac{i}{4}u\right)
\]
and%
\[
A_{0}=3\left[  -i\partial^{2}a_{0}+\partial a_{1}+\frac{i}{4}ua_{0}\right]  .
\]
The condition $A_{-1}=0$ forces the choice $u=\overline{\partial}^{-1}\partial
q$. We may then compute%
\[
A_{0}=\frac{3}{16}\left[  \partial\overline{\partial}^{-1}\left(
q\overline{\partial}^{-1}q\right)  -\left(  \overline{\partial}^{-1}q\right)
\cdot\left(  \partial\overline{\partial}^{-1}q\right)  \right]  .
\]
One the one hand, $A_{0}$ vanishes as $\left\vert z\right\vert \rightarrow
\infty$ for each fixed $\tau$ by the decay of $q$. On the other hand, a
straightforward computation shows that $\overline{\partial}A_{0}=0$. It now
follows from Liouville's Theorem that $A_{0}=0$, and hence $w=O\left(
k^{-1}\right)  $. We now used the generalized Liouville Theorem to conclude
that $w=0$.
\end{proof}

Finally, we prove:

\begin{proposition}
Suppose that $\mathbf{t}^{\sharp}\in\mathcal{S}\left(  \mathbb{R}^{2}\right)  $.
Then, the formula%
\[
q(z,\tau)=\frac{4i}{\pi}\overline{\partial}_{z}\left(  \int e^{itS(z,k,\tau)}%
\mathbf{t}^{\sharp}(k)\overline{\mu(z,k,\tau)}~dm(k)\right)
\]
yields a classical solution of the NV equation.
\end{proposition}

\begin{proof}
In what follows, we will freely use the commutation relations%
\begin{align}
\partial e^{i\tau S}  &  =e^{i\tau S}\left(  \partial-ik\right)
,\label{eq:NVcom1}\\
\overline{\partial}e^{i\tau S}  &  =e^{i\tau S}\left(  \overline{\partial
}-i\overline{k}\right)  \label{eq:NV.com2}%
\end{align}
and the equation
\begin{equation}
\overline{\partial}\left(  \partial+ik\right)  \mu=\frac{1}{4}q\mu.
\label{eq:NV.mu3}%
\end{equation}
For notational brevity we'll write $c=4i/\pi$. We compute%
\begin{align*}
\dot{q}  &  =
c\overline{\partial}
	\left(  \int e^{i\tau S}\mathbf{t}^\sharp\left\{
		ik^{3}+i\overline{k}^{3}\right\}  \mu\right) \\
		&  +c\overline{\partial}\left(  \int e^{i\tau S} \mathbf{t}^\sharp
			\left\{  -i\overline{k}%
		^{3}+\left(  \overline{\partial}-i\overline{k}\right)  ^{3}+\partial^{3}%
			-\frac{3}{4}\overline{u}\left(  \overline{\partial}-i\overline{k}\right)
			-\frac{3}{4}u\partial\right\}  \overline{\mu}\right)
\end{align*}
where in the second term we used Lemma \ref{lemma:mu.dot}. Using the
commutation relations above to move differential operators to the left of
$\exp\left(  itS\right)  $, we conclude that%
\[
\dot{q}=\partial^{3}q+\overline{\partial}^{3}q-\frac{3}{4}\overline{\partial
}\left(  u\partial\overline{\partial}^{-1}q\right)  -\frac{3}{4}%
\overline{\partial}\left(  \overline{u}q\right)  +I
\]
where%
\[
I=c\overline{\partial}\left(  \int\left\{  3ik\partial^{2}-3k^{2}%
\partial-\frac{3}{4}iku\right\}  e^{i\tau S}\mathbf{t}^\sharp \overline{\mu}\right)  .
\]
We claim that
\begin{equation}
I=\frac{3}{4}\left\{  \overline{\partial}\left(  u\partial\overline{\partial
}^{-1}q\right)  -\partial\left(  uq\right)  \right\}  . \label{eq:NV.I}%
\end{equation}
If so, then $q$ solves the NV\ equation as claimed. To compute $I$, write
$I=I_{1}-I_{2}$ where%
\begin{align*}
I_{1}  &  =c\overline{\partial}\left(  \int\left\{  3ik\partial^{2}%
-3k^{2}\partial\right\}  e^{i\tau S}\mathbf{t}^\sharp\overline{\mu}\right)  ,\\
I_{2}  &  =c\overline{\partial}\left(  \int e^{i\tau S}\mathbf{t}^\sharp\left\{  \frac{3}%
{4}iku\right\}  \overline{\mu}\right)  .
\end{align*}
Using (\ref{eq:NV.mu3}) and (\ref{eq:NV.com2}) we may write%
\begin{align*}
I_{1}  &  =\frac{3}{4}c\partial\left(  \int e^{i\tau S}\left(  ik\right)
\mathbf{t}^{\sharp} q\overline{\mu}\right) \\
&  =\frac{3}{4}c\partial\left(  q\int\left(  -\partial e^{i\tau S}\right)
\mathbf{t}^{\sharp}\overline{\mu}\right) \\
&  =-\frac{3}{4}c\partial\left(  q\partial\left(  \int e^{i\tau S}%
\mathbf{t}^{\sharp}\overline{\mu}\right)  -q\int e^{i\tau S}\mathbf{t}^{\sharp}\partial\overline{\mu}\right) \\
&  =-\frac{3}{4}\partial\left(  q\partial\overline{\partial}^{-1}q\right)
+\frac{3}{4}c\partial\left(  q\overline{\partial}^{-1}\overline{\partial
}\left(  \int e^{i\tau S}\mathbf{t}^{\sharp}\partial\overline{\mu}\right)  \right)
\end{align*}
where in the third line we used $u\partial v=\partial\left(  uv\right)
-v\partial u$. \ In the second term on the fourth line, we may use%
\begin{equation}
\overline{\partial}e^{i\tau S}\partial\overline{\mu}=e^{i\tau S}%
\partial\left(  \overline{\partial}-i\overline{k}\right)  \overline{\mu}
\label{eq:NV.com3}%
\end{equation}
and (\ref{eq:NV.mu3}) to conclude that%
\begin{align*}
\frac{3}{4}c\partial\left(  q\overline{\partial}^{-1}\overline{\partial
}\left(  \int e^{i\tau S}\mathbf{t}^{\sharp}\partial\overline{\mu}\right)  \right)   &
=\frac{3}{16}c\partial\left(  q\overline{\partial}^{-1}q\left(  \int
e^{i\tau S}\mathbf{t}^{\sharp}\overline{\mu}\right)  \right) \\
&  =\frac{3}{16}\partial\left(  q\overline{\partial}^{-1}\left(
q\overline{\partial}^{-1}q\right)  \right)
\end{align*}
so that%
\begin{align*}
I_{1}  &  =-\frac{3}{4}\partial\left(  q\partial\overline{\partial}%
^{-1}q\right)  +\frac{3}{16}\partial\left(  q\overline{\partial}^{-1}\left(
q\overline{\partial}^{-1}q\right)  \right) \\
&  =-\frac{3}{4}\partial\left(  qu\right)  +\frac{3}{16}\partial\left(
q\overline{\partial}^{-1}\left(  q\overline{\partial}^{-1}q\right)  \right)  .
\end{align*}
Similarly, we may compute%
\begin{align*}
I_{2}  &  =\frac{3}{4}c\overline{\partial}\left(  u\int ik~e^{i\tau S
}\mathbf{t}^{\sharp}\overline{\mu}\right) \\
&  =\frac{3}{4}c\overline{\partial}\left(  u\int\left(  -\partial
e^{i\tau S}\right)  ~\mathbf{t}^{\sharp}\overline{\mu}\right) \\
&  =-\frac{3}{4}c\overline{\partial}\left(  u\partial\left(  \int
e^{i\tau S}\mathbf{t}^{\sharp}\overline{\mu}\right)  -u\int e^{i\tau S}\mathbf{t}^{\sharp}\partial
\overline{\mu}\right) \\
&  =-\frac{3}{4}\overline{\partial}\left(  u\overline{\partial}^{-1}\partial
q\right)  +\frac{3}{4}c\overline{\partial}\left(  u\overline{\partial}%
^{-1}\left(  \int \mathbf{t}^{\sharp}\overline{\partial}\left(  e^{i\tau S}\partial
\overline{\mu}\right)  \right)  \right)  .
\end{align*}
Using (\ref{eq:NV.com3}) again we find
\begin{align*}
I_{2}  &  =-\frac{3}{4}\overline{\partial}\left(  u\overline{\partial}%
^{-1}\partial q\right)  +\frac{3}{16}\overline{\partial}\left(  u\overline
{\partial}^{-1}\left(  q\overline{\partial}^{-1}q\right)  \right) \\
&  =-\frac{3}{4}\overline{\partial}\left(  u^{2}\right)  +\frac{3}%
{16}\overline{\partial}\left(  u\overline{\partial}^{1}\left(  q\overline
{\partial}^{-1}q\right)  \right)
\end{align*}
where we used (\ref{eq:NV.mu3}) in the first line, and in the second line we
used $u=\overline{\partial}^{-1}\partial q$. Hence
\begin{align*}
I_{1}-I_{2}  &  =-\frac{3}{4}\partial\left(  qu\right)  +\frac{3}{4}%
\overline{\partial}\left(  u^{2}\right) \\
&  +\frac{3}{16}\partial\left(  q\overline{\partial}^{-1}\left(
q\overline{\partial}^{-1}q\right)  \right)  -\frac{3}{16}\overline{\partial
}\left(  u\overline{\partial}^{1}\left(  q\overline{\partial}^{-1}q\right)
\right)  .
\end{align*}
Since $\partial q=\overline{\partial}u$ and $\overline{\partial}^{-1}\left(
q\overline{\partial}^{-1}q\right)  =\frac{1}{2}\left(  \overline{\partial
}^{-1}q\right)  ^{2}$ we can conclude that the second line is zero and
(\ref{eq:NV.I}) holds. The conclusion now follows.
\end{proof}

\section{Special Solutions}  \label{sec:Special}

There are various powerful methods to find solutions of nonlinear
evolution equations, most notably the inverse scattering method.  However, the inverse scattering method is not readily useful for finding closed-form
solutions to the NV equation, and so techniques including  Hirota's method and the extended mapping approach (EMA)  are presented here to construct closed-form solutions of several types of solitons.   We begin by explaining the close connection between plane-wave solutions to NV and solutions to the KdV equation and present evolutions of KdV ring-type solutions.  Although the KdV ring-type solition is not of conductivity type, the scattering transform is computed in Section \ref{subsec:tk_ring_sol}, and the numerical results provide evidence of the presence of an exceptional circle.

\vspace{1em}

\subsection{KdV-type Solutions} \label{subsec:KdVsols}

Consider the NV equation ~\eqref{eq:NV.Cauchy} in the form
%\begin{equation}\label{eq:NV_num} 
%\dot q = -\frac{1}{4}\frac{\partial^3}{\partial x^3}\,q +\frac{3}{4}\frac{\partial^3}{\partial x\partial y^2}\,q + \frac{3}{4}\,\operatorname{div} ((q-E)\mathbf{u}),
%\end{equation}
\begin{equation}\label{eq:NV_num} 
\dot q = -\frac{1}{4}q_{xxx} +\frac{3}{4}q_{xyy} + \frac{3}{4}\,\operatorname{div} ((q-E)\mathbf{u}),
\end{equation}
where $u=u_1+iu_2$ and $\mathbf{u} = (u_1,u_2)$, and the auxiliary equation $\overline{\partial} u=\partial q$ as
\begin{equation}
  \label{eq:dbar_num}
\left\{ 
    \begin{array}{rcl} 
    (u_1)_x - (u_2)_y &=& +q_x\\
    (u_2)_x +(u_1)_y &=& -q_y
  \end{array}
  \right.
\end{equation}
As in \cite{CrokeMuellerStahel13} we use a FFT-based method to solve the
equations on the square $-L\leq x,y\leq +L$
with periodic boundary conditions. 

To examine the linear contributions we introduce a function
$$q(t,x,y)=\exp(i\,(\xi\,x+\eta\,y)).$$  Then the
$\overline{\partial}$ equation~\eqref{eq:dbar_num} is solved by
\begin{eqnarray*}
   u_1(t,x,y) &=& \frac{\xi^2-\eta^2}{\xi^2+\eta^2}\;\exp(i\,(\xi\,x+\eta\,y))\\
   u_2(t,x,y) &=& \frac{-2\,\xi\,\eta}{\xi^2+\eta^2}\;\exp(i\,(\xi\,x+\eta\,y)),
\end{eqnarray*}
and thus the linear part of the NV equation~\eqref{eq:NV_num}
\[\dot q
  = -\frac{1}{4}\frac{\partial^3}{\partial x^3}\,q
    +\frac{3}{4}\frac{\partial^3}{\partial x\partial y^2}\,q
 + E\,\frac{3}{4} \nabla \cdot \mathbf{u} \]
is transformed into a elementary linear ODE for the Fourier coefficient $c(t)$
\[  4\,\frac{d}{dt}\, c(t) = i\,\left( \xi^3-3\,\xi\,\eta^2\right)\,
\left(1-\frac{3\,E}{\xi^2+\eta^2}\right)\;c(t).\]
Assuming a Fourier approximation of the solutions
\[ q(t,x,y) = \sum_{j,k=0}^{N-1} c_{j,k}(t)\;\exp(i \pi\,(k\,x+j\,y)/L)\] 
this leads to a coupled system of ODEs for the Fourier coefficients~$c_{j,k}(t)$.
We use a Crank--Nicolson scheme for the linear part of NV and an
explicit method for the nonlinear contribution $\operatorname{div}( q\,\mathbf{u})$.
For details see \cite{CrokeMuellerStahel13, RyansThesis}.
\bigskip

There is a close connection between plane wave solutions to NV and solutions to KdV (see \cite{CrokeMuellerStahel13}):

\begin{remark}
\label{rem:NVtoKdV}
Assume the solutions to NV are planar waves
\begin{eqnarray*}
   q(t,s) &=& q(t,x,y) = q(t,n_1s,n_2s) \\
   u_i(t,s) &=& u_i(t,x,y) = u_i(t,n_1s,n_2s) 
\end{eqnarray*}
for some direction vector $\vec n = (n_1,n_2) = (\cos(\alpha),\sin(\alpha))$.
Then the bounded solutions to the $\overline{\partial}$
equation~\eqref{eq:dbar_num} are given by 
\begin{eqnarray*}
 u_1(t,s) &=& +(n_1^2 -n_2^2)\; q(t,s)+c_1\\
 u_2(t,s) &=& -(2\,n_1\,n_2)\;q(t,s)+c_2  
\end{eqnarray*}
for arbitrary constants $c_1, c_2$, and the NV equation~\eqref{eq:NV_num}
reduces to an equation similar to the KdV equation,
\begin{equation}
  \label{eq:NVKdV}
  \frac{4}{\kappa}\, q_t =  - q^{\prime\prime\prime} + 6\,q\,q^{\prime}
  +\frac{3\,\beta}{\kappa}\,q^{\prime} 
\end{equation}
with $\kappa = \cos(3\,\alpha)$ and $\beta = -\kappa\,E+c_1\,n_1+c_2\,n_2$.
If $v(t,s)$ denotes a solution to the standard KdV equation
\begin{equation}
  \label{eq:KdV}
\dot v(t,x) = - v^{\prime\prime\prime}(t,x) + 6\, v(t,x)\,v^{\prime}(t,x),
\end{equation}
we obtain explicit solutions to NV by
\[ q(t,s) = v\left(\frac{\kappa}{4}\, t\,,\,s+k_1\,t\right) - k_2
  = v\left(\frac{\kappa}{4}\, t\,,\,s+\frac{3}{4}\,(c_1\,n_1 + c_2\,n_2)\,t\right)
   + \frac{E}{2}.\]
For the special case $c_1=c_2=0$ we find
\[ q(t,n_1s,n_2s)=q(t,s) = v\left(\frac{\kappa}{4}\, t\,,\,s\right) + \frac{E}{2}\]
Thus we can relate all solutions to KdV as planar solutions to NV,
with different speeds of propagation depending on the direction of the plane wave.
%% Add per AS
An example of this are the solutions generated by Hirota's method, see section~\ref{subsub:Hirota}.
\end{remark}

\begin{example}
Based on the above remark we have an exact solution of the NV equation
at zero energy given by
\[ q(t,x,y) = -2\,c\,\cosh^{-2}(\sqrt{c}(x-c\,t)).\]
This solution is unstable with respect to perturbations periodic in
$y$~direction with period $\frac{2\,\pi}{k\,\sqrt{c}}$ for
$0.363<k<1$, see \cite{CrokeMuellerStahel13}. This is confirmed by numerical
evolution of the NV equation using the above spectral method.
\end{example}

\begin{example} \label{ex:3}  A KdV ring initial condition.\\
Using $r=\sqrt{x^2+y^2}$ we choose a radially symmetric initial value
\[ q_0(x,y)=q(0,x,y) = f(r)=-\frac{1}{2}\,\cosh^{-2}(\frac{1}{2}\,(r-20))<0\;.\]
This corresponds to  a solution of equation \eqref{eq:NVKdV} in the
radial variable $r$, a KdV ring with radius~20.

%% Changed wording here per JM
Observing that $0$ is a potential of conductivity type, we may use the argument in~\cite[Appendix B]{MPS:2013} (based on~\cite{Murata:1986})
to conclude that $q_0$, as a non-positive deviation from~0, is subcritical and
consequently not of conductivity type. In~\cite{Murata:1986} Murata
classifies general solutions of the Schr\"{o}dinger equation $\Delta  u=q\,u$,
For our special case we have an elementary proof for the required result.
For $q(0,x,y)$ to be of conductivity type we would need a positive, radially
symmetric function $u$ such that $ \Delta u = q\,u$ or in radial coordinates
$ (r\,u^\pr(r))^\pr = r\,f(r)\,u(r)$.
Since the function is radially symmetric we use $u^\pr(0)=0$ and an
integration leads to
\[  r\,u^\pr(r) = 0+\int_0^r s\,f(s)\,u(s)\; ds <0\]
Thus $u^\pr(r)$ is negative and $r\,u^\pr(r)$ is
decreasing. Consequently we have a constant $C=-r_0\,u^\pr(r_0)>0$ and
for all $r\geq r_0$ we conclude $u^\pr(r)\leq -C/r$. This implies
\begin{eqnarray*}
%  r\,u^\pr(r)&\leq& -C\\
%  u^\pr(r)&\leq& -C/r \\
  u(r) &=& u(r_0) + \int_{r_0}^r u^\pr(s)\; ds
  \leq u(r_0) - \int_{r_0}^s \frac{C}{s}\; ds
  = u(r_0) -C\,(\ln(r) -\ln(r_0))
\end{eqnarray*}
and for $r$ large enough this is in contradition to $u(r)>0$ and thus
$q_0(x,y)$ is not of conductivity type.

We solve the NV equation at zero energy $E=0$ with the above initial
condition. Based on the speed profile with the angularly dependent
speed factor $\kappa=\cos(3\,\alpha)$ from Remark~\ref{rem:NVtoKdV},
one expects that the initially circular shape will be deformed and its
shape will be more triangular at later times. This is confirmed by
Figure~\ref{fig:EvolutionZeroEnergy}, which shows graphs of
$-q(t,x,y)$ at different times~$t$.

To examine the possible blowup of the solution at a finite time we ran
the algorithm based on the spectral method on a domain $-50\leq
x,y\leq +50$ with Fourier grids of sizes $1024\times 1024$, $2048\times
2048$ and $4096\times 4096$. With time steps $dt=0.01$ and $dt=0.001$
we examined the solution and its $L_2$ norm. In all cases the solution
either blew up at times just beyond $t=38$ or displayed a sudden
occurrence of sizable noise. For a final decision of a blow up time
the exact shape and size of the spikes in Figure~\ref{fig:EvolutionZeroEnergy}
have to be examined carefully.

Letting the KdV ring initial condition evolve for negative times, we
observed the same solutions as for positive times, but rotated by in
the spatial plane by $60^\circ$. We observed blowup at times smaller
than $t=-38$.

\begin{figure}[htbp]
  \begin{center}
    \leavevmode
      \includegraphics[width=6cm]{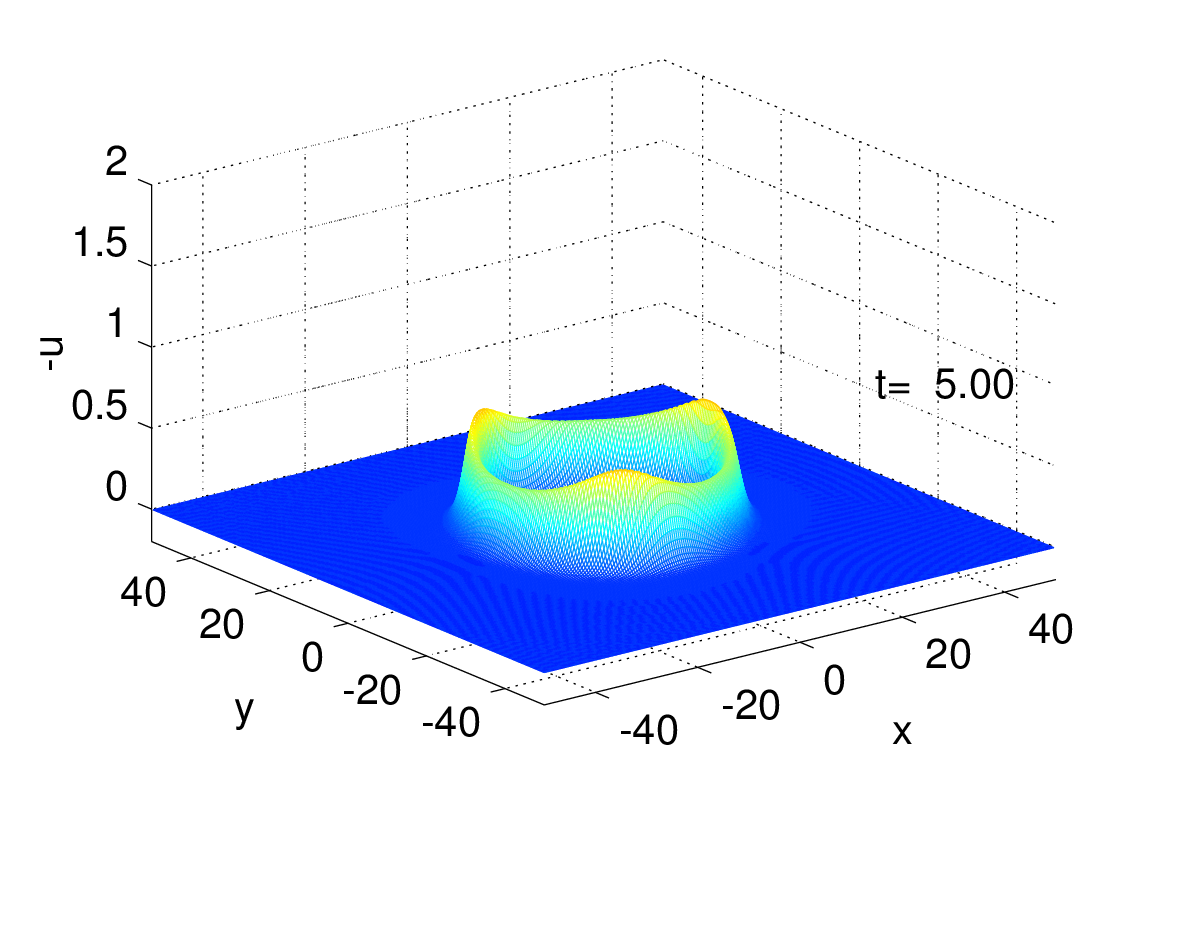}
      \includegraphics[width=6cm]{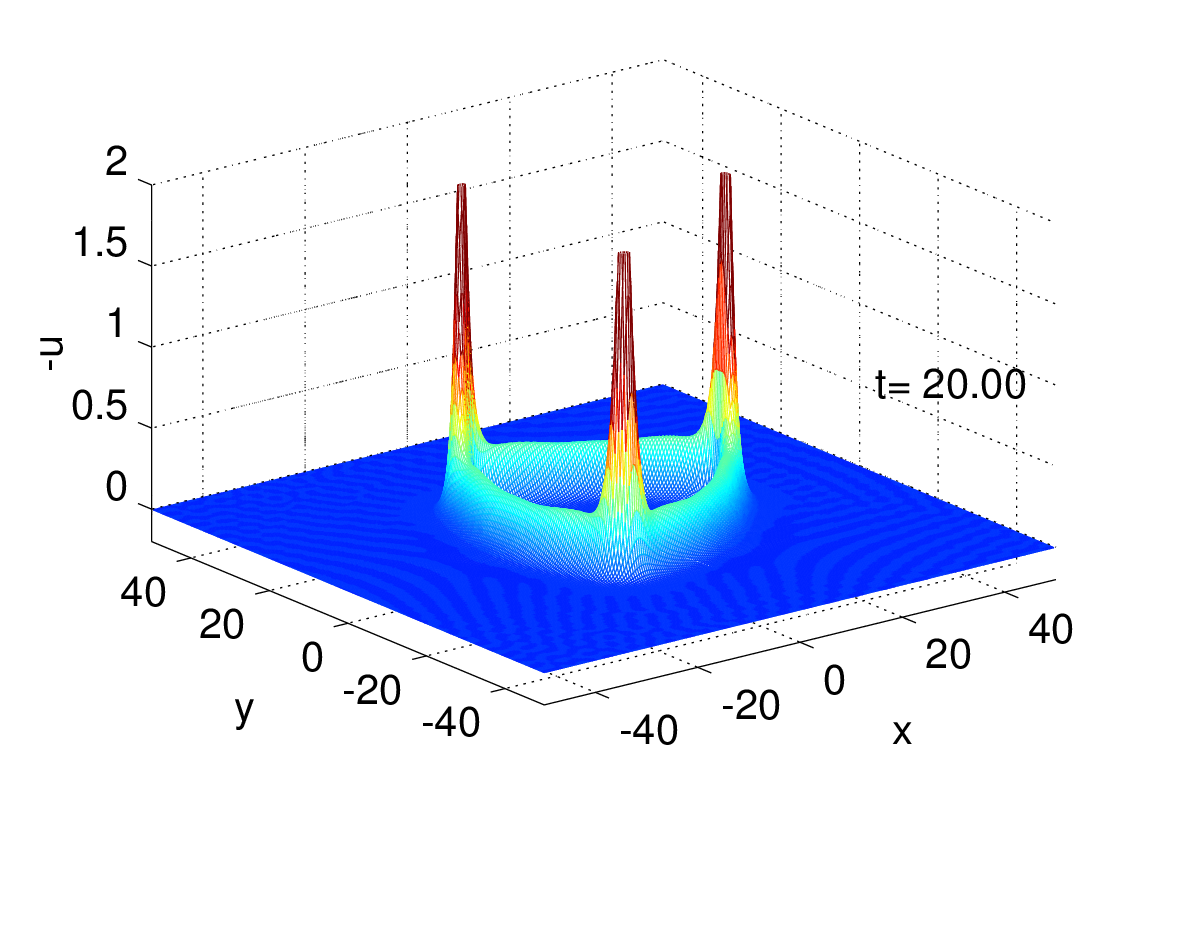}\\
      \includegraphics[width=6cm]{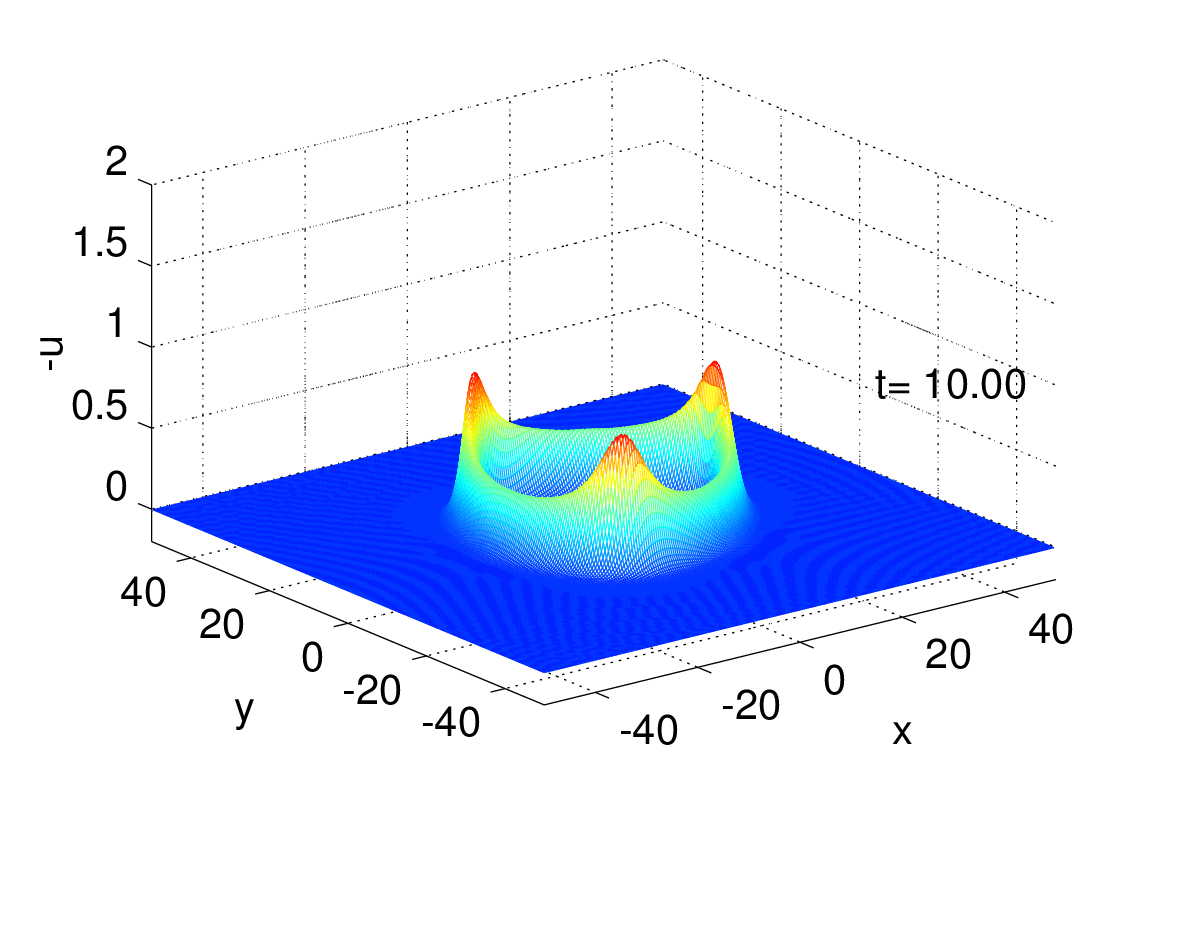}
      \includegraphics[width=6cm]{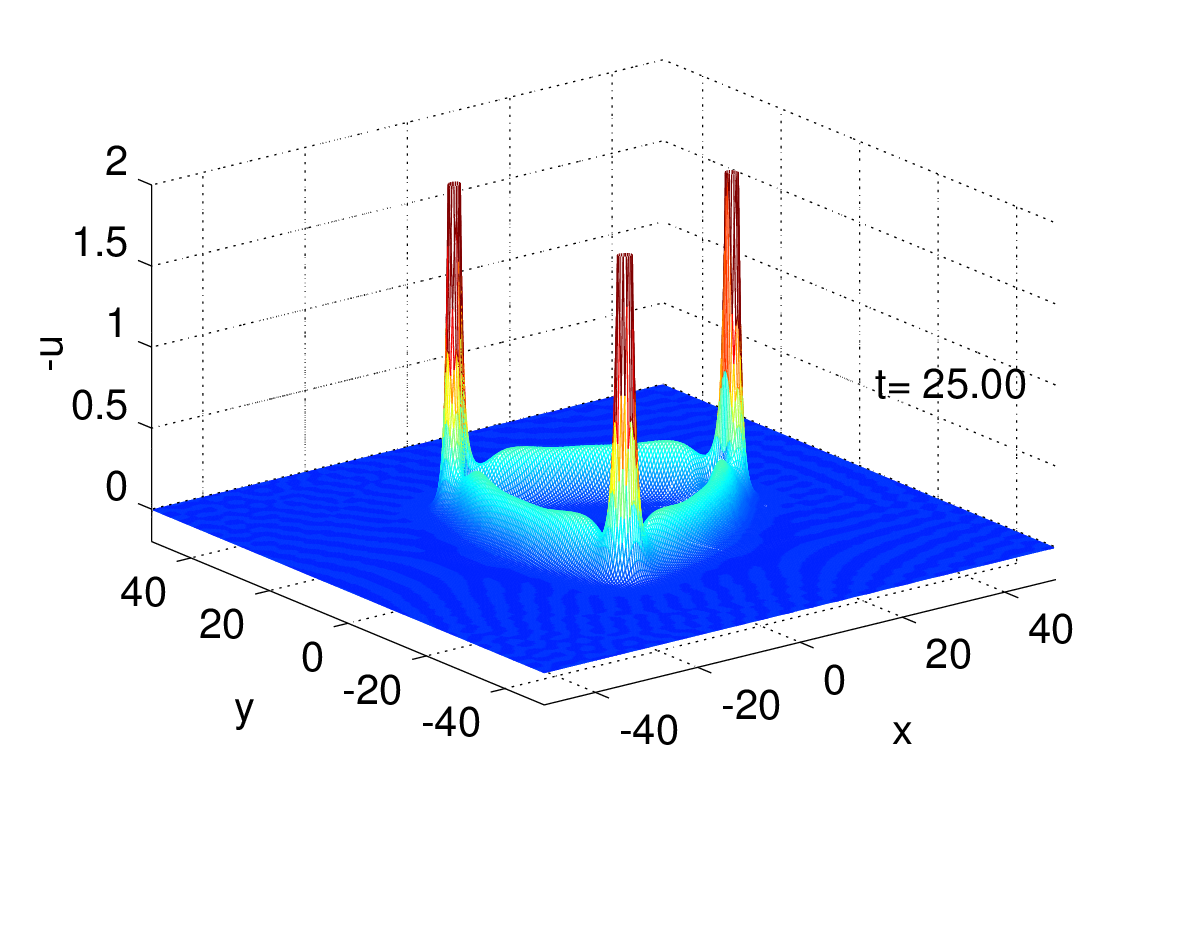}\\
      \includegraphics[width=6cm]{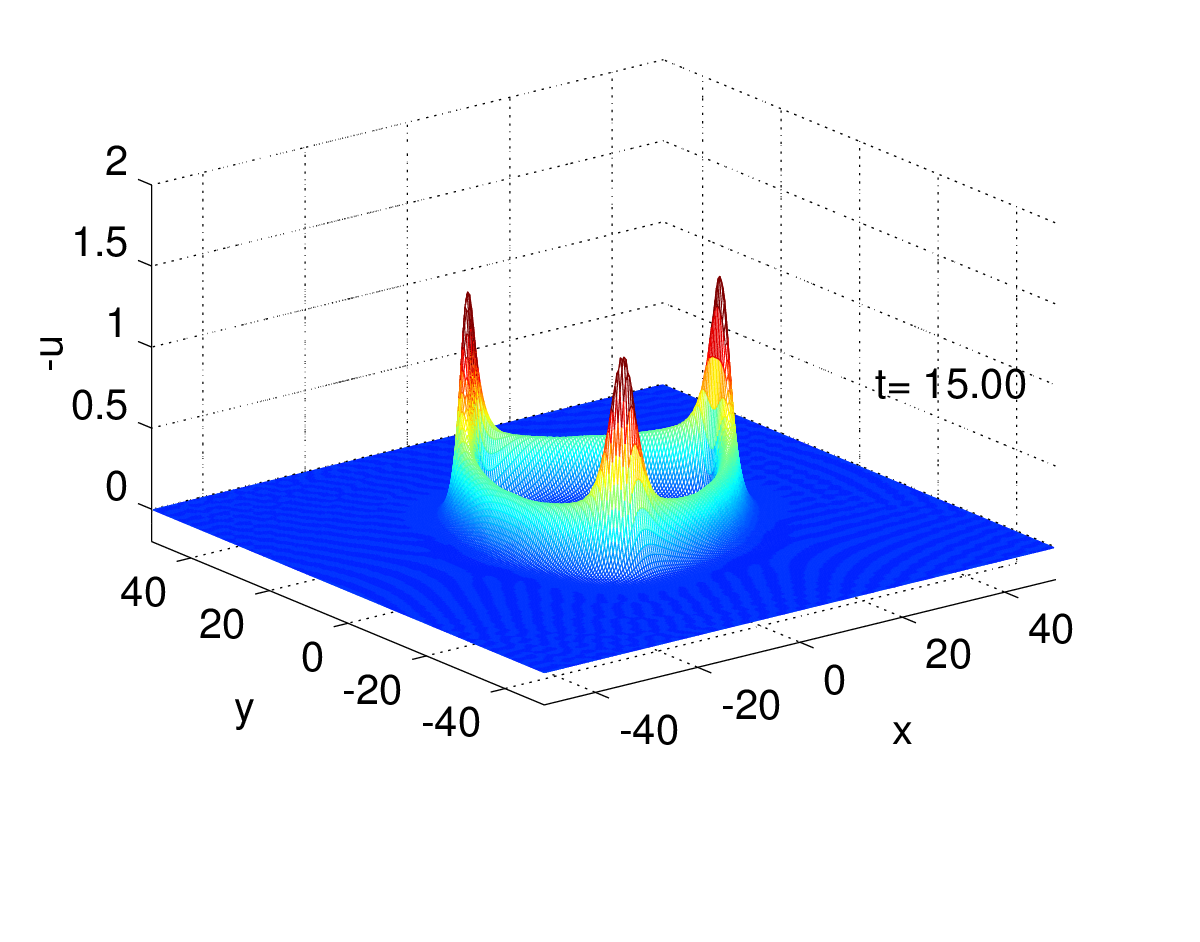}
      \includegraphics[width=6cm]{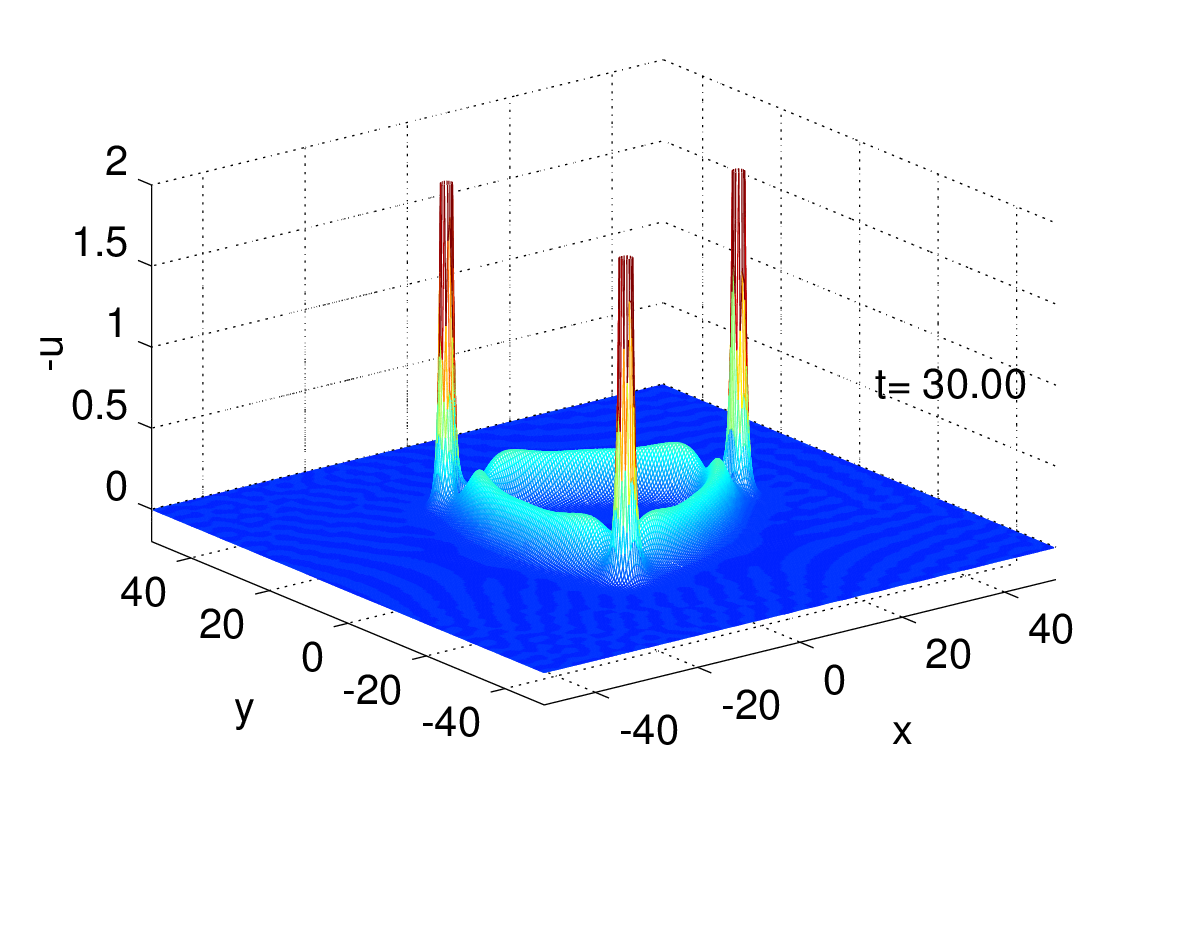}
    \caption{Evolution of a KdV ring by NV at zero energy }
    \label{fig:EvolutionZeroEnergy}
  \end{center}
\end{figure}
\end{example}

\begin{example}
With the initial value of the KdV ring in Example \ref{ex:3}, we evolve the
solution by the NV equation~\eqref{eq:NV_num} and~\eqref{eq:dbar_num}
with a positive energy $E=1/8$. Time snapshots of the
evolution are plotted in Figure~\ref{fig:EvolutionPositiveEnergy}. 
The initial dynamics are comparable to the previous example at zero
energy: three spikes appear and grow rapidly in size, but as time progresses, these spikes decay in amplitude,  and  separate from the previous KdV ring, and a new triple
of spikes appears. The process is repeated.
Observe that the solution exists at least until time $t =100$, also
confirmed by the graph of the $L_2$ norm of the solution as function
of time in Figure~\ref{fig:EvolutionPositiveEnergyNorm}.
\end{example}

\begin{figure}[htbp]
  \begin{center}
    \leavevmode
      \includegraphics[width=4.1cm]{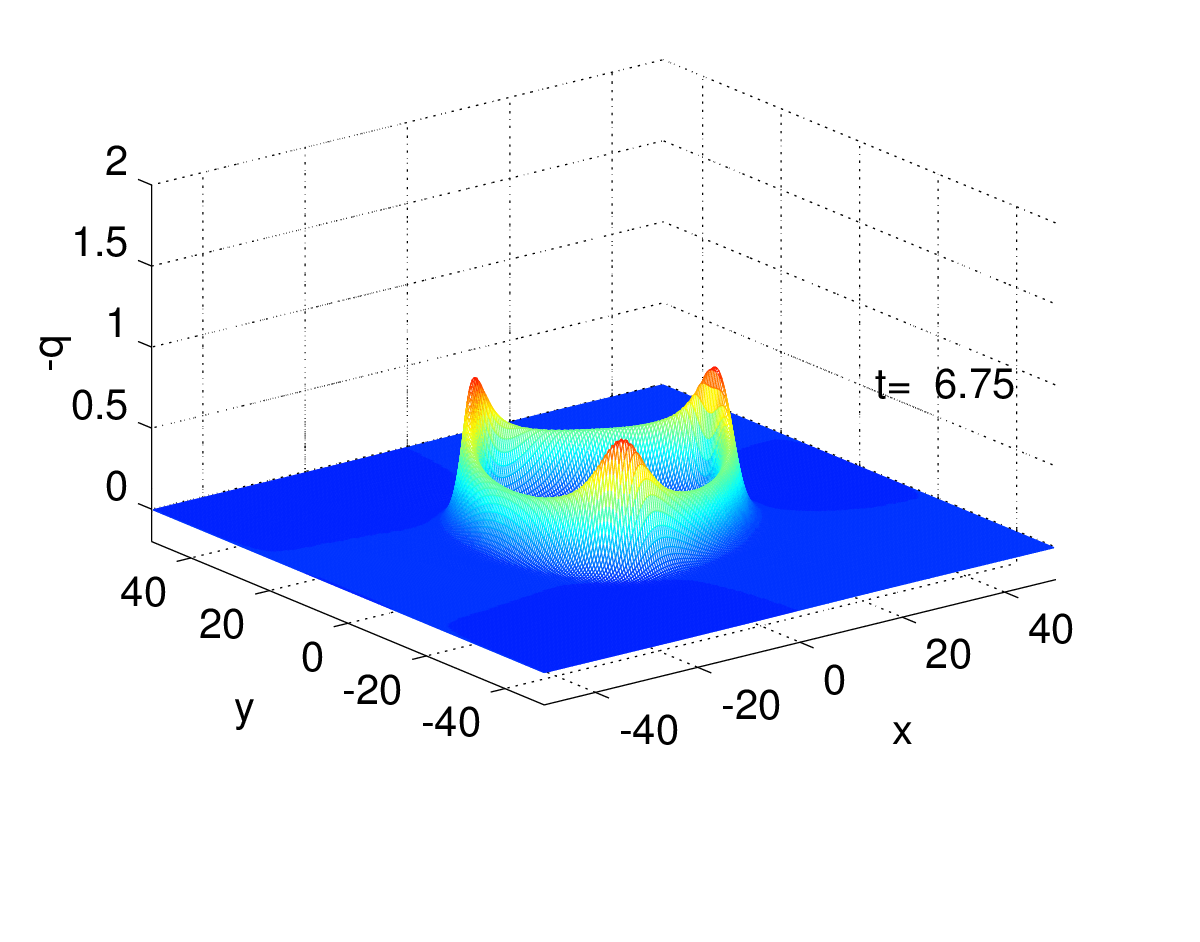}
      \includegraphics[width=4.1cm]{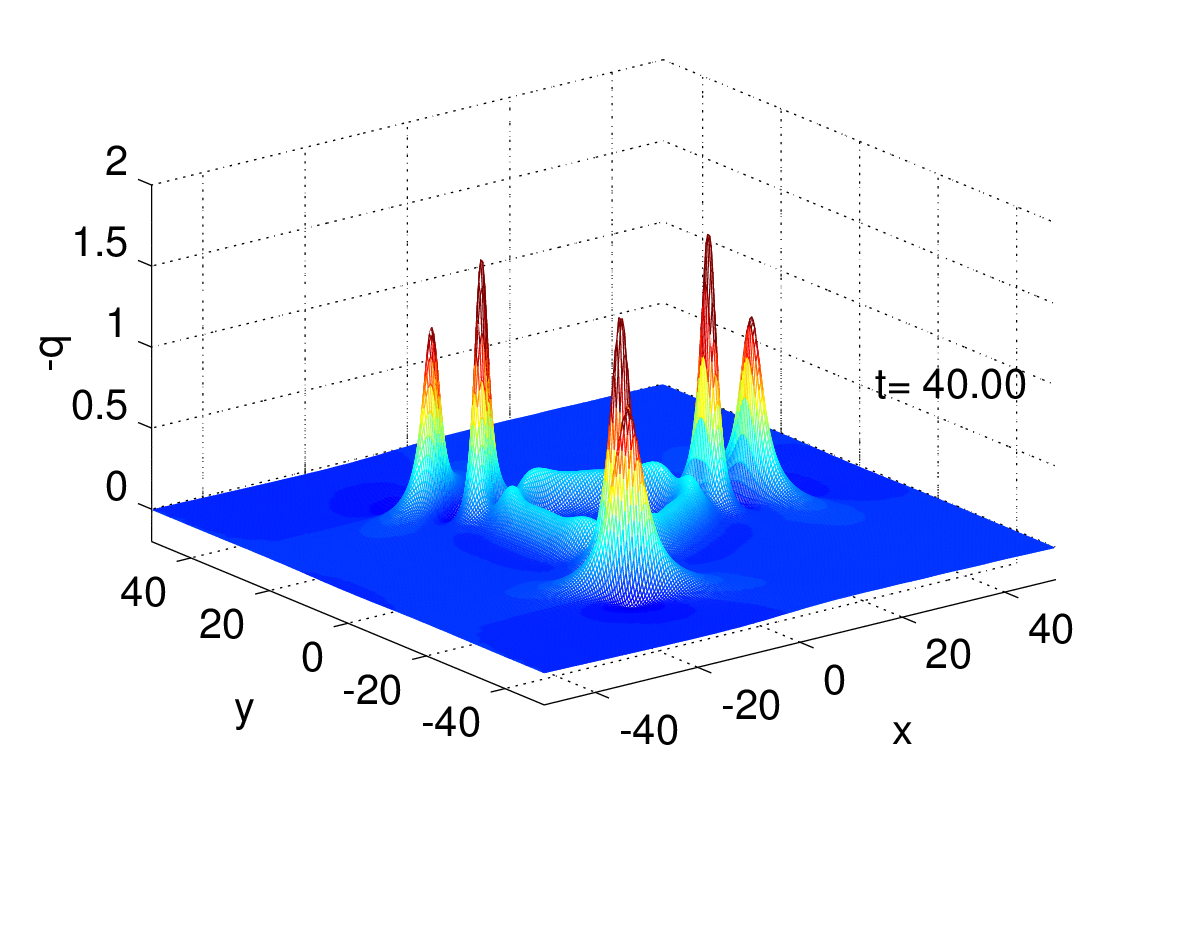}
      \includegraphics[width=4.1cm]{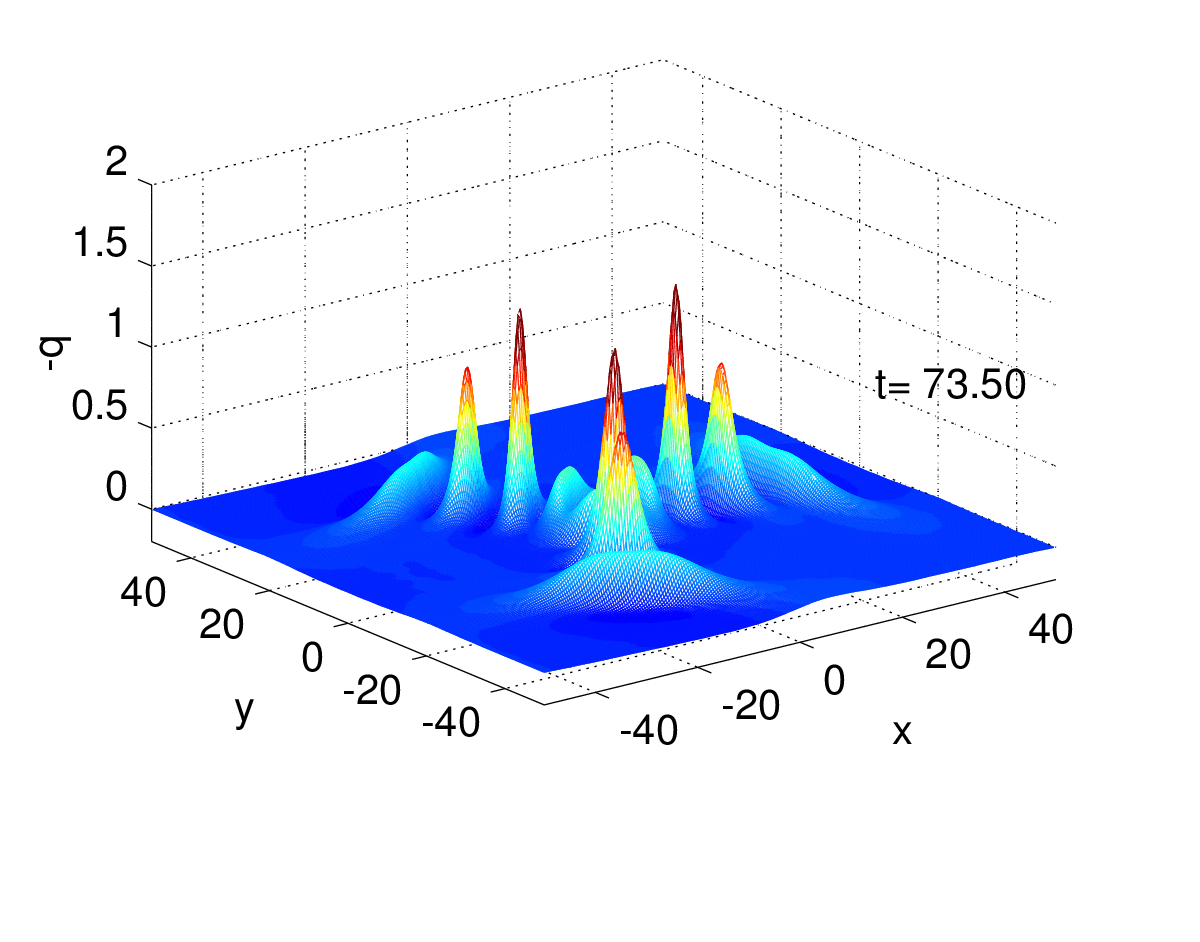}\\
      \includegraphics[width=4.1cm]{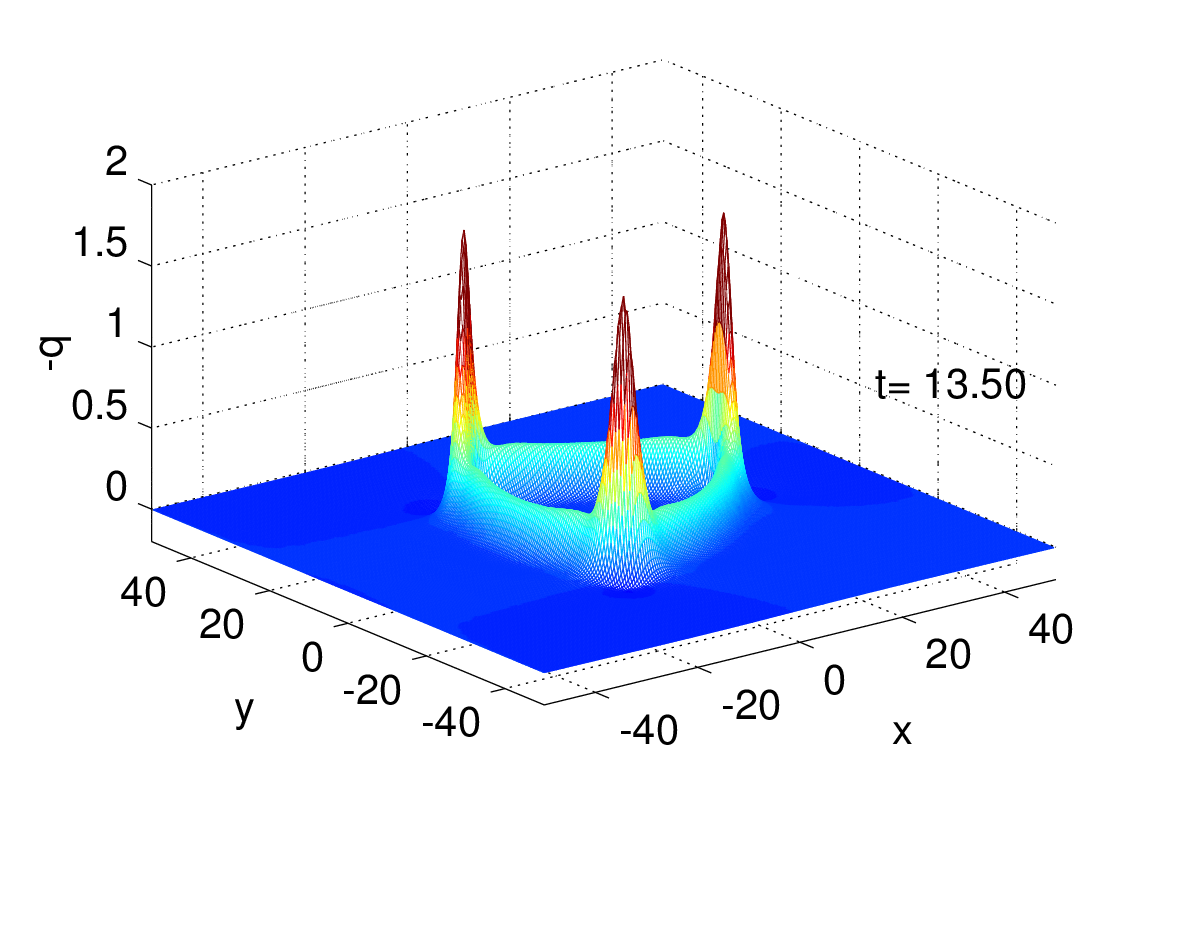}
      \includegraphics[width=4.1cm]{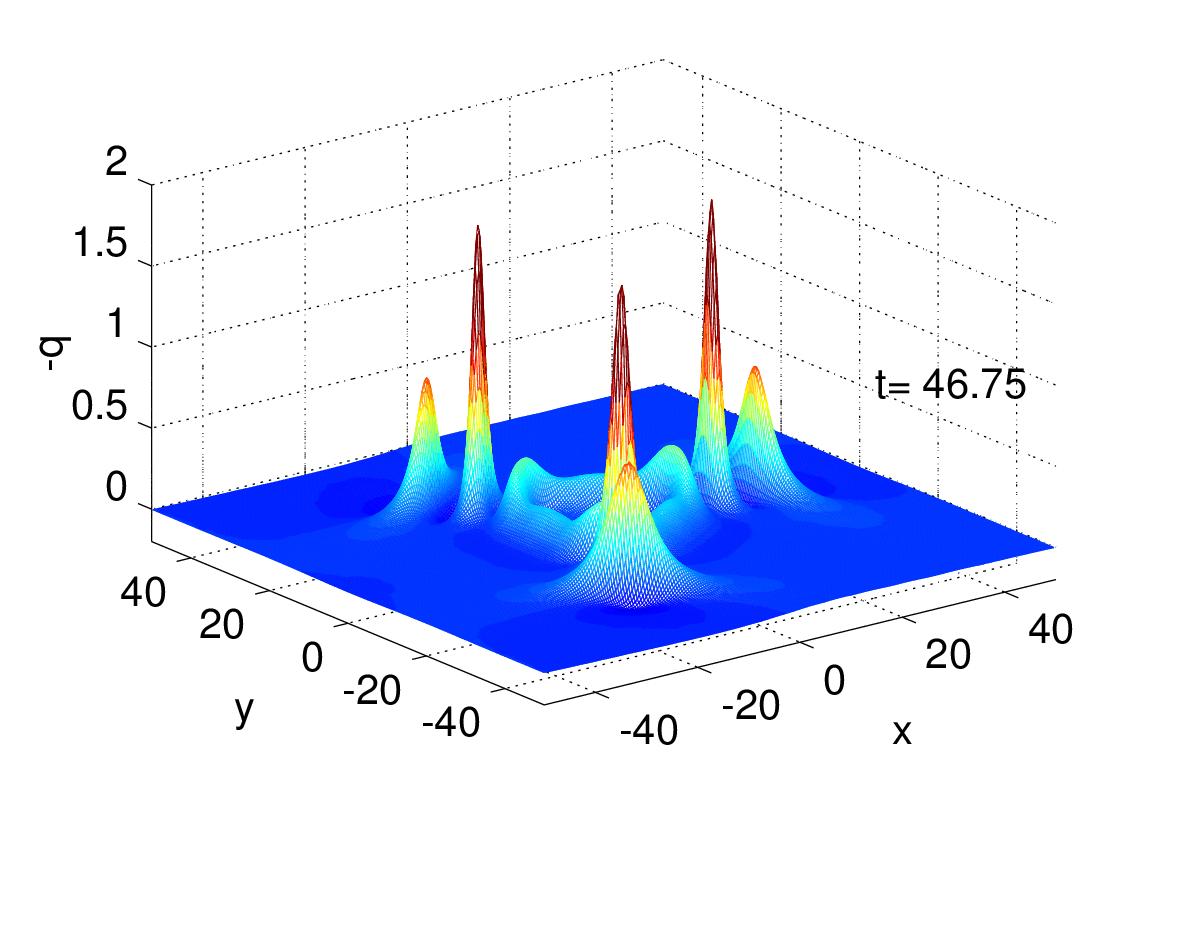}
      \includegraphics[width=4.1cm]{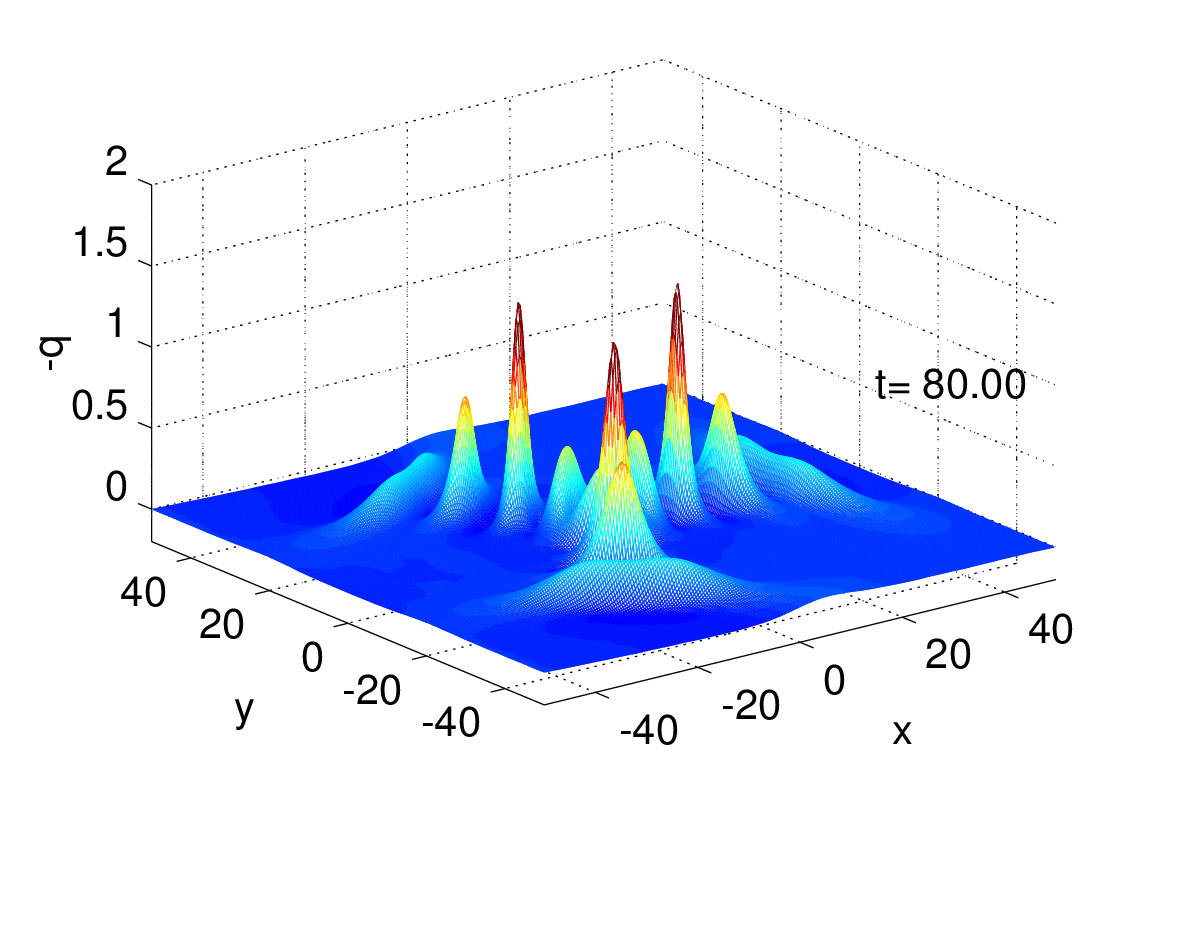}\\
      \includegraphics[width=4.1cm]{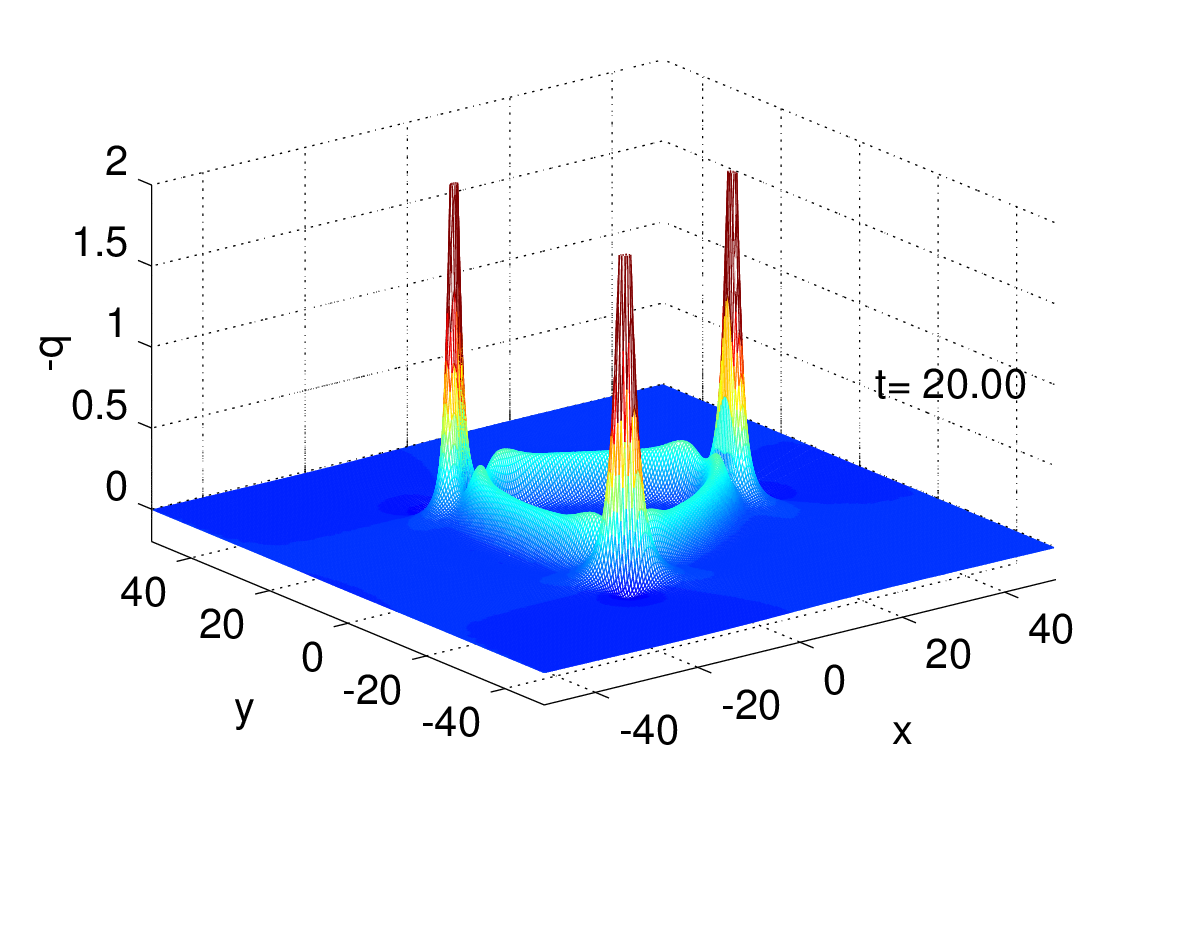}
      \includegraphics[width=4.1cm]{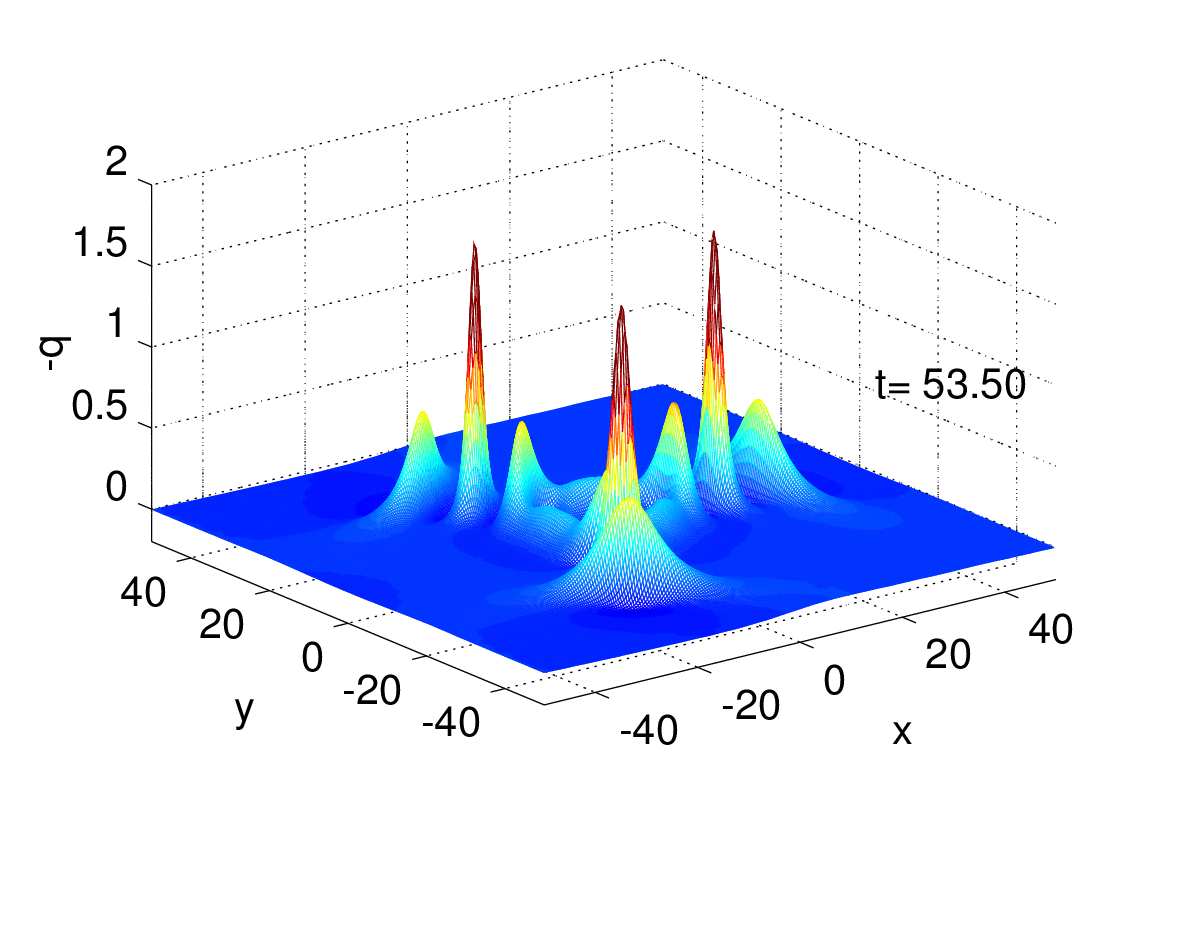}
      \includegraphics[width=4.1cm]{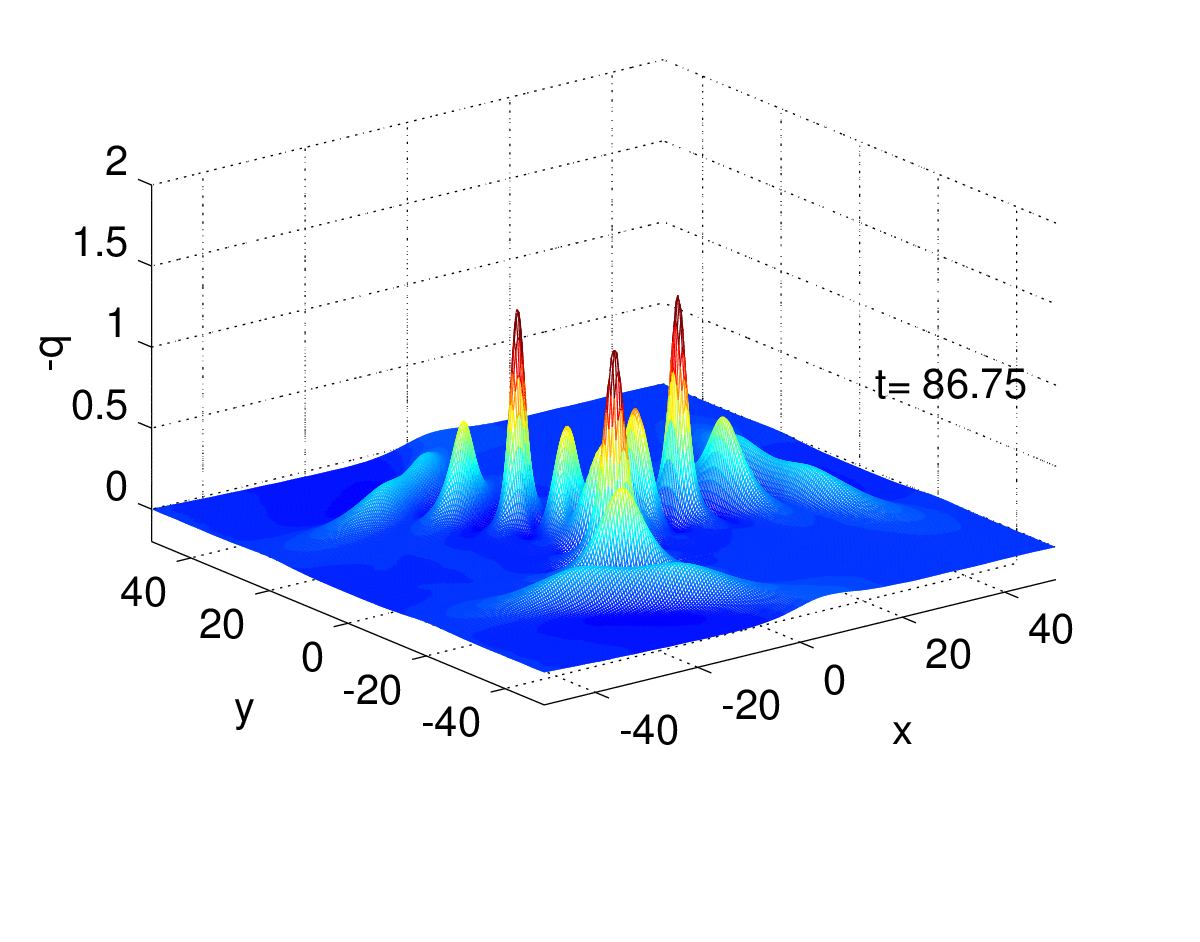}\\
      \includegraphics[width=4.1cm]{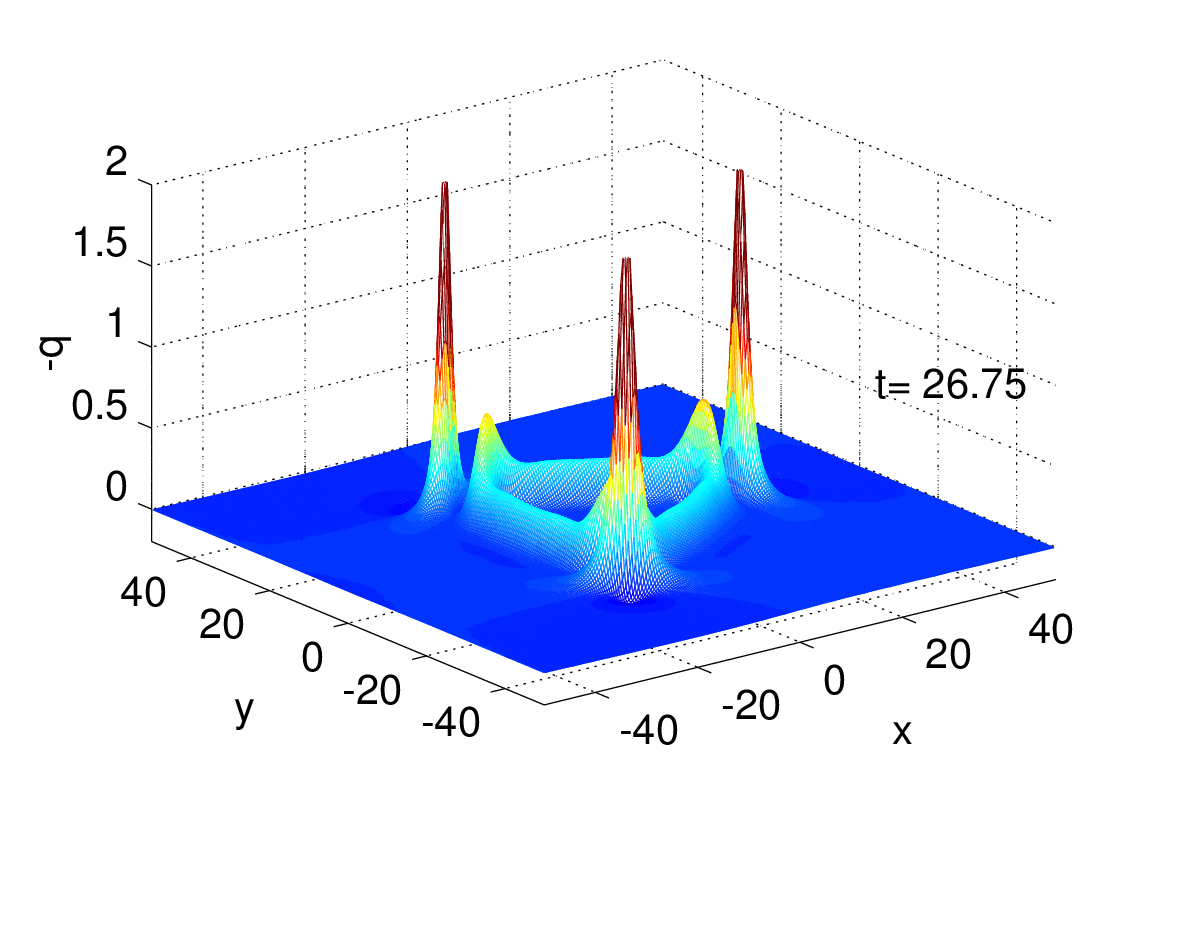}
      \includegraphics[width=4.1cm]{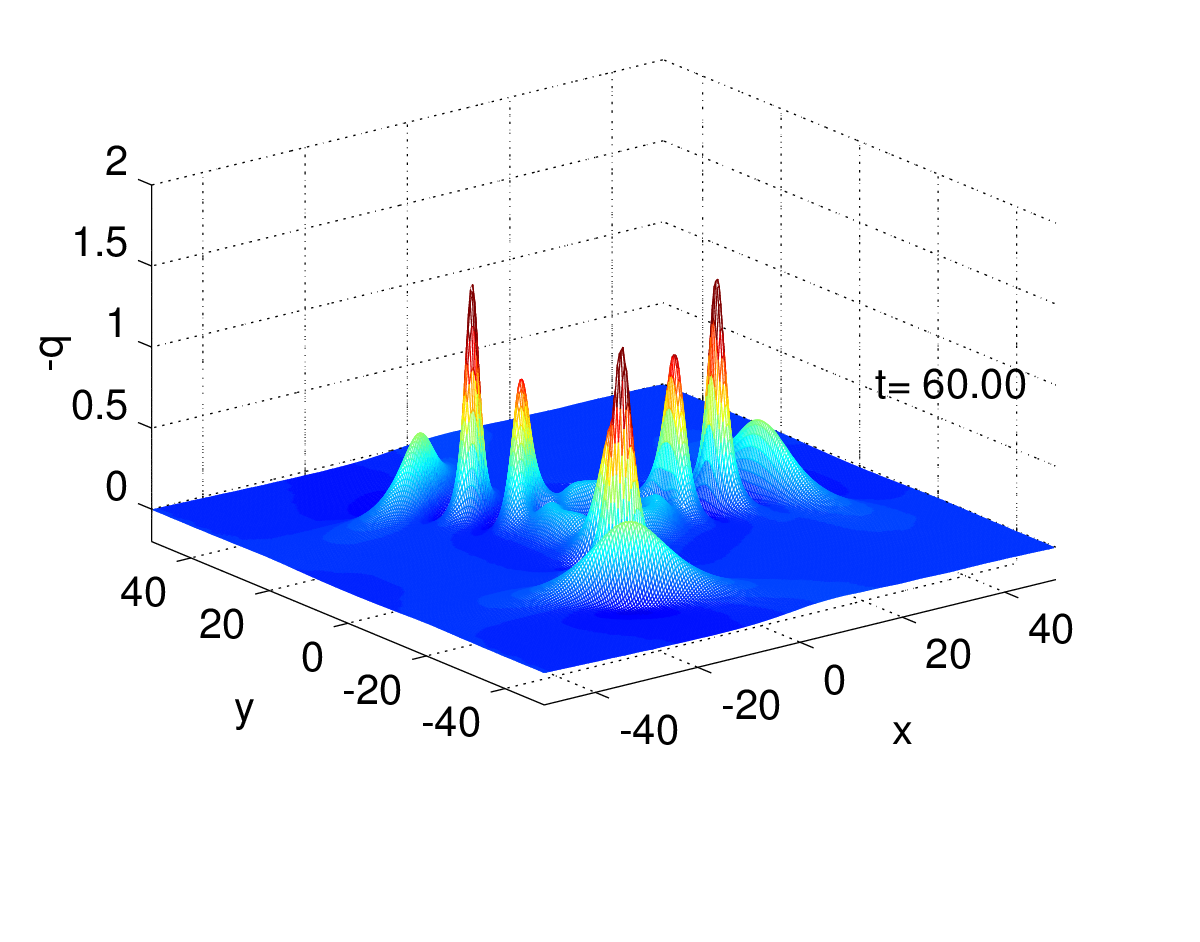}
      \includegraphics[width=4.1cm]{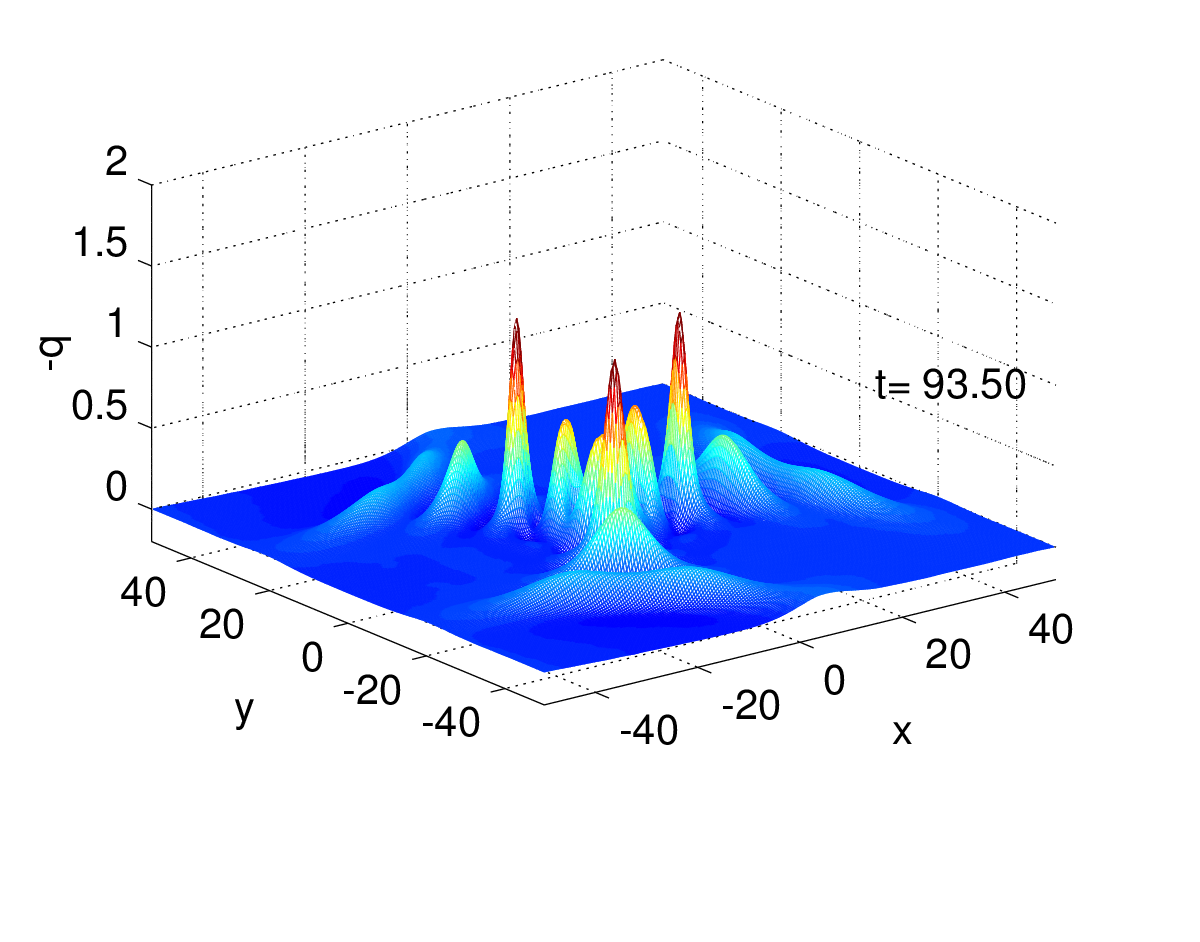}\\
      \includegraphics[width=4.1cm]{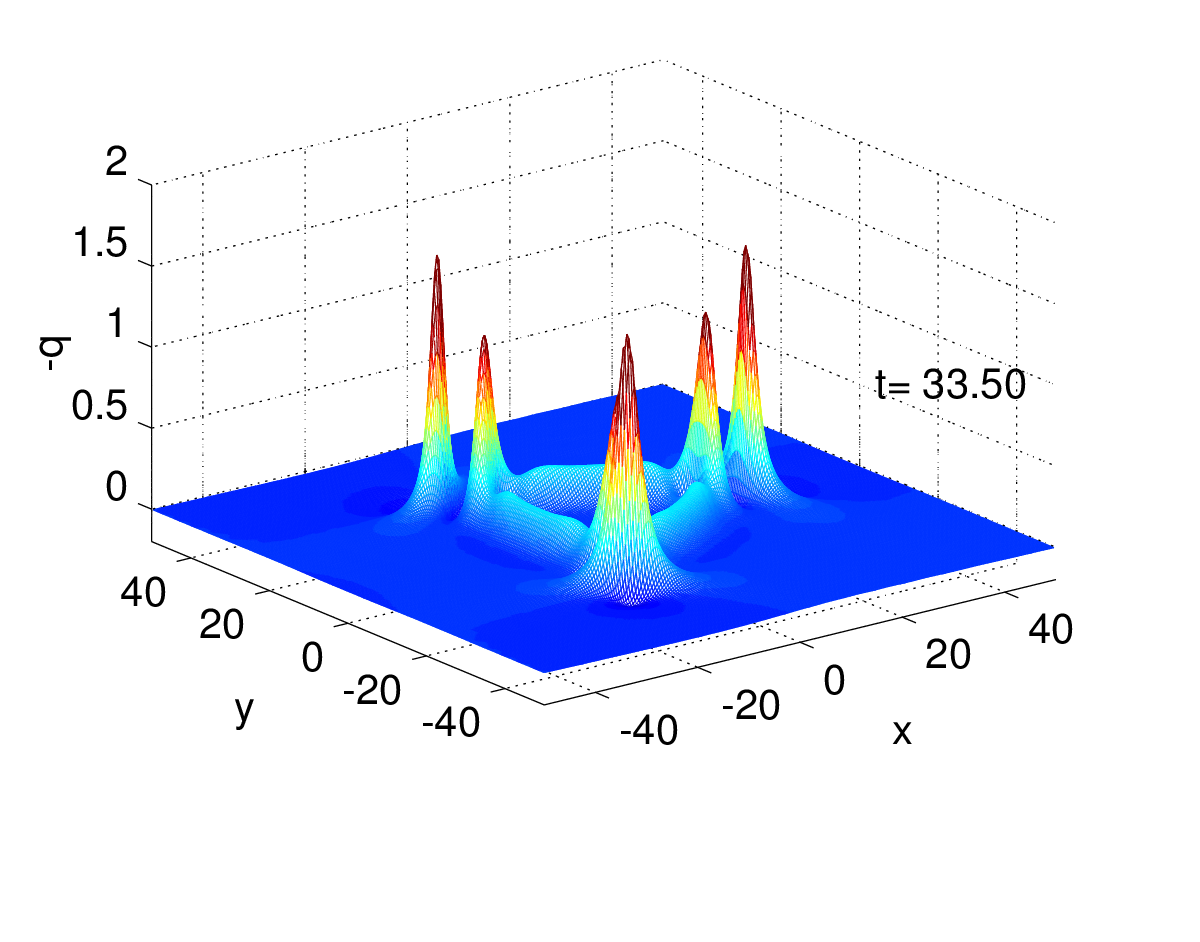}
      \includegraphics[width=4.1cm]{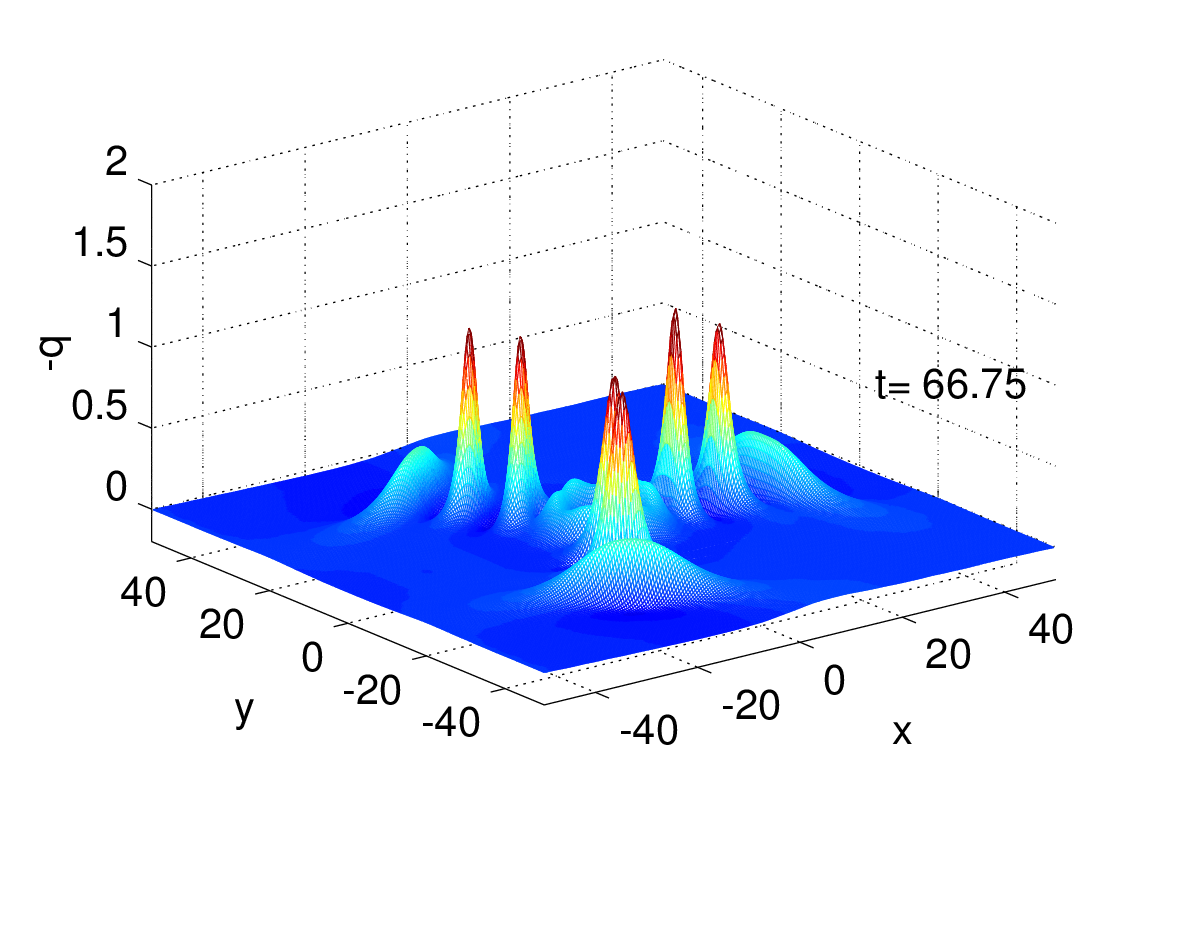}
      \includegraphics[width=4.1cm]{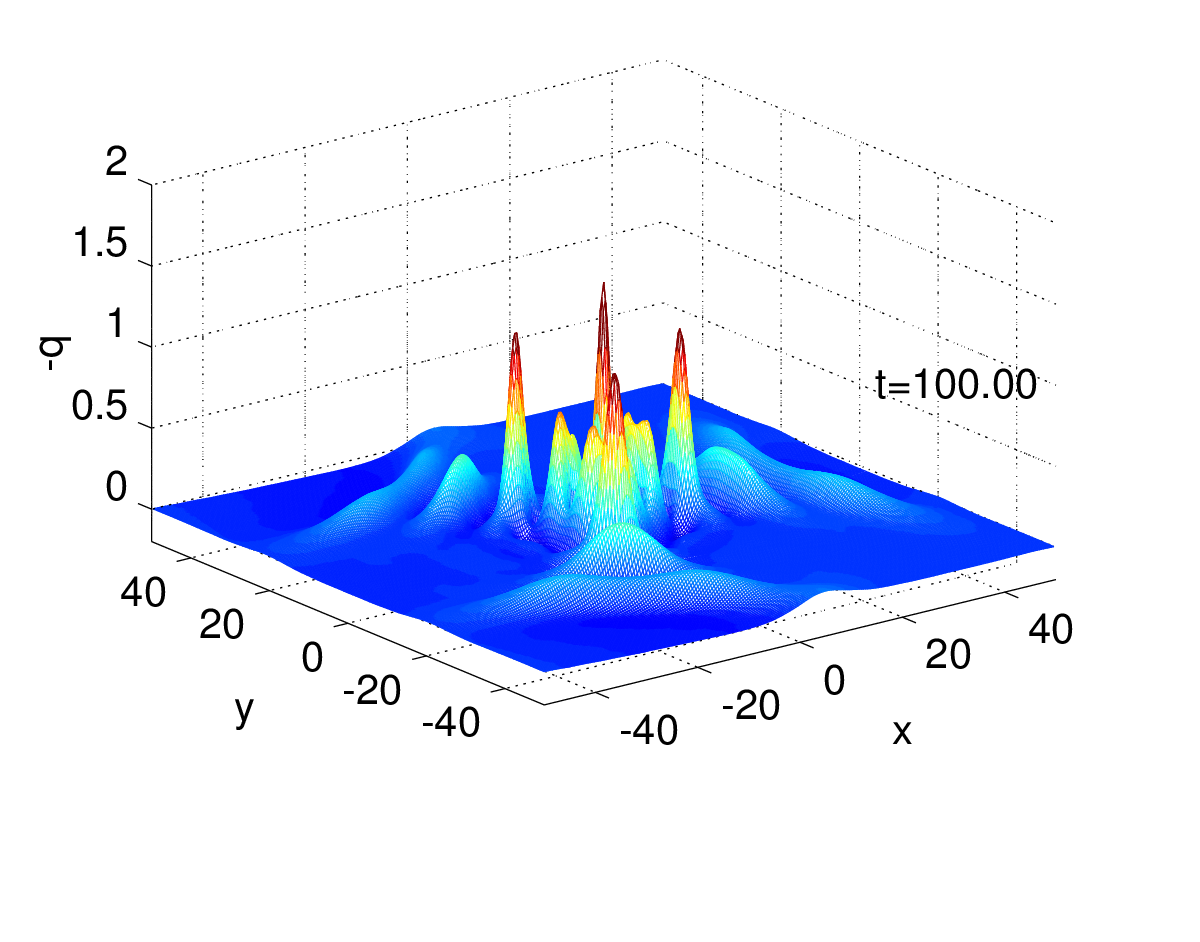}
    \caption{Evolution of a KdV ring by NV at positive energy}
    \label{fig:EvolutionPositiveEnergy}
  \end{center}
\end{figure}

\begin{figure}[htbp]
  \begin{center}
    \leavevmode
      \includegraphics[width=5.5cm]{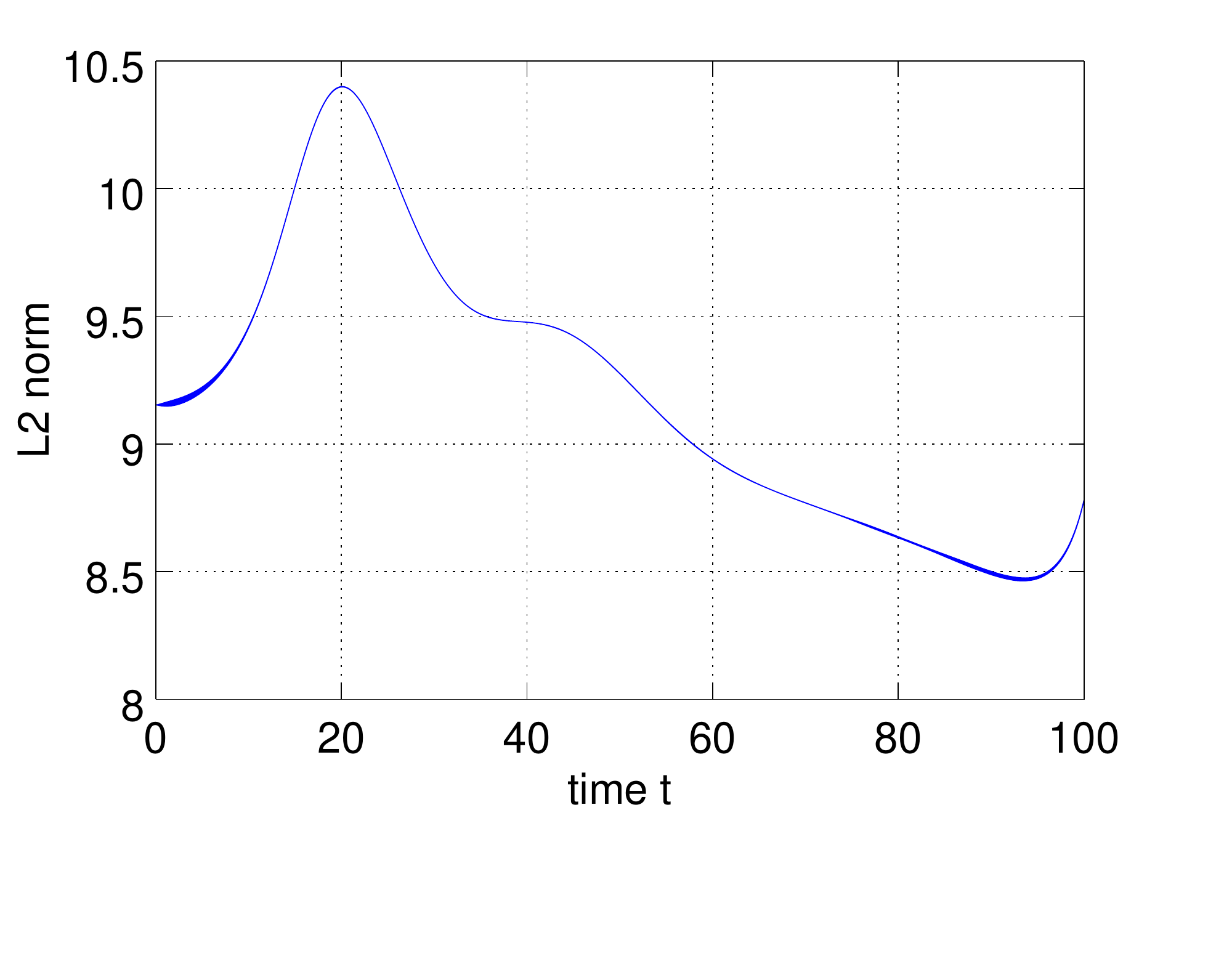}
    \caption{$L_2$ norm of the positive-energy NV solution of Figure 			
    		\ref{fig:EvolutionPositiveEnergy}.}
    \label{fig:EvolutionPositiveEnergyNorm}
  \end{center}
\end{figure}

\clearpage

\subsection{Closed-form Solutions} \label{subsec:RyansThesis}

Most, if not all,  soliton equations admit traveling waves solutions
that involve the hyperbolic secant function, which can be written in
terms of the hyperbolic tangent function.  Moreover, the hyperbolic
tangent function is a solution to the Riccati equation, $\phi' = l_0 +
\phi^2$, for $l_0 < 0$ for certain initial conditions.  The ubiquity
of the hyperbolic functions as traveling wave solutions naturally leads
to the idea of expansion methods for solving soliton equations.

%% note PP uncommented the first line
%% Not JM re-commented the first line and revised this paragraph
%The reference \cite{Nikolai:2009} points out 
%%some common errors researchers make when looking for new
%%traveling-wave solutions to nonlinear partial differential equations.
%Alas, the NV equation is under-studied, and at the time of this 
% publication the only 
Solutions in the literature to the NV equation include the solutions
from the inverse scattering transform
\cite{LMS:2007,LMSS:2011,LMSS:2012,VN:1984},
the classic hyperbolic secant and cnoidal solutions \cite{Nickel:2006},
and rational solutions derived using Darboux transformations that lead
to finite time blow--up (see \cite{TT:2010b} and references therein). In this section we present new solutions
to the NV equation using Hirota's bilinear method and the Extended Mapping Approach (EMA).  New multiple traveling wave solutions using the Modified Extended Tanh-Function Method can be found in \cite{RyansThesis}.  This approach results in closed-form solutions, most of which contain singularities.  We note that the solutions found by Hirota's method are plane-wave solutions, that is, KdV-type solutions, while the EMA-derived solutions are not.

\subsubsection{Hirota's Method} \label{subsub:Hirota}

Following the pioneering work of Hirota \cite{Hirota:1973}, multisoliton solutions can be derived using {\em Hirota's bilinear method}.  This method yields soliton solutions as a sum of polynomials of exponentials and was used in \cite{Wazwaz:2010} to find multisoliton solutions to the Nizhnik-Novikov-Veselov equation
\begin{eqnarray*}
\dot q &=& -aq_{xxx} +bq_{yyy} -3a(qu_1)_x -3b(qu_2)_y \\
q_x &=& (u_1)_y \\
q_y &=& (u_2)_x
\end{eqnarray*}
  The main idea is to reduce the nonlinear equation to a bilinear form through a transformation involving the logarithmic function. To express the wave velocity $c$ in terms of the dispersive coefficients, assume $u$ is a plane wave solution with $k_1=k_2=k$, $u=e^{kx+ky-ct}$, and substitute $u$ into \eqref{eq:linearNV}.  This results in $c=-k^3/2$.  Under the transformations $q=R(\ln(f))_{xx}$, $v=R(\ln(f))_{xy}$, and $w=R(\ln(f))_{yy}$ where $f(x,y,t) = 1+Ce^{kx+ky+\frac{k^3}{2}t}$ and $C$ is an arbitrary constant, one can algebraically solve for $R$ to find a bilinear form (one finds $R=2$).  This method results in the soliton solution
% corrected per AS
%\begin{equation} q(x,y,t) = v(x,y,t) = w(x,y,t) = \frac{2Ck^2e^{k(2x + 2y + k^2t)/2}}{\left(1 + e^{k(2x + 2y + k^2t)/2}\right)^2.} \label{NVHirota1Sol}\end{equation}
\begin{equation} q(x,y,t) = u_1(x,y,t) = u_2(x,y,t) =
\frac{2\,C\,k^2e^{k\,(2\,x + 2\,y + k^2t)/2}}
{\left(1 + e^{k(2\,x + 2\,y + k^2t)/2}\right)^2} . \label{NVHirota1Sol}
\end{equation}

Choosing 
$$f(x,y,t) = 1+e^{\theta_1}+e^{\theta_2}+a_{12}e^{\theta_1+\theta_2},$$
where $\theta_i=k_ix+k_iy+\frac{1}{2}k_i^3t$, $i=1,2$, in the logarithmic transformations above results in the  two-soliton solution with $a_{12}$ given in terms of $k_1$ and $k_2$ by $a_{12} = (k_1-k_2)^2/(k_1+k_2)^2$
\begin{align}
q(x,y,t) & = \frac{2\left( k_1^2e^{\theta_1} + k_2^2e^{\theta_2} + (k_1 - k_2)^2e^{\theta_1 + \theta_2}\right)}{1 + e^{\theta_1} + e^{\theta_2} + \frac{(k_1 - k_2)^2}{(k_1 + k_2)^2}e^{\theta_1 + \theta_2}}  \notag \\
&  \hspace{1cm}  \hspace{1cm} - \frac{2\left( k_1e^{\theta_1} + k_2e^{\theta_2} + \frac{(k_1 - k_2)^2}{k_1 + k_2}e^{\theta_1 + \theta_2}\right)^2}{\left(1 + e^{\theta_1} + e^{\theta_2} + \frac{(k_1 - k_2)^2}{(k_1 + k_2)^2}e^{\theta_1 + \theta_2}\right)^2}. \label{eq:2SolitonNV}
\end{align}
The evolution of the two-soliton solution is plotted in Figure \ref{fig:Hirota}.  Further details and a three-soliton solution are found in \cite{RyansThesis}.

The above Hirota solutions are planar solutions to the NV equations and thus are connected to solutions of
KdV, see Remark~\ref{rem:NVtoKdV}.

\begin{figure}[!ht]
\centering
\begin{tabular}{cc}
\includegraphics[width=0.5\linewidth,height=0.32\textheight]{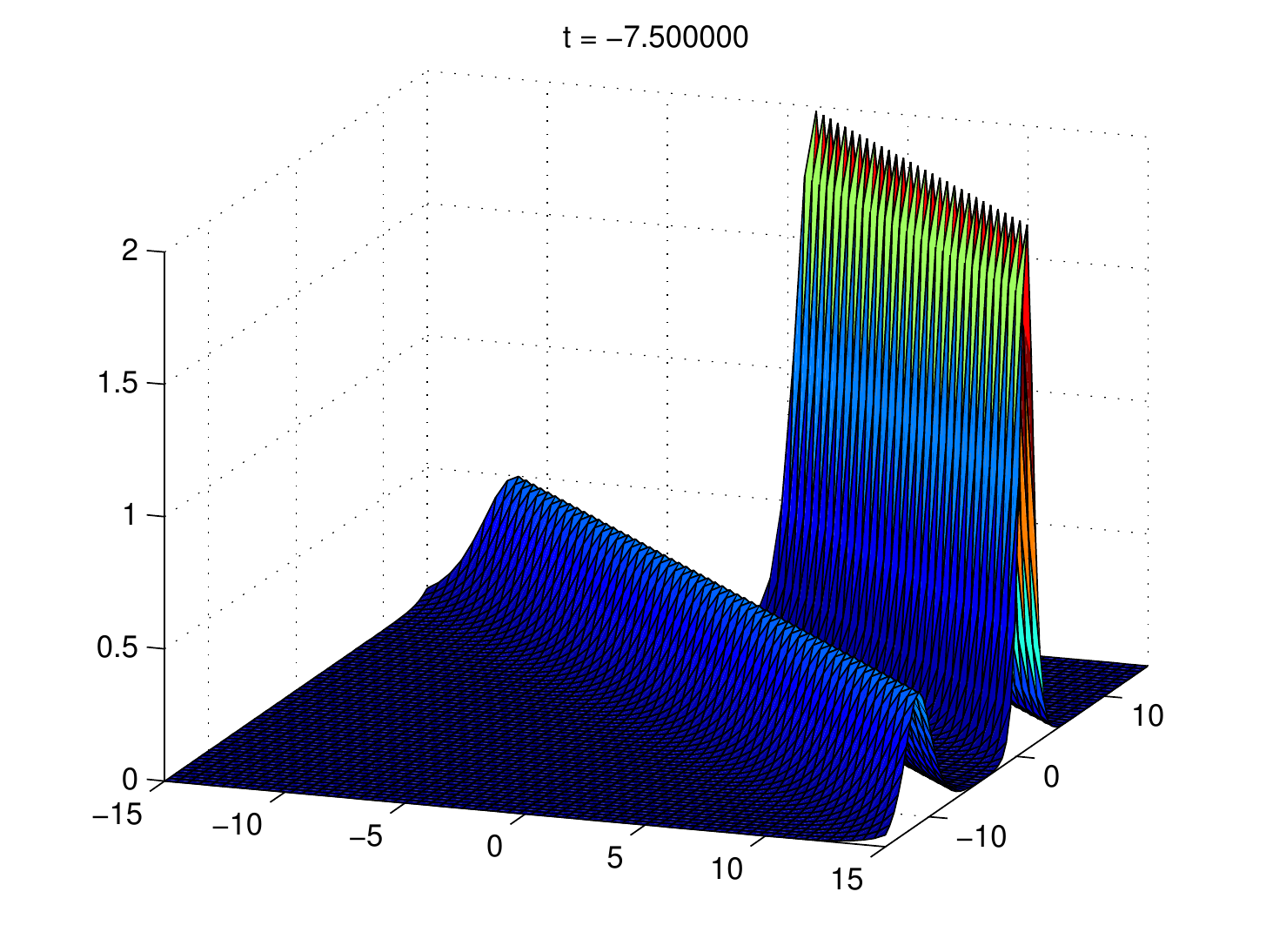}&
\includegraphics[width=0.5\linewidth,height=0.32\textheight]{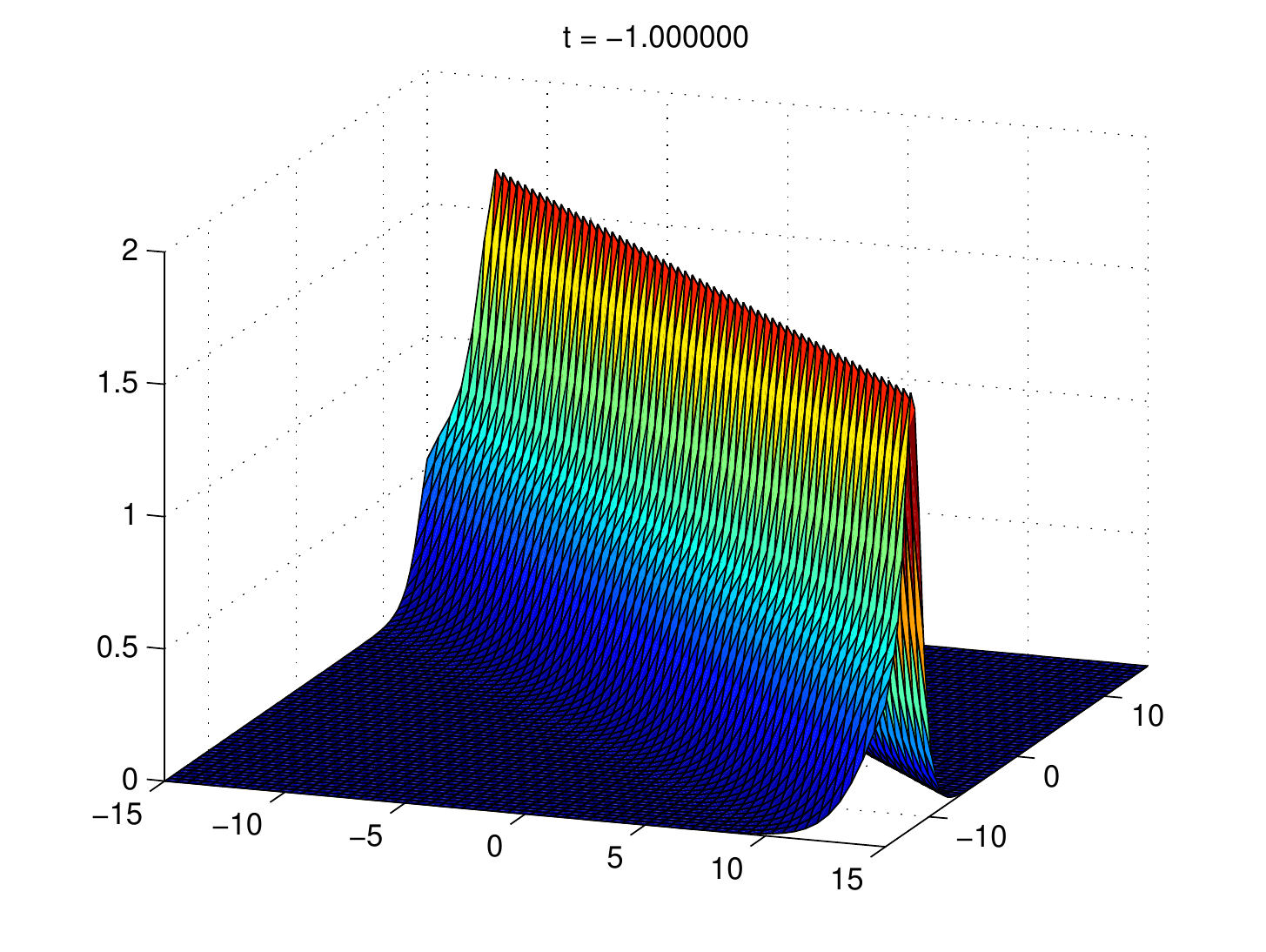}\\
\includegraphics[width=0.5\linewidth,height=0.32\textheight]{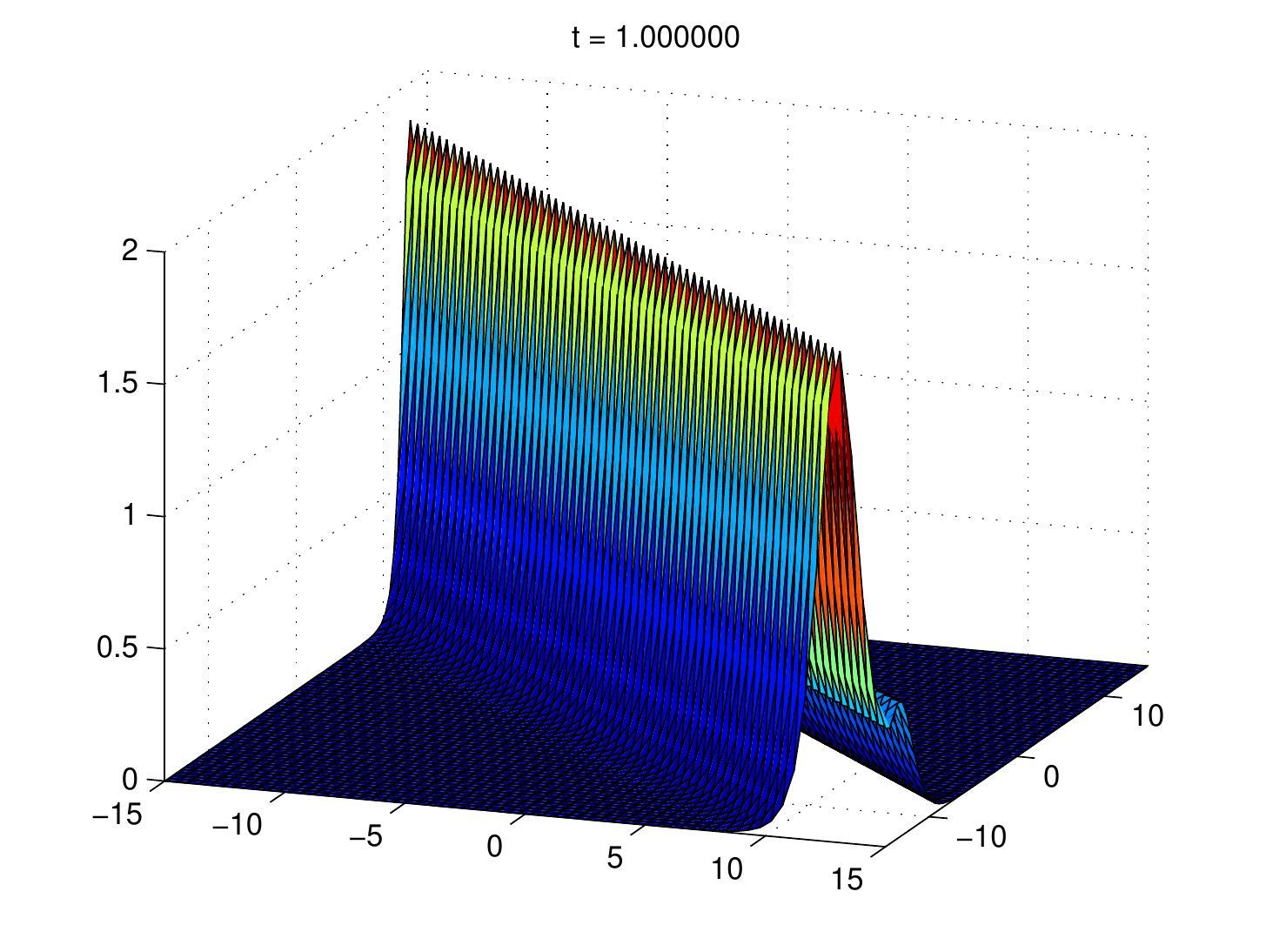}&
\includegraphics[width=0.5\linewidth,height=0.32\textheight]{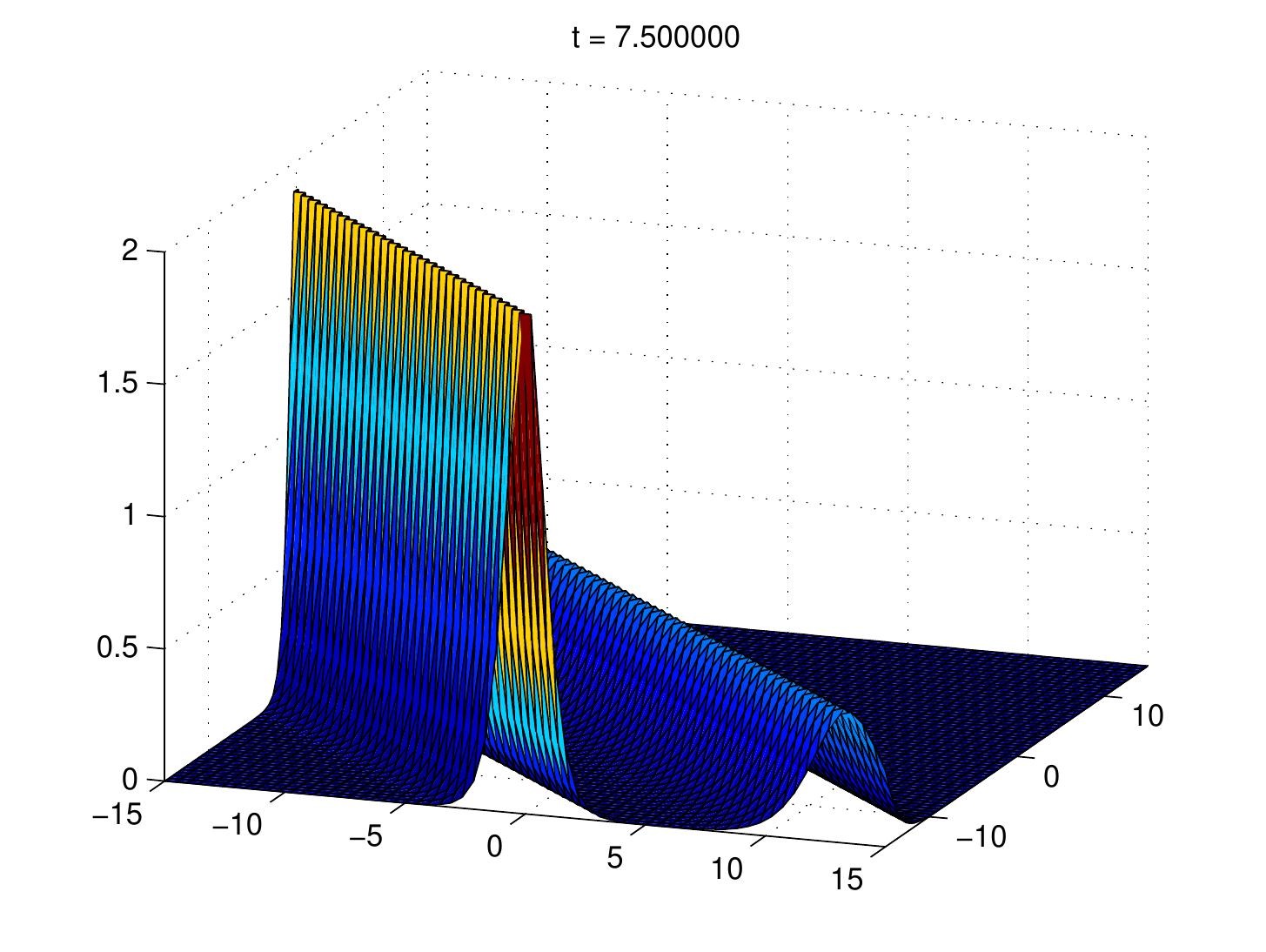}
\end{tabular}
\caption[Time snapshots of the evolution of a 2-soliton solution derived Hirota's bilinear method.]{\centering Time snapshots of the evolution of a 2-soliton solution derived Hirota's bilinear method.}
\label{fig:Hirota}
\end{figure}

\subsubsection{Extended Mapping Approach}

The extended mapping approach (EMA) was presented formally by Zheng \cite{Zheng:2005} and extends results by Lou and Ni \cite{Lou:1989}.  The method is designed to find mappings between nonlinear PDE's.  In this approach, $q, v, w$ are expanded in terms of a function $\phi_i$ that satisfies the Riccati equation
$$\frac{d\phi}{dR} = \ell_0 +\phi^2,$$
where $R=R(x,y,t)$.  Thus, $q(x,y,t) = \sum_{i=0}^n a_i\phi^i$, $v(x,y,t) = \sum_{i=0}^m b_i\phi^i$, and $w(x,y,t) = \sum_{i=0}^k c_i\phi^i$, where the values of $n, m$ and $k$ are determined by balancing the highest order derivative terms with the nonlinear terms of the PDE.  The method is described nicely in \cite{Singareva:2011}.  The balancing method results in $n=m=k=2$.  Substituting these expansions into the NV equation and equating coefficients of the resulting polynomial in $\phi$ results in a system of thirteen PDE's from which we need to solve for the coefficients $a_i(x,y,t), b_i(x,y,t)$, and $c_i(x,y,t)$, $i=1,2$.  Using a separation technique for $R$, namely, $R(x,y,t) = p(x,t)+q(y,t)$ results in $sech^2$ solutions, static (time-independent) solutions, and breather-type solutions of the NV equation.  Further details, including the choices of $\phi$ are found in \cite{RyansThesis}.
%% Replace per AS
A time-independent solution is given by
\begin{eqnarray*}
q(x,y,t) &=&
\frac{-1728\,y^6+(-96+1728\,C)y^4+(-40+288\,C)\,y^2-36\,C+5}
   {432\,y^4-36\,y^2}-\\
&&   -4\,\tanh(x+y^2) + (2+8\,y)\,\tanh^2(x+y^2)\\
v(x,y,t) & =& \frac{144\,y^4+(-12+432\,C)y^2-36\,C+5}{36\,y^2} \\
&&+ 4\,\tanh(x+y^2) + (2-8\,y)\tanh^2(x+y^2)\\
w(x,y,t) &=& 8\,y - 8\,y\,\tanh^2(x+y^2),
\end{eqnarray*}
% \begin{align} q(x,y,t) & =
%   \frac{-1728y^6+(-96+1728C)y^4+(-40+288C)y^2-36C+5}{432y^4-36y^2}
%   \notag \\ &  \hspace{1cm}   -4\tanh(x+y^2) + (2+8y)\tanh^2(x+y^2)
%   \\ v(x,y,t) & = \frac{144y^4+(-12+432C)y^2-36C+5}{36y^2} \notag \\ &
%   \hspace{1cm}    + 4\tanh(x+y^2) + (2-8y)\tanh^2(x+y^2)\\ w(x,y,t) &
%   = 8y - 8y\tanh^2(x+y^2), \end{align}
where $C$ is an arbitrary constant.  See Figure \ref{fig:StaticEMA} for a plot of the solution $q$ with $C=1$. 

\begin{figure}[!ht]
	\centering
		\includegraphics[width=0.5\linewidth]{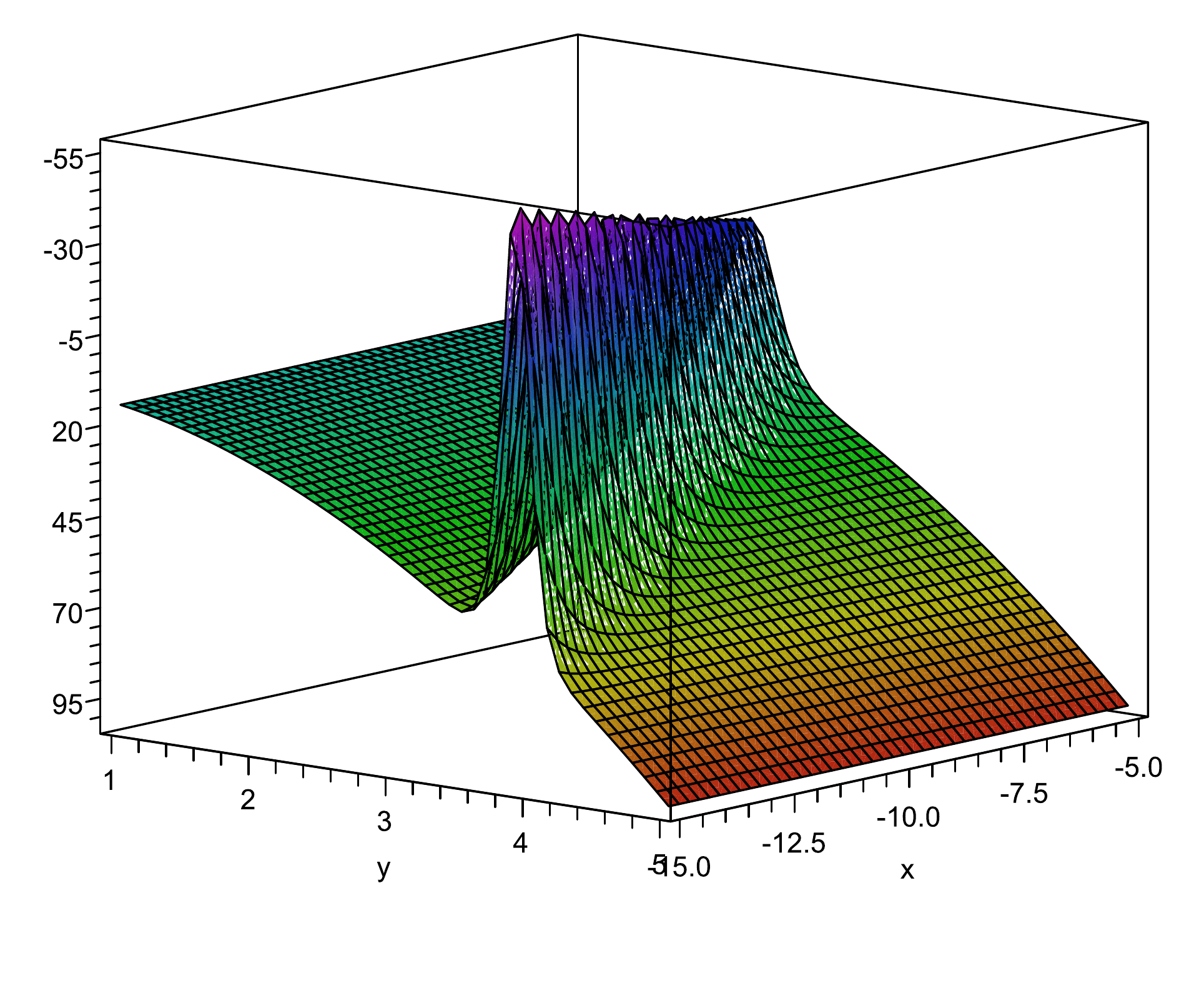}
		\caption{A static solution to the NV equation}
	\label{fig:StaticEMA}
\end{figure}

Breather solutions are solutions with a type of periodic back--and--forth motion in time.  One particular breather solution from \cite{RyansThesis} is 
%% Replace per AS
\begin{eqnarray*}
q(x,y,t) &=&
  \frac{-1728\,y^6+(-96+1728\,C)\,y^4+(-40+288\,C)\,y^2-36\,C+5}
  {432\,y^4-36\,y^2}-\\
&& -4\,\tanh(1+x+y^2+4\,\cos t) +
  (2+8\,y^2)\,\tanh^2(1+x+y^2+4\,\cos t)  \\
 v(x,y,t) &=&
  \frac{-192\,\sin(t)\,y^2 + 144\,y^4+(-12+432\,C)\,y^2-36\,C+5}{36\,y^2}+\\
  &&+ 4\,\tanh(1+x+y^2+4\,\cos t) +
  (2-8\,y^2)\tanh^2(1+x+y^2+4\,\cos t)\\
w(x,y,t) &=& 8\,y - 8\,y\,\tanh^2(1+x+y^2+4\,\cos t).
\end{eqnarray*}
% \begin{align} q(x,y,t) & =
%   \frac{-1728y^6+(-96+1728C)y^4+(-40+288C)y^2-36C+5}{432y^4-36y^2}
%   \notag \\ &  \hspace{1cm}   -4\tanh(1+x+y^2+4\cos t) +
%   (2+8y^2)\tanh^2(1+x+y^2+4\cos t)  \\ v(x,y,t) & =
%   \frac{-192\sin(t)y^2 + 144y^4+(-12+432C)y^2-36C+5}{36y^2} \notag \\
%   &  \hspace{1cm}    + 4\tanh(1+x+y^2+4\cos t) +
%   (2-8y^2)\tanh^2(1+x+y^2+4\cos t)\\ w(x,y,t) & = 8y -
%   8y\tanh^2(1+x+y^2+4\cos t). \end{align}
Several time snapshots are shown in Figure~\eqref{fig:BreathersEMA}.
For multisoliton solutions the reader is referred to~\cite{RyansThesis}.

\begin{figure}[!ht]
\centering
\begin{tabular}{cc}
\includegraphics[width=0.5\linewidth,height=0.32\textheight]{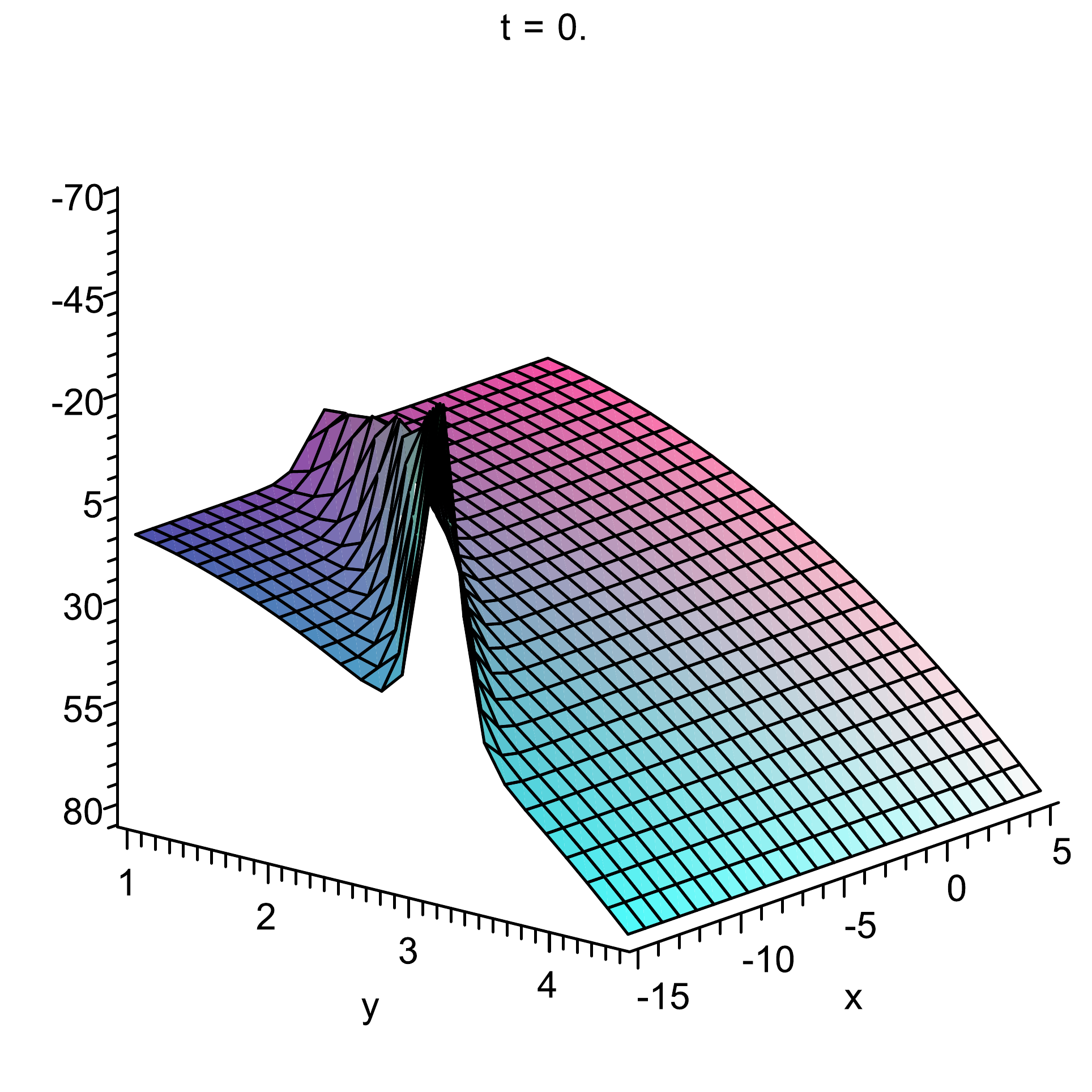}&
\includegraphics[width=0.5\linewidth,height=0.32\textheight]{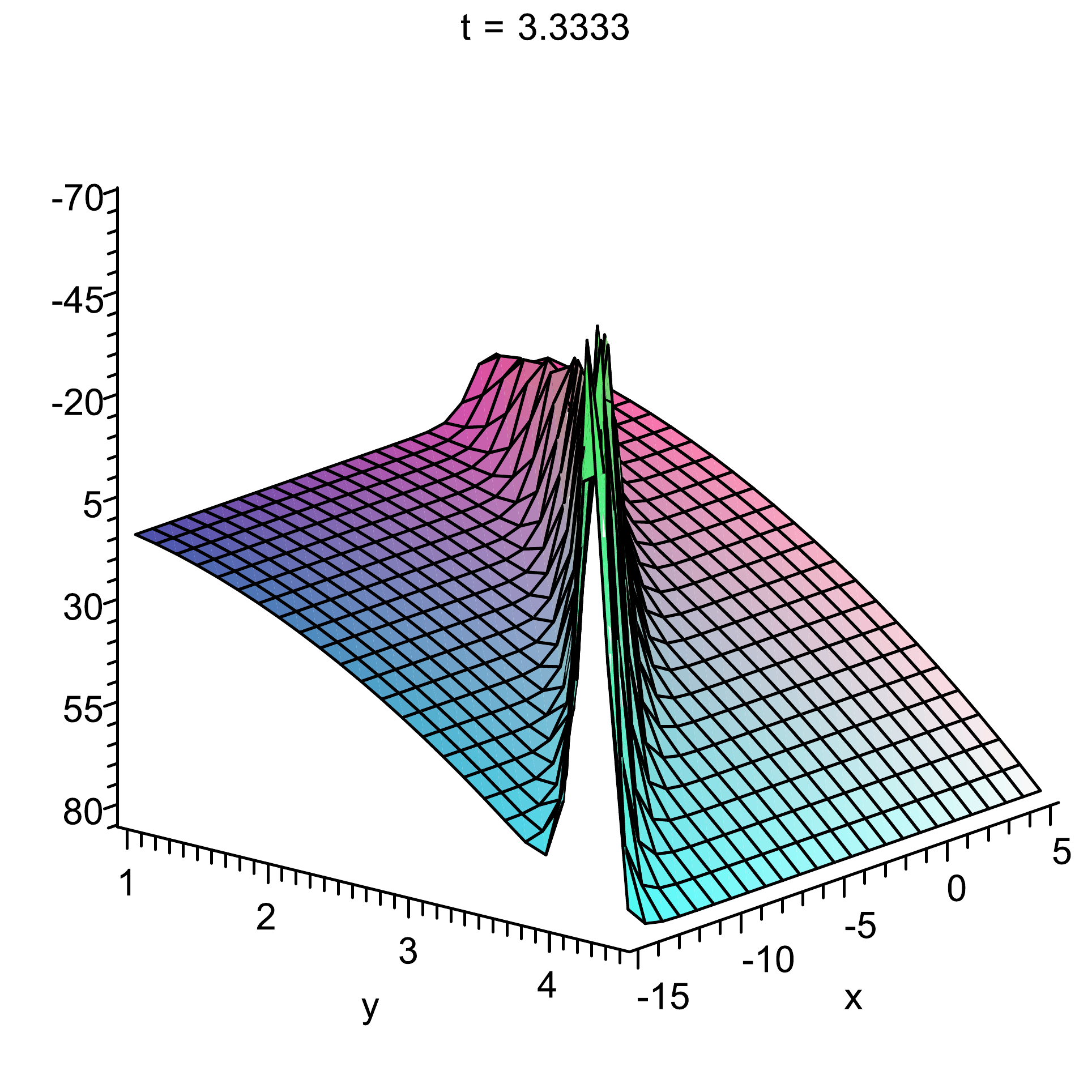}\\
\includegraphics[width=0.5\linewidth,height=0.32\textheight]{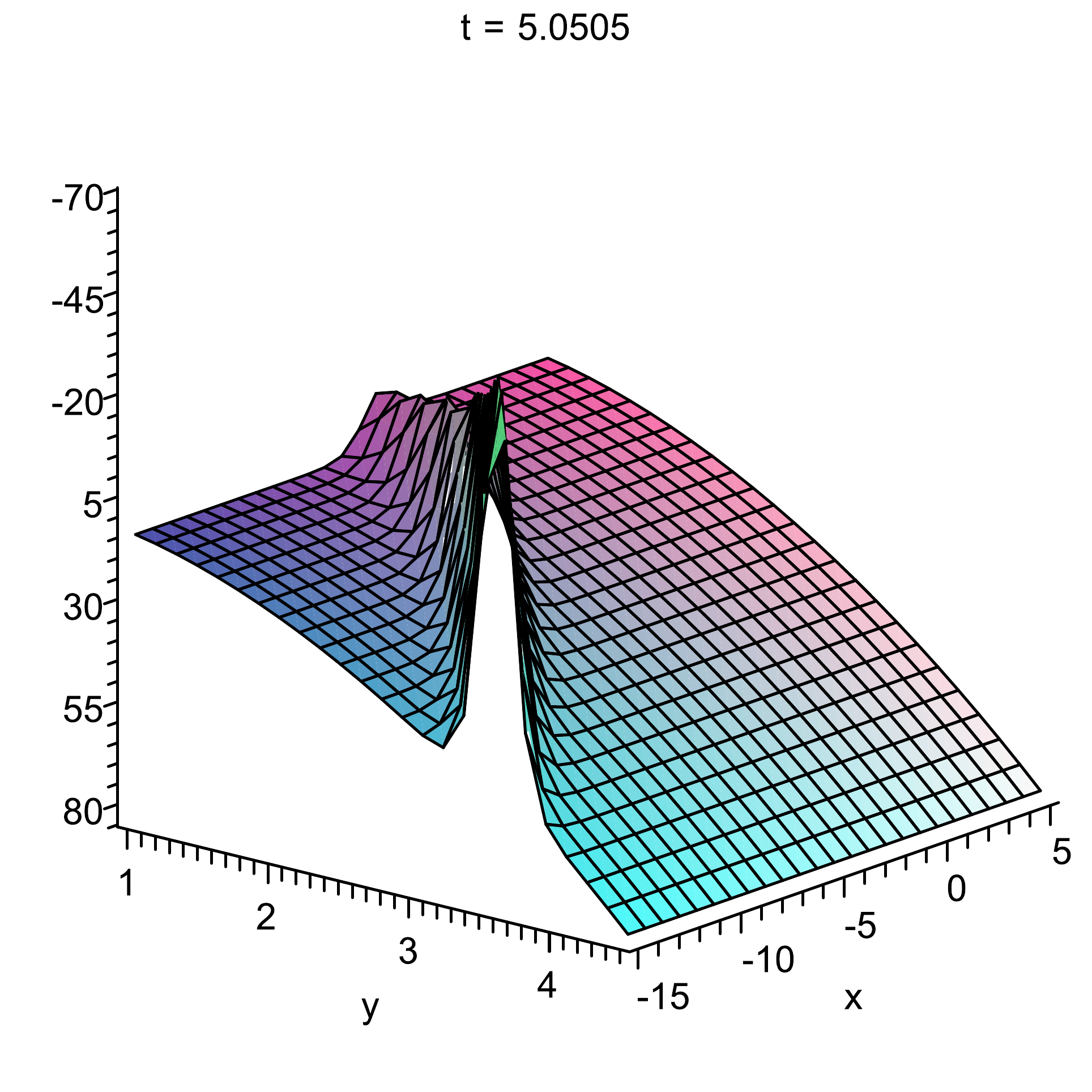}&
\includegraphics[width=0.5\linewidth,height=0.32\textheight]{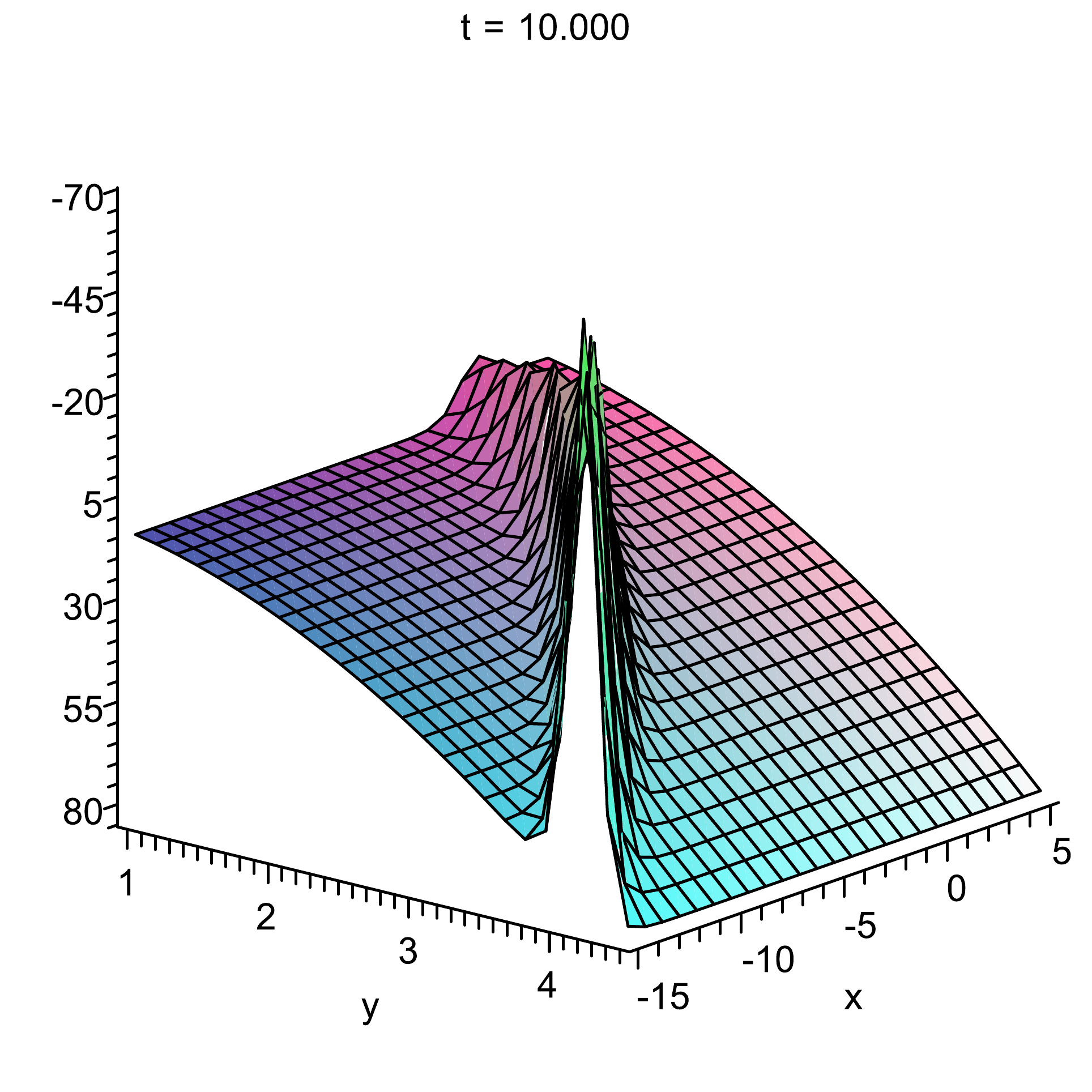}
\end{tabular}
\caption[Time snapshots of a breather solution derived using the EMA.]{\centering Time snapshots of a breather solution derived using the EMA.}
\label{fig:BreathersEMA}
\end{figure}

For multisoliton solutions the reader is referred to \cite{RyansThesis}

\clearpage

\subsection{Scattering transform of the ring soliton at time
  zero} \label{subsec:tk_ring_sol} 

In this section we compute numerically the scattering transform of the
KdV ring soliton discussed in example \ref{ex:3} and illustrated in
Figure \ref{fig:EvolutionZeroEnergy}. Since the initial potential is
supercritical, we expect the scattering transform to have a
singularity.
Since the intial potential is real-valued and rotationally symmetric
in the $z$-plane, also the scattering transform is real-valued and
rotationally symmetric in the $k$-plane. See \cite[Appendix
A]{MPS:2013} for a proof. Therefore it is enough to compute
$\mathbf{t}(k)$ only for real and positive $k$.

We use both the LS and the DN methods described in Section
\ref{sec:scatcomp} and compare the results to verify accuracy. The
values in the range $0.1\leq |k|\leq 4$ are reliable as the results of
both NV and LS methods closely agree there. However, the DN method
does not give reliable results for $|k|>4$.

To assure accuracy for $|k|>4$, we compare the results of the LS
method with two different grids in the $z$-domain. The coarser grid
has $4096\times 4096$ points, and the finer grid has $8192\times 8192$
points. The coarser grid is not a subset of the finer grid. We remark
that both of these grids are significantly finer than those we
typically use for computing scattering transforms for
conductivity-type potentials. Due to high memory requirements, we used a liquid-cooled HP Z800 Workstation with 192 GB of memory. Even with that powerful machine, the evaluation of one point value using the finer grid takes more than 11 hours. The LS method gives closely matching results in the regions $3\leq
|k|\leq 5$ and $|k|>9$.

Figure \ref{Fig:KdVscat} shows the profile of the scattering transform. In the interval $5<|k|<9$ the numerical computation does not converge, resulting either in inaccurate evaluation of the point values of the scattering transform or in complete failure of the algorithm due to using up all the memory.
We suspect that the observed numerical divergence arises
from the existence of at least one exceptional circle in the interval $5<|k|<9$.

\begin{figure}
\begin{picture}(300,120)
\put(-20,-5){\includegraphics[width=12.5cm]{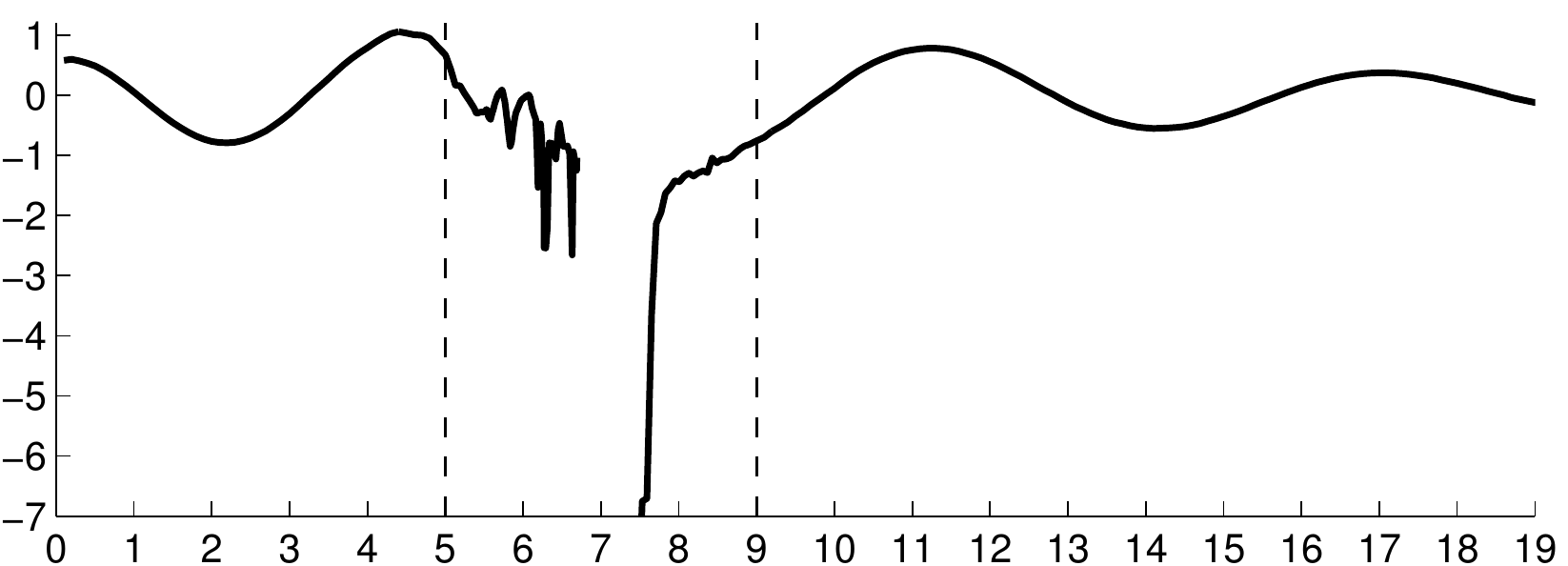}}
\end{picture}
\caption{\label{Fig:KdVscat} Profile of the scattering transform
$\mathbf{t}(k)$ of the KdV ring soliton for $k$ ranging in the
positive real axis. The computation was unstable in the region
$5<|k|<9$, suggesting the presence of an exceptional circle.}
\end{figure}

\clearpage

\section{Zero-energy exceptional points}\label{sec:zero_excep}

\noindent
The inverse scattering method for the solution of the Novikov-Veselov
equation is based on the complex geometrical optics (CGO) solutions
$\psi$ of the Schr\"odinger equation (\ref{eq:CGO-1}). The function
$\psi(z,k)$ is  asymptotically close to the exponential function
$e^{ikz}$ in the sense of formula (\ref{asy_cond_psi}); the point is
that $\psi$ can be used to define a nonlinear Fourier transform
$\mathbf{t}(k)$ specially designed for linearizing the NV
equation. See diagram (\ref{diagram}) above. 

However, there is a possible difficulty in using $\psi$ and
$\mathbf{t}$. Even in the case of a smooth and compactly supported
potential $q\in C_0^\infty$, there may exist complex numbers $k\not=0$
for which equation (\ref{eq:CGO-1}) does not have a unique solution
satisfying the asymptotic condition (\ref{asy_cond_psi}). Such $k$ are
called {\em exceptional points} of $q$. It is shown in \cite{MPS:2013}
that that exceptional points of rotationally symmetric potentials come
in circles centered at the origin and that the scattering transform
has a strong singularity at the circles. The singularity prevents any
currently understood use of the inverse nonlinear Fourier transform in
the diagram (\ref{diagram}). It seems safe to assume that the
situation becomes only worse for more general potentials. 

What is the connection between exceptional points and dynamics of
solutions of the Novikov-Veselov equation? For example, does the
absence of exceptional points in the initial data ensure smooth NV
evolution? Do exceptional points perhaps correspond to lumps or
solitons or finite-time blow-ups? Such a conjecture was presented
already in \cite[page 27]{BLMP:1987}, but the question is still open. 

This section is devoted to a computational experiment illustrating
exceptional points of a parametric family of rotationally symmetric
potentials. The example clarifies the relationship between exceptional
points and the trichotomy supercritical/critical/subcritical presented
in Definition \ref{def:tri}. 

Take a radial $C^2_0$ function $w(z)=w(|z|)$ as shown in Figure
\ref{fig:testfunw}. A detailed definition of $w$ is given in
\cite[Section 5.1]{MP:2013}. Define a family of potentials by
$q_\lambda=\lambda w$, parameterized by $\lambda\in\mathbb{R}$. Now
the case $\lambda=0$ gives $q_0\equiv 0$, which is a critical
potential since it arises as $q_0=\sigma^{-1}\Delta\sigma$ with the
positive function $\sigma\equiv 1$. From  Murata \cite{Murata:1984} we
see that $\lambda<0$ gives a supercritical potential and $\lambda>0$
gives a subritical potential. See \cite[Appendix B]{MPS:2013} for
details. 

We use the DN method described in Section \ref{sec:scatcomp} to
compute the scattering transforms of the potentials $q_\lambda$ for
the parameter $\lambda$ ranging in the interval $[-25,5]$. Since each
potential $q_\lambda(z)$ is real-valued and rotationally symmetric in
the $z$-plane, also the scattering transform is real-valued and
rotationally symmetric in the $k$-plane. See \cite[Appendix
A]{MPS:2013} for details. Therefore it is enough to compute
$\mathbf{t}(k)$ only for $k$ ranging along the positive real axis. 
%In
%Figure \ref{fig:scat1A}  we show the result of the computation as a
%two-dimensional grayscale image.  
We show a selection of the results of the computation as function plots in Figure 
\ref{fig:exceptionalcircles}, and all the results in  Figure \ref{fig:scat1A} as a 
two-dimensional grayscale image.

It is known that critical potentials do not have nonzero exceptional
points; see \cite{MPS:2013,Nachman:1996}. Thus there are no
singularities in Figure \ref{fig:scat1A} for $\lambda=0$ (actually in
this simple example we have $\mathcal{T}q_0\equiv 0$). Furthermore, a
Neumann series argument shows that for a fixed $\lambda$ there exists
such a positive constant $K=K(\lambda)$ that there are no exceptional
points for $q_\lambda$ satisfying $|k|>K$. See the analysis in
\cite[above formula (1.12)]{Nachman:1996}.  

According to \cite{Music:2013}, subcritical potentials do not have
nonzero exceptional points. Thus there are no singularities in Figure
\ref{fig:scat1A} for parameter values $\lambda\geq 0$. (We remark that
the seemingly exceptional curves in the upper right corner of Figures
3 and 9 in \cite{MPS:2013} are due to deteriorating numerical accuracy
for large positive values of $\lambda$. Those figures are trustworthy
only for $\lambda$ close to zero.) 

For negative $\lambda$ close to zero it is known from \cite{MPS:2013}
that there is exactly one circle of exceptional points. The asymptotic
form of that radius as a function of $\lambda$ is calculated
explicitely in \cite{MPS:2013}. 

For $\lambda<\!<0$ there is no precise understanding of exceptional
points as of now; numerical evidence such as Figure \ref{fig:scat1A}
suggests that there may be several exceptional circles. Also,
something curious seem to happen around $\lambda\approx -8$ and
$\lambda\approx -20.5$; at present there is no explanation available. 

See \cite{MPS:2013} for more zero-energy examples and \cite{Siltanen:2013} for analogous evidence of exceptional points at positive energies.

\begin{figure}
\begin{picture}(300,180)
\put(-50,-55){\includegraphics[width=14cm]{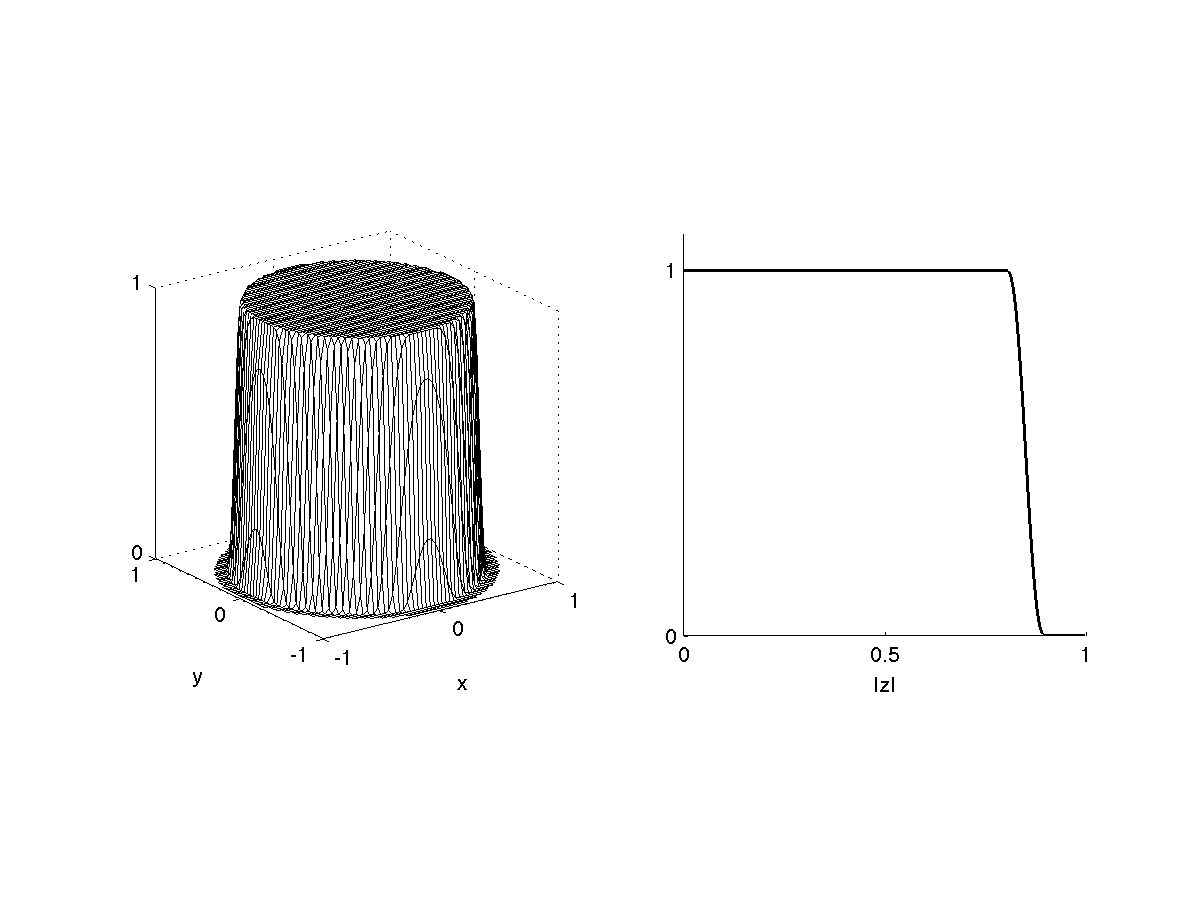}}
\end{picture}
\caption{\label{fig:testfunw}Left: rotationally symmetric test function $w(z)$. Right: the profile $w(|z|)$.}
\end{figure}

\begin{figure}
\begin{picture}(300,500)
\put(30,10){\includegraphics[width=9cm]{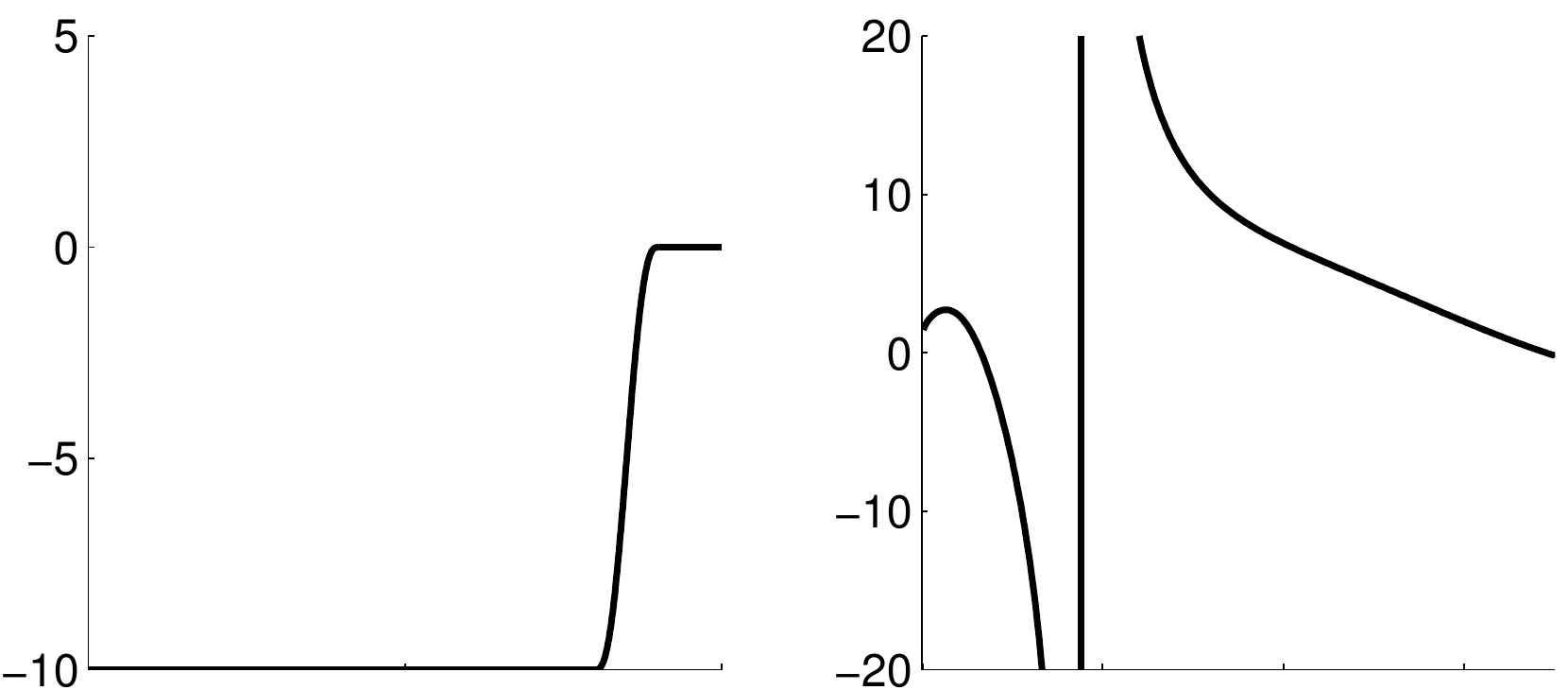}}
\put(30,130){\includegraphics[width=9cm]{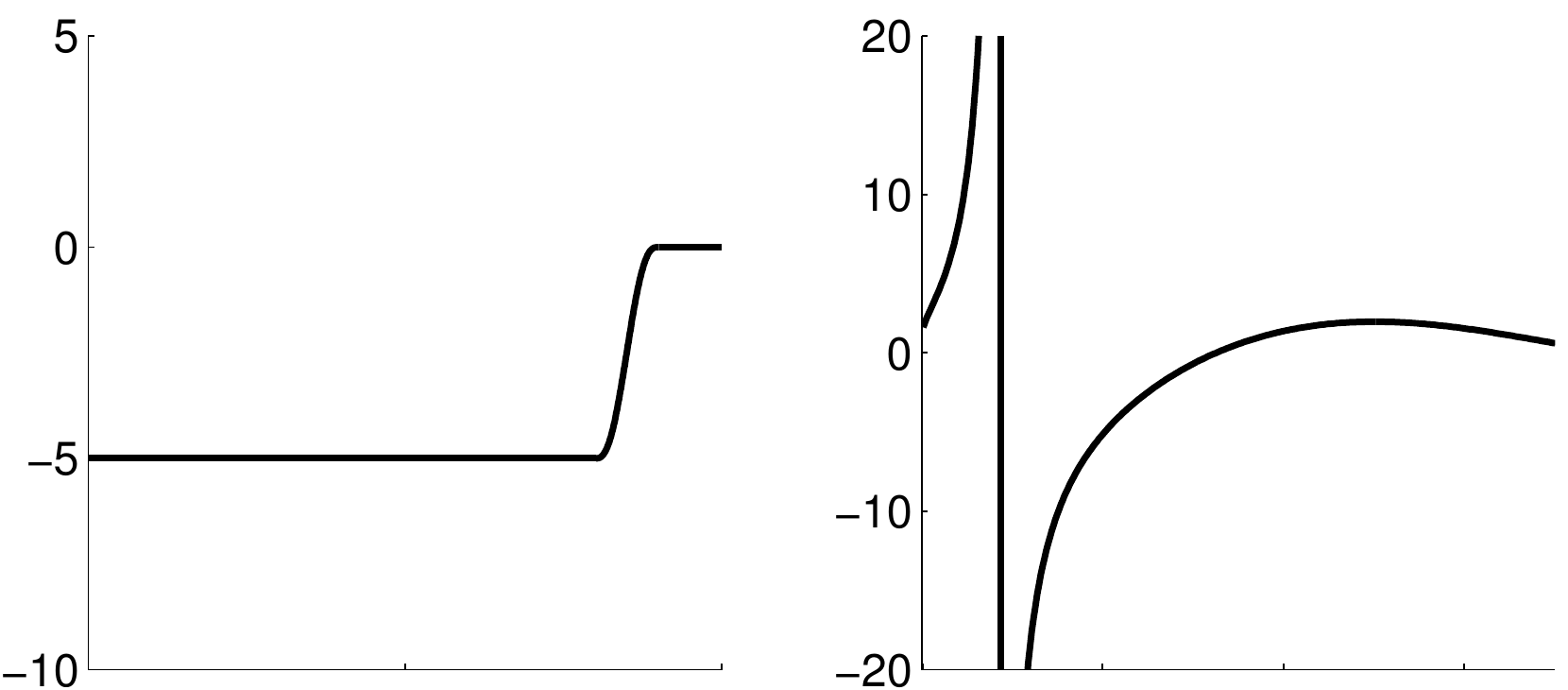}}
\put(30,250){\includegraphics[width=9cm]{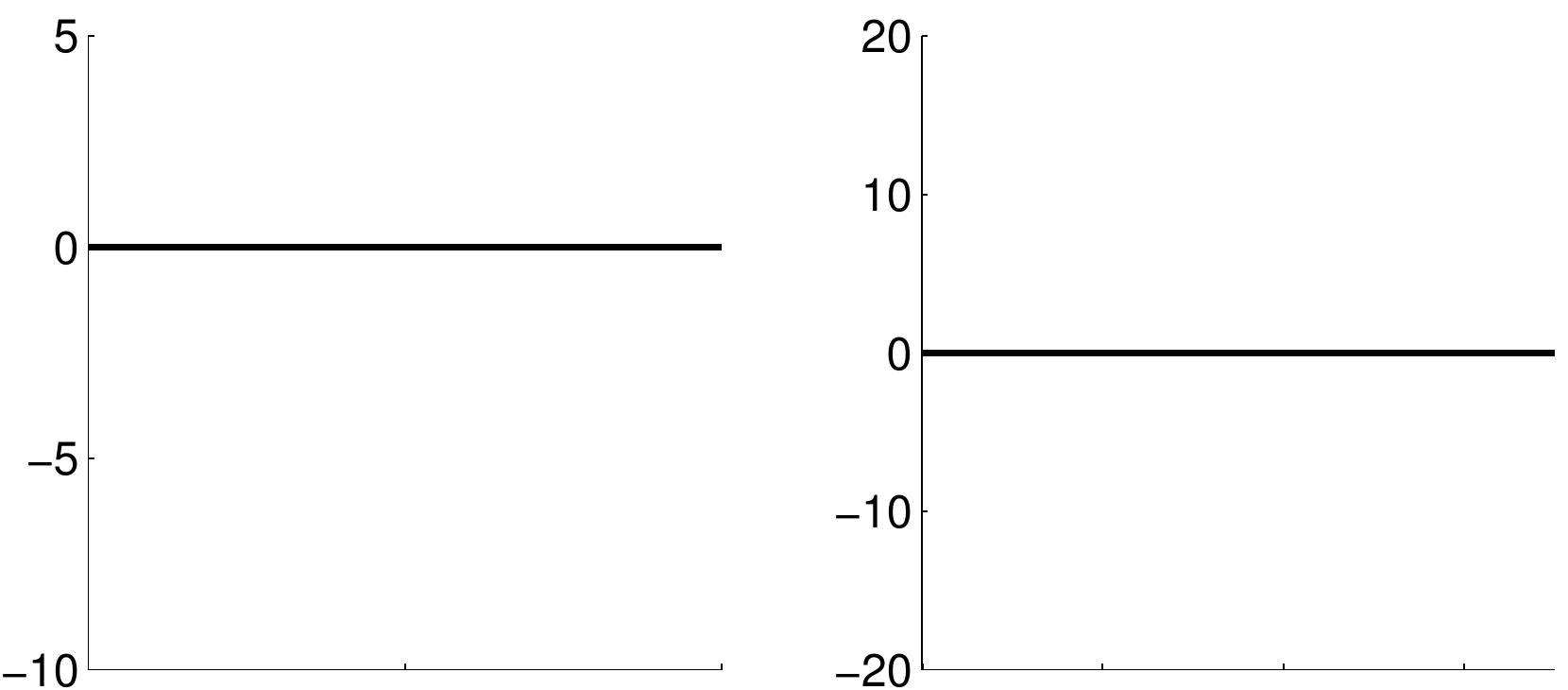}}
\put(30,370){\includegraphics[width=9cm]{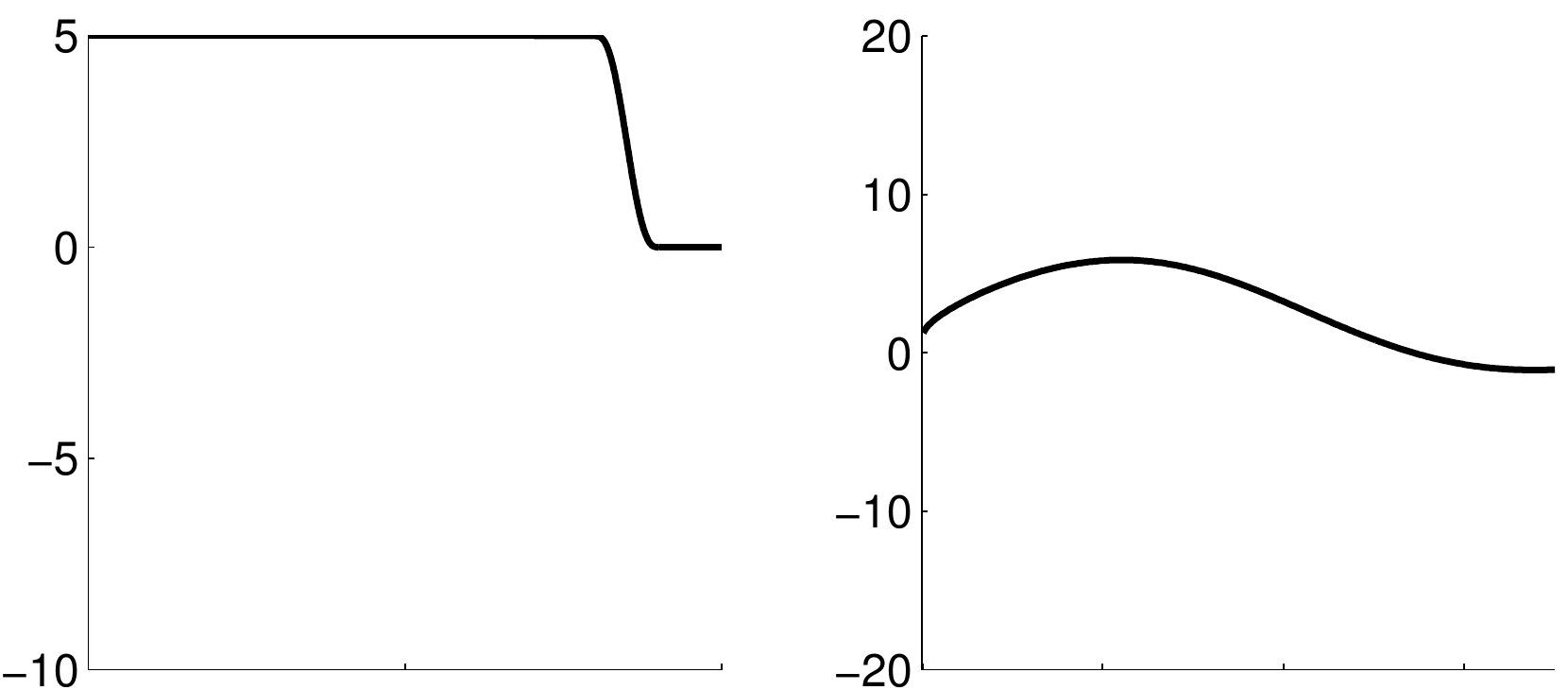}}
\put(40,490){Profile of potential}
\put(170,490){Profile of scattering transform}
\put(-10,430){$\lambda=5$}
\put(-10,310){$\lambda=0$}
\put(-10,190){$\lambda=-5$}
\put(-10,70){$\lambda=-10$}
\end{picture}
\caption{\label{fig:exceptionalcircles}Left column: profiles of
  rotationally symmetric potentials $q_\lambda(|z|)$ resulting from
  different values of $\lambda$. Right column: profiles of
  corresponding scattering transforms $t_\lambda(|k|)$. Note that
  negative values of $\lambda$ lead to exceptional circles where the
  scattering transform is singular. See also Figure \ref{fig:scat1A}
  which shows scattering transform profiles corresponding to more
  values of $\lambda$.}
\end{figure}
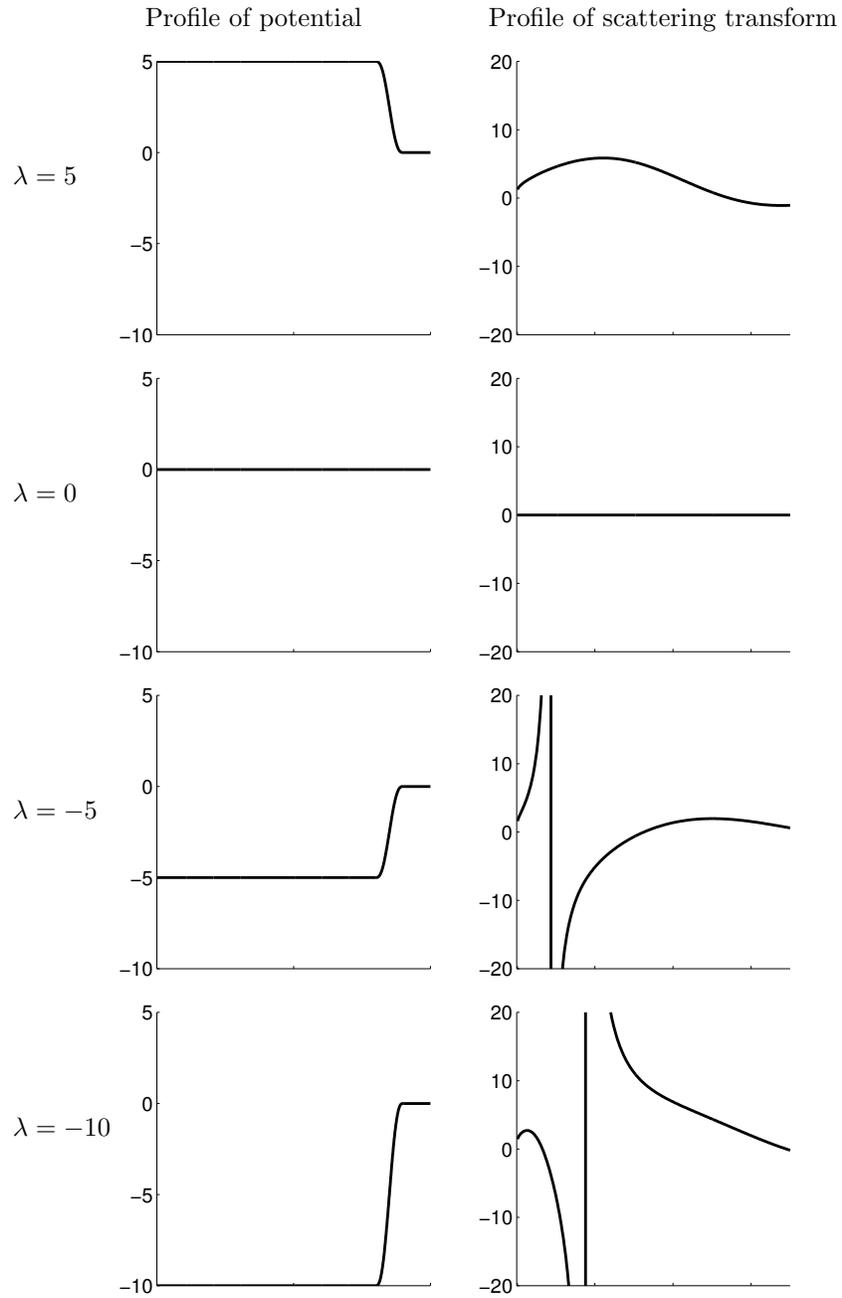

\begin{figure}
\begin{picture}(400,325)
 \put(0,20){\includegraphics[width=14cm]{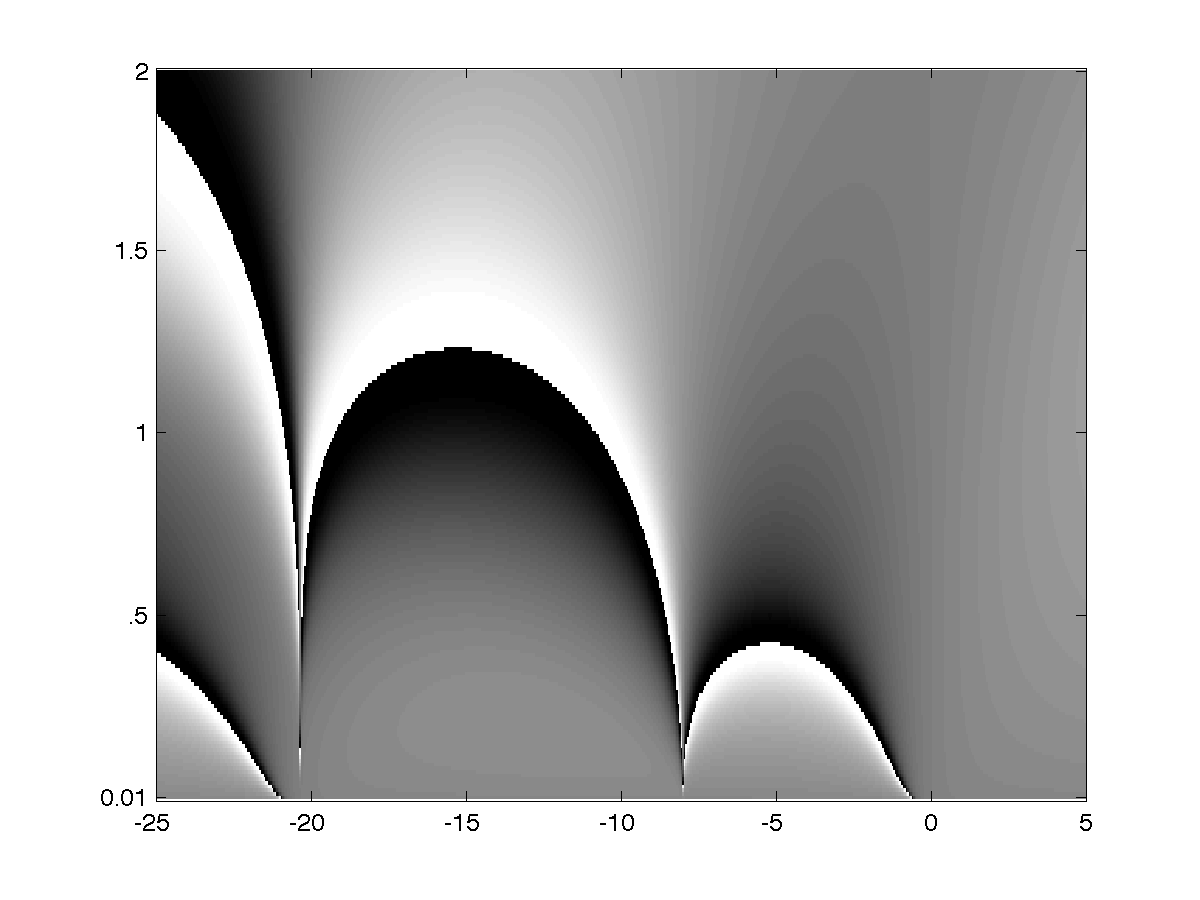}}
\put(200,0){\Large $\lambda$}
\put(-15,200){\Large $|k|$}
\end{picture}
\caption{\label{fig:scat1A}Profiles of rotationally symmetric
  scattering transforms. Lighter shade indicates larger values and
  darker shade smaller values. The horizontal axis is the parameter
  $\lambda$ in the definition $q_\lambda(z) = \lambda w(z)$ of the
  potential. The vertical axis is $|k|$. There are curves along which
  a singular jump ``from $-\infty$ to $+\infty$'' appears, indicated
  by an abrupt change from black to white. The $k$ values at those
  curves are exceptional points. Note that the profiles corresponding
  to $\lambda=-10,-5,0,5$ are shown also in Figure
  \ref{fig:exceptionalcircles}.}
\end{figure}

\clearpage

\section{Open Problems}
\label{sec:open}

\subsection{Applications of the NV Equation}

The stationary NV equation \cite{Ferapontov} and the modified NV
equation \cite{Taimanov} have applications in differential geometry.
Although the NV equation \eqref{eq:NV} is not known to be a
mathematical model for any physical dynamical system, there has been
some research in this direction for the dispersionless Novikov-Veselov
(dNV) equation  
\begin{eqnarray}\label{dNV}
q_\xi &=& (u q)_z + (\overline{u}q)_{\bar{z}} \\
u_{\bar{z}} &=& -3 q_z. \nonumber
\end{eqnarray}
 Equation \eqref{dNV} was derived in \cite{KonoMoro2004} as the
 geometrical optics limit of Maxwell's equations in an anisotropic
 medium.  The model governs the propagation of monochromatic
 electromagnetic (EM) waves of high frequency $\omega$.  In
 particular, they consider nonlinear media with Cole-Cole dependence
 \cite{ColeCole} of the dielectric function and magnetic permeability
 on the frequency.  Assuming slow variation of all quantities along
 the $z$ axis, it is shown that Maxwell's equations reduce to
 \eqref{dNV} where $q=n^2$, the refractive index, and $\xi$ is a
 ``slow'' variable defined by $z=\omega^{\nu}\xi$.  The phase of the
 electric field is governed by \cite{KonoMoro2004}
\begin{eqnarray} \label{S_phase}
S_x^2 + S_y^2 &=& n^2(x,y,\xi) \\
S_{\xi} &=& \phi(x,y,\xi;S_x,S_y) \nonumber
\end{eqnarray}
for a real-valued function $\phi$, and the dNV heirarchy characterizes
both the phase and refractive index.

In \cite{BogKonoMoro2006} hydrodynamic-type reductions of the dNV
equation are presented, but the physical interpretation of these
reductions are left for future work.  Interesting open problems are
whether the inclusion of the dispersion terms in the NV equation
models EM waves in a manner related to those derived for the dNV
equation, and whether either NV or dNV serves as a physical model for
any kind of hydrodynamic phenomenon.

\subsection{Exceptional Sets and Large-Time Behavior for the NV Equation}

Recall that the \emph{exceptional set} of a potential $q_0$ is the set of all those $k \in \mathbb{C}$ for which there is not a unique NCGO solution $\mu(z,k)$. Here we discuss the relationship between: (1) the ``size'' of the exceptional set,'' (2) whether the potential $q_0$ is subcritical, critical, or supercritical (recall this trichotomy, Definition \ref{def:tri} from the introduction) and (3) whether the NV equation with initial data $q_0$ has a global solution. In these remarks, we will usually restriction attention to real-valued potentials $q_0$ belonging to the space $L^p_\rho(\mathbb{R}^2)$ for $\rho>1$ and $p\in(1,2)$ (see the remarks preceding Definition \ref{def:cond.Nachman}); one may think of potentials of order $\mathcal{O}\left(|z|^{-2-\varepsilon}\right)$ as $|z| \rightarrow \infty$. Recall that, in the inverse scattering literature, critical potentials are usually referred to as ``potentials of conductivity type.''

To date, the only rigorous results on the size of exceptional sets for the zero-energy NV equation are due to Nachman \cite{Nachman:1996} and Music \cite{Music:2013}.  As explained above, Nachman showed that a potential is of conductivity type (or, equivalently, a critical potential as defined in the introduction, Definition \ref{def:tri})  if and only if the exceptional set is empty and the scattering transform $\mathbf{t}(k)$ is $\mathcal{O}(|k|^\varepsilon)$ as $|k| \rightarrow 0$ for some $\varepsilon>0$. Music, extending Nachman's ideas and techniques, showed that a subcritical potential with sufficient decay at infinity has an empty exceptional set and characterized the singularity of the potential as $|k| \rightarrow 0$.  Perry \cite{Perry:2012} showed that, if $q_0$ is a sufficiently smooth critical potential, the NV equation with initial data $q_0$ has a solution global in time. There is strong evidence to suggest that a similar result can be proved for the NV equation with subcritical initial data, based on the work of Music \cite{Music:2013}.

Thus, it remains to understand the singularities of the scattering transform for supercritical potentials. Examples due to Grinevich and Novikov \cite{GN:2012}  and Music, Perry and Siltanen \cite{MPS:2013} show that supercritical potentials may have circles of exceptional points. It is not known whether supercritical potentials \emph{must} or \emph{may} have exceptional points, nor is it known how to extend the inverse scattering formalism to potentials with nonempty exceptional sets. The following result due to Brown, Music, and Perry \cite{BMP:2014} gives an initial constraint on the size of exceptional sets for particularly nice potentials.

\begin{theorem} \label{thm:BMP} \cite{BMP:2014} Suppose that $q$ is a real-valued measurable valued function with the property that $|q(z)| \leq C_1 \exp(-C_2|z|)$ for some constants $C_1$ and $C_2$. Then the exceptional set of $q$ consists at most of isolated points together with at most finitely many smooth curves with at most finitely many intersections.
\end{theorem}

To analyze the exceptional set, the authors define a renormalized determinant whose zero set is exactly the exceptional set. To describe it, 
let $T_k$ is the integral operator 
$$ T_k \psi = \dfrac{1}{4} g_k*(q\psi). $$
The differential equation for $\mu(z,k)$, the NCGO solution, may be rewritten
$\mu=1+T_k \mu$. Hence,  uniqueness of solutions is equivalent to invertibility of $(I-T_k)$, and the exceptional set is exactly the set of points $k$ for which $(I-T_k)$ fails to be invertible. It can be shown that $T_k$ is a compact linear operator from $L^p$ to itself for any $p>2$, and that $T_k$ belongs to the so-called Mikhlin-Itskovich algebra of integral operators on $L^p$.  For this reason we can apply the theory renormalized determinants due to Gohberg, Goldberg, and Krupnik \cite{GGK:2000}
and define
$$ \Delta(k) = {\det}_2(I-T_k) $$
where the determinant ${\det}_2$ is the renormalized determinant. Brown, Music, and Perry show that this determinant is a real-analytic function of $k$ for exponentially decaying potentials. It now follows from the Weierstrass preparation theorem that the zero set of $\Delta(k)$ is locally the zero set of a polynomial. Since the exceptional set is known to be closed and bounded, one can completely analyze the behavior of $\Delta(k)$ near the exceptional set using finitely many such local representations. It can be shown that $\Delta(k)$ is also real-valued, from which it follows that the zero set has the claimed form. 

Theorem \ref{thm:BMP} opens up several areas for further investigation.

First, it would be of considerable interest to determine what additional data is needed to reconstruct a potential from $\mathbf{t}(k)$ when $\mathbf{t}(k)$ has point or line singularities. 

Second, it would be very interesting to know whether singularities are always present for supercritical potentials, or whether, on the other hand, singularities are generically absent.

Third, our understanding of the NV equation and its dynamics would be greatly improved by connecting `spectral' properties of the scattering transform (i.e., the nature of its singularities) to long-term behavior of solutions. The form of the time evolution for $\mathbf{t}(k)$ suggests that the `trichotomy' of subcritical, critical, and supercritical potentials is invariant under the NV flow. It is known that critical initial data give rise to global solutions (see \cite{Perry:2013}), and there is strong evidence that the same is true of subcritical initial data. On the other hand, numerical experiments such as the ring soliton, Example 4.3,  and analytical solutions such as those produced by Taimanov and Tsarev \cite{TT:2008a,TT:2008b,TT:2010a,TT:2010b}  strongly suggest that supercritical initial data lead to solutions of NV that blow up in finite time. It would be very interesting to obtain a rigorous proof that this is the case, and to analyze the nature of the blow-ups by inverse scattering methods.

\appendix{}

\section{Some Useful Analysis}
\label{sec:analysis}

In the direct scattering problem at zero energy, Faddeev's Green's function plays a critical role in elucidating properties of the CGO solutions that define the scattering transform. Recall that the normalized CGO solutions solve the equation 
$$ \overline{\partial}\left(\partial + ik \right) \mu=(1/4)q \mu$$ 
and that Faddeev's Green's function is Green's function for the operator $-4\overline{\partial}\left( \partial + ik \right)$. On the other hand, the solid Cauchy transform is an inverse for the $\overline{\partial}$ operator with range in $L^p$ functions for $p>2$, and hence is a fundamental tool for solving the $\overline{\partial}$ problem that defines the inverse scattering transform. Finally, the Beurling transform is an integral operator which gives a meaning to the operator $\overline{\partial}^{-1} \partial$ that occurs in the definition of the nonlinearity in the NV equation. Here we collect some useful properties of Faddeev's Green's function and the Cauchy and Beurlng transforms, and recall some essential estimates.

\subsection{Faddeev's Green's Function at Zero Energy}
\label{sec:Faddeev}

We recall some key facts about Faddeev's Green's function $g_{k}$. We refer the
reader to Siltanen's thesis \cite{Siltanen:1999} for details and references to
the literature.

Recall that $g_{k}$ is defined by the formula%
\[
g_{k}(z)=\frac{1}{\left(  2\pi\right)  ^{2}}\int e^{i\xi\cdot x}\frac
{1}{\overline{\xi}(\xi+k)}~dm(\xi)
\]
where $z=x_{1}+ix_{2}$, $\xi\cdot x=\xi_{1}x_{1}+\xi_{2}x_{2}$, $\xi=\xi
_{1}+i\xi_{2}$, and $\overline{\xi}=\xi_{1}-i\xi_{2}$. By the Hausdorff-Young
inequality, $g_{k}\in L^{p}$ for any $p>2$. In fact, the estimate%
\[
\left\Vert g_{k}\right\Vert _{p}\leq C_{p}\left\vert k\right\vert ^{-2/p}%
\]
(see Siltanen \cite{Siltanen:1999}, Theorem 3.10) holds for any $k\neq0$. It
is important to note that%
\[
g_{k}(z)=g_{1}(kz).
\]
The following large-$x$ asymptotic expansion of $g_{1}\left(  x\right)  $ is
proved in \cite{Siltanen:1999}, Theorem 3.11.

\begin{lemma}
Let $z=x_{1}+ix_{2}$ with $z\neq0$ and $x_{1}>0$. For any integer $N\geq0,$%
\begin{align}
g_{1}(z)  &  =-\frac{1}{4\pi}\sum_{j=0}^{N}\left[  \frac{j!}{\left(
iz\right)  ^{j+1}}-e^{-2ix_{1}}\frac{j!}{\left(  -i\overline{z}\right)
^{j+1}}\right] \label{eq:g1.exp}\\
&  +E_{N}(z)\nonumber
\end{align}
where%
\begin{equation}
\left\vert E_{N}(z)\right\vert \leq\frac{(N+1)!2^{\left(  N+1\right)  /2}}%
{\pi\left\vert z\right\vert ^{N+2}}. \label{eq:g1.err}%
\end{equation}
Since $g(-x_{1}+ix_{2})=\overline{g_{1}(x_{1}+ix_{2})}$, similar formulas hold
for $x_{1}<0$.
\end{lemma}

\begin{remark}
Since the error estimate (\ref{eq:g1.err}) does not depend on the condition
$x_{1}>0$, and since $g_{1}(z)$ is continuous, we can conclude that the
expansion (\ref{eq:g1.exp}) remains valid for $z\neq0$ and $\operatorname{Re}%
(z)=0$.
\end{remark}

Now consider $g_{k}(z)=g_{1}(kz)$. Since $\operatorname{Re}\left(  kz\right)
=\frac{1}{2}\left(  kz+\overline{k}\overline{z}\right)  $, we immediately obtain:

\begin{lemma}
Let $z=x_{1}+ix_{2}$ and $k\in\mathbb{C}$. For any integer $N\geq0$, the
expansion%
\begin{align}
g_{k}(z)  &  =-\frac{1}{4\pi}\sum_{j=0}^{N}\left[  \frac{j!}{\left(
ikz\right)  ^{j+1}}-e^{-i\left(  kz+\overline{k}\overline{z}\right)  }%
\frac{j!}{\left(  -i\overline{k}\overline{z}\right)  ^{j+1}}\right]
\label{eq:gk.exp}\\
&  +E_{N}(kz)\nonumber
\end{align}
holds, where
\[
\left\vert E_{N}(kz)\right\vert \leq C_{N}\left\vert kz\right\vert ^{-(N+2)}.
\]
\end{lemma}

The small-$k$ behavior of Fadeev's Green's function also plays an important role in the analysis of NCGO solutions. Let
$$G_0(z)=-\frac{1}{2\pi} \log|z|$$
be the logarithmic potential, and let
$$\ell(k)=\frac{1}{2\pi}\left(\log|k| + \gamma\right)$$
where $\gamma$ is the Euler constant. Finally let
$$\tilde{g}_k(z) = g_k(z)+\ell(z).$$
Nachman \cite[Lemma 3.4]{Nachman:1996} proves:
\begin{lemma}
For any $\varepsilon \in (0,1)$, all $k \in \mathbb{C}$ with $0<|k|<1/2$, and a positive constant $C_\varepsilon$,
the estimate
$$\left| \tilde{g}_k(z)-G_0(z) \right| 
\leq C_\varepsilon \left| k \right|^\varepsilon \langle z \rangle^\varepsilon
$$
holds.
\end{lemma}

\subsection{The Cauchy Transform and the Beurling Operator}
\label{sec:Cauchy}

Following \cite{AIM:2009}, chapter 4, we study the Cauchy transform and the Beurling
operator through the logarithmic potential associated with Poisson's equation in two dimensions.
For $\varphi\in\mathcal{C}_{0}^{\infty}(\mathbb{R}^{2})$, we may define the
logarithmic potential
\[
(L\varphi)(z)=\frac{2}{\pi}\int\log\left\vert z-z^{\prime}\right\vert
~\varphi(z^{\prime})~dm(z^{\prime})
\]
which has the property%
\[
\partial\overline{\partial}\left(  L\varphi\right)  =\varphi.
\]
Associated to $L$ are the \emph{Cauchy transform,}%
\[
\left(  P\varphi\right)  (z)=\frac{\partial}{\partial z}\left(  L\varphi
\right)  (z),
\]
the transform%
\[
\left(  \overline{P}\varphi\right)  (z)=\frac{\partial}{\partial\overline{z}%
}(L\varphi)(z),
\]
and the \emph{Beurling transform}%
\[
\left(  \mathcal{S}\varphi\right)  (z)=\frac{\partial^{2}}{\partial z^{2}%
}(L\varphi)(z).
\]
From these definitions and (\ref{eq:L}), it is easy to see that
\begin{equation}
P\circ\frac{\partial}{\partial\overline{z}}=\frac{\partial}{\partial
\overline{z}}\circ P=I \label{eq:P-0}%
\end{equation}
where $I$ is the identity on $\mathcal{C}_{0}^{\infty}(\mathbb{R}^{2})$, and
\begin{equation}
\mathcal{S}\left(  \frac{\partial\varphi}{\partial\overline{z}}\right)
=\frac{\partial\varphi}{\partial z}. \label{eq:B-0}%
\end{equation}
We have%
\begin{align*}
\left(  P\varphi\right)  (z)  &  =\frac{1}{\pi}\int\frac{1}{z-z^{\prime}%
}\varphi\left(  z^{\prime}\right)  ~dm(z^{\prime}),\\
\left(  \overline{P}\varphi\right)  (z)  &  =\frac{1}{\pi}\int\frac
{1}{\overline{z-z^{\prime}}}\varphi(z)~dm(z^{\prime}),
\end{align*}
and
\begin{equation}
(\mathcal{S}\varphi)(z)=-\frac{1}{\pi}\lim_{\varepsilon\downarrow0}\left(
\int_{\left\vert z-z^{\prime}\right\vert >\varepsilon}\frac{1}{(z-z^{\prime
})^{2}}\varphi(z^{\prime})~dm(z^{\prime})\right)  . \label{eq:B-1}%
\end{equation}

\subsubsection{Estimates on the Cauchy Transform}.
The following estimates on $P$ extend the Cauchy transform to $L^{p}$ spaces
and are standard consequences of the Hardy-Littlewood-Sobolev and H\"{o}lder
inequalities (see Vekua \cite{Vekua:1962} or \cite{AIM:2009}, \S 4.3). They 
are used to prove existence and uniqueness of solutions to the $\overline{\partial}$ problem that defines the inverse problem. 

\begin{lemma}
\label{lemma:P1}$~$(i) For any $p\in\left(  2,\infty\right)  $ and $f\in
L^{2p/(p+2)}(\mathbb{R}^{2})$,%
\begin{equation}
\left\Vert Pf\right\Vert _{p}\leq C_{p}\left\Vert f\right\Vert _{2p/(p+2)}.
\label{eq:P-1}%
\end{equation}
\newline(ii) For any $p,q$ with $1<q<2<p<\infty$ and any $f\in L^{p}\left(
\mathbb{R}^{2}\right)  \cap L^{q}\left(  \mathbb{R}^{2}\right)  $, the
estimate%
\begin{equation}
\left\Vert Pf\right\Vert _{\infty}\leq C_{p,q}\left\Vert f\right\Vert
_{L^{p}\cap L^{q}} \label{eq:P-2}%
\end{equation}
holds. Moreover, $P$ is H\"{o}lder continuous of order $\left(  p-2\right)
/p$ with%
\begin{equation}
\left\vert \left(  Pf\right)  (z)-\left(  Pf\right)  (w)\right\vert \leq
C_{p}\left\vert z-w\right\vert ^{(p-2)/p}\left\Vert f\right\Vert _{p}.
\label{eq:P-3}%
\end{equation}
\newline(iii) If $v\in L^{s}(\mathbb{R}^{2})$ and $q>2$ with $q^{-1}%
+1/2=p^{-1}+s^{-1}$, then for any $f\in L^{p}(\mathbb{R}^{2})$,
\begin{equation}
\left\Vert P\left(  vf\right)  \right\Vert _{q}\leq C_{p,q}\left\Vert
v\right\Vert _{s}\left\Vert f\right\Vert _{p}. \label{eq:P-4}%
\end{equation}

\end{lemma}

\begin{remark}
\label{rem:P1}Since $\mathcal{C}_{0}^{\infty}(\mathbb{R}^{2})$ is dense in
$L^{p}\cap L^{q}$ and $(Pf)(z)=\mathcal{O}\left(  z^{-1}\right)  $ as
$\left\vert z\right\vert \rightarrow\infty$ for any $f\in\mathcal{C}%
_{0}^{\infty}(\mathbb{R}^{2})$, it follows from (ii) that if $f\in L^{p}\cap
L^{q}$ for $1<p<2<q<\infty$, then $Pf\in C_{0}(\mathbb{R}^{2})$, the
continuous functions that vanish at infinity.
\end{remark}

\begin{remark}
\label{rem:P2} Note that with $s=q$ in (\ref{eq:P-4}) we have%
\[
\left\Vert P\left(  vf\right)  \right\Vert _{p}\leq C_{p}\left\Vert
v\right\Vert _{2}\left\Vert f\right\Vert _{p}.
\]

\end{remark}

For any $f\in L^{2p/(p+2)}(\mathbb{R}^{2})$, Lemma \ref{lemma:P1} together
with (\ref{eq:P-0}) imply that $u=Pf$ solves $\overline{\partial}u=f$ in
distribution sense. Suppose, on the other hand, that $u\in L^{p}%
(\mathbb{R}^{2})$ for some $p\in\lbrack1,\infty)$ and $\overline{\partial}u=0$
in distribution sense. It follows that $\partial\overline{\partial}u=0$ in
distribution sense, so that $u\in\mathcal{C}^{\infty}$ by Weyl's lemma. Thus,
$u$ is actually holomorphic, so $u$ vanishes identically by Liouville's
Theorem. From this fact and (\ref{eq:P-0}), we deduce:

\begin{lemma}
\label{lemma:P2}Suppose that $p\in(2,\infty)$, that $u\in L^{p}(\mathbb{R}%
^{2})$, that $f\in L^{2p/(p+2)}(\mathbb{R}^{2})$, and that $\overline
{\partial}u=f$ in distribution sense. Then $u=Pf$. Conversely, if $f\in
L^{2p/(p+2)}(\mathbb{R}^{2})$ and $u=Pf$, then $\overline{\partial}u=f$ in
distribution sense.
\end{lemma}

The following expansion for solutions of $\overline
{\partial}u=f$ when $f$ is rapidly decaying gives rise to the large-$k$ asymptotic expansion
for $\mu(z,k)$.

\begin{lemma}
\label{lemma:P3}Suppose that $p\in\left(  2,\infty\right)  $, that $u\in
L^{p}(\mathbb{R}^{2})$, that $f\in L^{2p/(p+2),N}(\mathbb{R}^{2})$, and that
$\overline{\partial}u=f$. Then%
\[
z^{N}\left[  u(z)-\sum_{j=0}^{N-1}\frac{1}{z^{j+1}}\int\zeta^{j}%
f(\zeta)~dm(\zeta)\right]  \in L^{p}(\mathbb{R}^{2}).
\]

\end{lemma}

\begin{proof}
An immediate consequence of the estimate (\ref{eq:P-1}), Lemma \ref{lemma:P2}
and the formula%
\[
\frac{1}{z-\zeta}=\frac{1}{z}\sum_{j=0}^{N-1}\left(  \frac{\zeta}{z}\right)
^{j}+\frac{1}{z^{N}}\frac{\zeta^{N}}{z-\zeta}.
\]

\end{proof}

\subsubsection{A Generalized Liouville's Theorem}

The following analogue of Liouville's Theorem plays a fundamental role in the theory of $\overline{\partial}$-problems such as those used to define the inverse scattering map for the NV equation. The version cited here is due to Brown and Uhlmann \cite{BU:1997}. A slightly more general result that plays an important role in the study of the NV equation with subcritical initial data may be found in \cite{Music:2013}.

\begin{theorem} \cite[Corollary 3.11]{BU:1997} Suppose that $p\in [1,\infty)$ and that $u \in L^p(\mathbb{R}^2) \cap L^2_{\mathrm loc}(\mathbb{R}^2)$ is a weak solution of 
$$ \overline{\partial} u = a u + b \overline{u} $$
with coefficients $a,b \in L^2(\mathbb{R}^2)$. Then $u=0$.
\end{theorem}

\begin{remark}
If $f\in\mathcal{S}\left(  \mathbb{R}^{2}\right)  $ and depends smoothly on
parameters, then the asymptotic expansion holds pointwise and is
differentiable in the parameters.
\end{remark}

\subsubsection{The Beurling Transform}

The principal value integral (\ref{eq:B-1}) identifies $\mathcal{S}$ as a
Calder\'on-Zygmund type integral operator. We have (see, for example,
\cite{AIM:2009}, \S 4.5.2):

\begin{lemma}
\label{lemma:S}Suppose that $p\in\left(  1,\infty\right)  $. The operator
$\mathcal{S}$ extends to a bounded operator from $L^{p}(\mathbb{R}^{2})$ to
itself, unitary if $p=2$. Moreover, if $\nabla\varphi$ belongs to
$L^{p}(\mathbb{R}^{2})$ for $p\in\left(  1,\infty\right)  $, then
$\mathcal{S}\left(  \partial\varphi\right)  =\overline{\partial}\varphi$.
\end{lemma}

\end{document}